\documentclass[11pt]{article}
\usepackage[T1]{fontenc}
\usepackage[utf8]{inputenc}
\usepackage{lmodern}
\usepackage{amsmath,amssymb,amsthm,mathtools,mathrsfs}
\usepackage{enumitem}
\usepackage{microtype}
\usepackage{needspace}
\usepackage[margin=1.1in]{geometry}
\usepackage[
  backend=bibtex,
  style=numeric,
  sorting=nyt,
  giveninits=true,
  maxbibnames=99,
  doi=true,
  url=true,
  isbn=true,
  eprint=true
]{biblatex}
\usepackage[colorlinks=true,linkcolor=blue,citecolor=blue,urlcolor=blue]{hyperref}
\usepackage[nameinlink,capitalise]{cleveref}
\hypersetup{
  pdftitle={Dolbeault--Hochster Theory of Polytopal LVM Manifolds},
  pdfauthor={Ludmil Katzarkov, Kyoung-Seog Lee, Ernesto Lupercio, Laurent Meersseman},
  pdfsubject={Dolbeault cohomology of polytopal LVM manifolds and Hochster decompositions},
  pdfkeywords={LVM manifold, Dolbeault cohomology, Hochster formula, moment-angle manifold, Frolicher spectral sequence, Stanley-Reisner ring, Cox cover}
}

\newcommand{\C}{\mathbb C}
\newcommand{\R}{\mathbb R}
\newcommand{\Z}{\mathbb Z}
\newcommand{\NN}{\mathbb N}
\newcommand{\PP}{\mathbb P}
\newcommand{\XL}{X_{\Lambda}}
\newcommand{\GammaL}{\Gamma_{\Lambda}}
\newcommand{\GammaLCdual}{\operatorname{Hom}_{\Z}(\GammaL,\C)}
\newcommand{\Kos}{\operatorname{Kos}}
\newcommand{\Pole}{\operatorname{Pole}}

\theoremstyle{plain}
\newtheorem{theorem}{Theorem}[section]
\newtheorem{proposition}[theorem]{Proposition}
\newtheorem{lemma}[theorem]{Lemma}
\newtheorem{corollary}[theorem]{Corollary}

\theoremstyle{definition}
\newtheorem{definition}[theorem]{Definition}
\newtheorem{example}[theorem]{Example}
\theoremstyle{remark}
\newtheorem{remark}[theorem]{Remark}

\title{Dolbeault--Hochster Theory of Polytopal LVM Manifolds}
\author{Ludmil Katzarkov, Kyoung-Seog Lee,\\
Ernesto Lupercio, and Laurent Meersseman}
\date{}

\begin{document}
\maketitle
\begin{abstract}
We compute the Dolbeault cohomology of minimally stable polytopal LVM
manifolds using a finite curvature complex. Let $B$ be the indispensable
weight block. A canonical exact sequence identifies the kernel of the
curvature map with $(\operatorname{coker}B)^\vee$ and its cokernel with
$(\ker B)^\vee$. At full rank, the Dolbeault groups are direct sums of
reduced cohomology groups of induced subcomplexes of the visible simplicial
sphere, tensored with the exterior algebra on the complex dual of the deck
lattice. One proof uses the curvature model and Stanley--Reisner Tor; a
second gives an additive comparison through completed Laurent expansions
and holomorphic descent.

Full curvature rank is equivalent to Fr\"olicher degeneration at $E_1$.
Every manifold in this class fails the $\partial\bar\partial$ lemma.
For polygons, the number of facets and the curvature rank determine the
entire Hodge diamond. We construct a proper holomorphic family over a disk
with fixed visible square whose curvature rank drops at the origin. The
total dimension of Dolbeault cohomology increases there by sixteen.

At full rank, holomorphic descent also computes tangent-sheaf cohomology,
including resonant contributions in positive \v{C}ech degree. We determine
the Kodaira--Spencer map of the normalized weight family and give a
cohomological criterion for semiuniversality with smooth base. For a
resonant example over the square, the global vector fields have a
polynomial basis of nine elements and $h^1(\Theta)=19$. The primary
obstruction map is nonzero. The additional \v{C}ech class integrates in a
holomorphic family.
\end{abstract}
\medskip
\noindent\textbf{2020 Mathematics Subject Classification.}
Primary 32Q55; Secondary 13F55, 55U10, 57S12.

\smallskip
\noindent\textbf{Keywords.}
Polytopal LVM manifold, Dolbeault cohomology, Hochster formula,
moment-angle manifold, Fr\"olicher spectral sequence, Stanley--Reisner ring,
Cox cover.
\clearpage
\tableofcontents
\clearpage
\section{Introduction}
\label{sec:introduction}

For a compact K\"ahler manifold, the Hodge decomposition and Hodge symmetry give
\[
 b_k=\sum_{p+q=k}h^{p,q},\qquad h^{p,q}=h^{q,p}.
\]
In degree one, these imply $b_1=2h^{1,0}=2h^{0,1}$. For the LVM manifolds
studied here, the Hodge numbers can vary while the underlying smooth
manifold remains fixed. Their torus actions have convex polytopes as orbit
spaces, and these polytopes control the topology in the minimally stable
class considered below \cite{BosioMeersseman2006}.

How much of the Dolbeault cohomology can be recovered from the associated
polytope? We define a curvature map whose full-rank condition characterizes
a formula for every Hodge number in terms of the cohomology of induced
subcomplexes. A drop in curvature rank produces holomorphic one-forms.
We identify the curvature map from the defining weights and compute its
kernel and cokernel.

An LVM manifold is a compact quotient of an open set in projective
space by a holomorphic $\C^m$-action.  The action is diagonal, with
weights $\Lambda_1,\ldots,\Lambda_n\in\C^m$.  Convexity determines the
open set: a point belongs to it when the weights indexed by its
nonzero coordinates contain the origin in their convex hull.
Admissibility makes the action free and proper.  A label is
\emph{indispensable} when its coordinate is nonzero throughout this
open set.

We assume that there are exactly $m+1$ indispensable labels.  This is
the minimally stable case.  The remaining $N=n-m-1$ labels are the
facets of a simple polytope of dimension $d=N-m$.  We call this the
visible polytope and write $K$ for the boundary complex of its
simplicial dual.  The manifold has complex dimension $N$ and is
homeomorphic to $\mathcal Z_K\times T^m$.  Thus $K$ and $m$ determine
its cohomology as a topological space.  The Dolbeault groups also
remember the complex weights.

Choose an indispensable coordinate as projective denominator.  The
quotient action then has a visible block $A$ and a square indispensable
block $B$:
\[
 \Psi_\Lambda=\begin{pmatrix}A\\B\end{pmatrix},
 \qquad A\in\operatorname{Mat}_{N\times m}(\C),
 \qquad B\in\operatorname{Mat}_{m\times m}(\C).
\]
The rank of $B$ is independent of this choice.  Put
\[
 \rho=\operatorname{rank}_{\C}B,
 \qquad \delta=m-\rho.
\]
Weak hyperbolicity implies
$\lceil m/2\rceil\leq\rho\leq m$, and every rank in this range occurs
(Theorems~\ref{thm:real-indispensable-frame}
and~\ref{thm:exact-realizable-rank-range}).  Thus admissible configurations realize singular as well as invertible blocks.

\subsection{Curvature and the finite model}

An LVM manifold carries a canonical holomorphic foliation, whose basic Dolbeault ring is
\[
 R_\Lambda=\C[K]/J_\Lambda,
 \qquad \deg v_j=(1,1),
\]
where $J_\Lambda$ consists of the linear relations of the projected
fan.  Its degree-$(1,1)$ part has dimension $m$.  There are also
$m$-dimensional spaces $W_\Lambda^{1,0}$ and $W_\Lambda^{0,1}$ of
leafwise generators.  The transverse Dolbeault model gives a weak
equivalence of differential bigraded algebras
\[
 (\Omega^{*,*}(\XL),\bar\partial)
 \simeq
 \left(R_\Lambda\otimes
 \Lambda(W_\Lambda^{1,0}\oplus W_\Lambda^{0,1}),\bar d\right),
\]
with $\bar dR_\Lambda=\bar dW_\Lambda^{0,1}=0$ and
$\bar d|_{W_\Lambda^{1,0}}=c_\Lambda$.
Our first result identifies this transgression from the quotient
weights and places it in the canonical exact sequence
\[
 0\longrightarrow(\operatorname{coker}B)^\vee
 \longrightarrow W_\Lambda^{1,0}
 \overset{c_\Lambda}{\longrightarrow}R_\Lambda^{1,1}
 \longrightarrow(\ker B)^\vee\longrightarrow0.
\]
The construction is natural under a change of indispensable affine
chart (Theorem~\ref{thm:curvature-exact-sequence} and
Proposition~\ref{prop:curvature-chart-covariance}).  Consequently
$\operatorname{rank}c_\Lambda=\rho$.

This finite model computes the Dolbeault groups at every realizable
rank.  For example,
\[
 h^{p,0}(\XL)=\binom\delta p,
 \qquad h^{0,q}(\XL)=\binom mq,
 \qquad h^{1,q}(\XL)=\delta\binom{m+1}q.
\]
The closed antiholomorphic generators give a factor $(1+v)^m$ in the
Hodge polynomial at every rank.  The curvature kernel gives
$(1+u)^\delta$.  A complement to this kernel gives a tensor
decomposition of the finite complex.  The kernel and the reduced
quotient complex are intrinsic; the tensor decomposition depends on
the complement (Proposition~\ref{prop:defect-exterior-factor}).

For polygons, the multiplication pairing on $R_\Lambda^1$ determines
the ranks of all reduced Koszul differentials.  The entire Hodge
diamond therefore depends only on $m$ and $\rho$
(Theorem~\ref{thm:polygonal-rigidity}).  For each of the two combinatorial types, the triangular prism and the cube, the Hodge diamond is also determined by the curvature rank, independently of the normal fan
(Section~\ref{sec:three-dimensional-rigidity}).  In higher dimension, Hodge numbers can vary at fixed rank. The Panov--Ustinovsky comparison over
$\Delta^1\times\Delta^1\times\Delta^2\times\Delta^2$
\cite[Example~5.13]{PanovUstinovsky2012} has two minimally stable
realizations with the same visible complex and curvature rank, and
different $h^{2,1}$
(Proposition~\ref{prop:known-six-dimensional-fixed-rank}).

\subsection{The full-rank formula}

Suppose that $B$ is invertible.  Setting the indispensable coordinates
equal to one gives a Cox covering with deck lattice
$\GammaL\cong\Z^m$.  We use its complex dual $\GammaLCdual$.
For $I\subseteq[N]$, let $K_I$ be the induced subcomplex on $I$.
Then
\begin{equation}
\label{eq:intro-main-formula}
 H^q(\XL,\Omega_{\XL}^p)
 \cong
 \bigoplus_{\substack{I\subseteq[N],\ |I|=p\\
                      a+s=q,\ 0\leq a\leq d,\ 0\leq s\leq m}}
 \widetilde H^{a-1}(K_I;\C)
 \otimes\Lambda^s\GammaLCdual.
\end{equation}
We use $\widetilde H^{-1}(K_\varnothing;\C)=\C$.
Theorem~\ref{thm:Dolbeault-Hochster-first} proves this formula by
curvature; Corollary~\ref{thm:Dolbeault-Hochster-second-proof} proves
it by holomorphic descent.

The formula~\eqref{eq:intro-main-formula} holds precisely at full curvature rank. Its right-hand side gives $h^{1,0}=0$, whereas the finite
model gives $h^{1,0}=\delta$.  The same condition characterizes
Fr\"olicher degeneration at $E_1$
(Theorem~\ref{thm:exact-full-rank-criterion}).  Indeed, at full rank
the formula agrees in total degree with the topological Hochster
decomposition.  At positive defect,
\[
 h^{1,0}+h^{0,1}=m+\delta>b_1(\XL)=m.
\]
The rank bound gives $h^{1,0}<h^{0,1}$ for every manifold in this class, so the $\partial\bar\partial$ lemma fails (Corollary~\ref{cor:all-rank-ddbar-failure}).

A square realizes the rank change in a proper holomorphic family.
With $m=2$, we give admissible weights on $|t|<1/20$ for which
$\det B(t)=(-1+i)t$.  The visible square stays fixed; the central
fibre has rank one and every other fibre has rank two.  Their Hodge
polynomials satisfy
\[
 \operatorname{Dol}_{X_0}(u,v)-\operatorname{Dol}_{X_t}(u,v)
 =(1+v)^3(u+u^3v),\qquad t\neq0.
\]
Thus the total dimension of Dolbeault cohomology increases by sixteen at the origin.  The weights are admissible throughout the disk, and the action on the total
space is proper (Section~\ref{sec:holomorphic-square-family}).

\subsection{The two proofs and the tangent sheaf}

The curvature proof uses the Cohen--Macaulay property of $\C[K]$.
At full rank, lifts of a basis of $R_\Lambda^1$, together with the
relations in $J_\Lambda$, form a basis of the coordinate linear
forms.  The resulting Koszul comparison identifies curvature
cohomology with
\[
 \operatorname{Tor}^{\C[v_1,\ldots,v_N]}(\C[K],\C).
\]
Its squarefree multidegrees are the reduced cochain complexes of the
$K_I$.  This comparison preserves multiplication and internal degree.
The coordinate multigrading on Tor induces a support decomposition on
curvature cohomology through the chosen comparison. This decomposition
depends on the comparison
(Remark~\ref{rem:first-bridge-choice-control}).

The analytic proof starts from the discrete quotient presentation of a full-rank LVM manifold (Theorem~\ref{thm:full-rank-Cox-quotient}) and holomorphic forms on its Cox cover.
Their Laurent expansions give a completed complex, and a coinduced
group--Koszul resolution computes deck descent.  Euler operators
first reduce cohomology to finite Laurent support.  Character contraction then involves finitely many scalar inverses in each cohomology class.  A nonzero resonant exponent gives
a Gale height of both signs; the associated pole complex is
contractible.  The zero exponent remains, and its relative
\v{C}ech complex gives the induced subcomplexes.  The construction
provides an additive comparison
\[
 R\Gamma(\XL,\Omega_{\XL}^p)
 \simeq
 \bigoplus_{|I|=p}\widetilde C^{\bullet-1}(K_I;\C)
 \otimes\Lambda^\bullet\GammaLCdual
 \quad\text{in }D(\C)
\]
(Theorem~\ref{thm:Cox-Cech-Hochster-splitting}).  The construction also gives explicit \v{C}ech cocycles for each cohomology class.

The same analytic tools compute tangent-sheaf cohomology.
For $E_i=z_i\partial_{z_i}$, the sector $z^rE_i$
has pole set
\[
 T_i(r)=\{j:r_j+\delta_{ij}<0\}.
\]
The induced complex on this set can have nonzero reduced cohomology.
The tangent formula therefore retains every resonant \v{C}ech degree
(Theorem~\ref{thm:tangent-sector-formula}).  For the resonant full-rank square of Section~\ref{sec:tangent-square}, the space of global vector fields has a polynomial basis of nine elements, including three cubic fields, and
\[
 \sum_q h^q(X,\Theta_X)t^q=(9+t)(1+t)^2.
\]
The additional degree-one class comes from the reduced zeroth cohomology of a
shifted pole complex consisting of two opposite vertices.

The normalized weights $C=AB^{-1}$ define a proper holomorphic family.
We compute its Kodaira--Spencer map and identify its image with the
$Nm$-dimensional zero-exponent summand of $H^1(\Theta)$
(Theorem~\ref{thm:normalized-KS}).  If $h^1(\Theta)=Nm$, this family is
semiuniversal and the Kuranishi germ is smooth
(Theorem~\ref{thm:diagonal-semiuniversality}).  Nonresonance is a
sufficient condition.  For the resonant square of Section~\ref{sec:tangent-square}, the map has rank eight and cokernel dimension eleven.  The primary obstruction map is nonzero
(Theorem~\ref{thm:square-primary-obstruction}). The extra \v{C}ech class
integrates in the holomorphic family of
Proposition~\ref{prop:square-Cech-family}.

\subsection{Earlier work}

LVM manifolds originated in the work of L\'opez de Medrano--Verjovsky and Meersseman. L\'opez de Medrano--Verjovsky introduced the quotient construction for
a one-dimensional complex action \cite{LopezMedranoVerjovsky1997}, and Meersseman extended it to
arbitrary action dimension
\cite{Meersseman2000}. Bosio characterized the projective coordinate-subspace open sets on which a diagonal exponential action is proper and has compact quotient. The characterization uses the imbrication and substitute-existence conditions \cite[Theorem~1.4]{Bosio2001}.  The resulting
LVMB class strictly contains the LVM manifolds. Bosio and Meersseman related their smooth models to intersections of real quadrics and convex polytopes \cite{BosioMeersseman2006}. This work was inspired by the earlier work of L\'opez de Medrano \cite{LopezdeMedrano1988,LopezdeMedrano1989, LopezdeMedrano2025}.

Denham and Suciu studied the relation with moment-angle manifolds \cite{DenhamSuciu2007}. Bosio and Meersseman used this relation to compute the integral cohomology of LVM manifolds \cite{BosioMeersseman2006}. Complex moment-angle manifolds of Panov--Ustinovsky \cite{PanovUstinovsky2012} are precisely the subclass of LVMB manifolds whose datum has at least one indispensable point.
Tambour and Panov--Ustinovsky independently proved that these complex
structures arise on even-dimensional moment-angle manifolds associated
with star-shaped simplicial spheres, that is, simplicial spheres defining a fan
\cite{Tambour2012,PanovUstinovsky2012}. Battisti recast Bosio's conditions in terms of this fan together with its linear marking, the first marked-fan description of LVMB manifolds, and characterized the LVM case by polytopality of the projected fan \cite{Battisti2013}.

Ishida studied compact complex manifolds with maximal torus actions, a class containing the LVM and LVMB manifolds.  They are compact connected complex manifolds carrying an effective action of a compact real torus $G$ by holomorphic transformations such that $\dim G+\dim G_x=\dim M$ for some $x\in M$, where all dimensions are real.  Ishida classified these manifolds by fan and linear data and proved an equivalence of categories; LVMB manifolds occur in this class as building blocks \cite{Ishida2019}.

The canonical foliation and its transverse K\"ahler geometry appeared earlier.  Loeb--Nicolau established the one-dimensional case, and Meersseman proved the result for LVM manifolds \cite{LoebNicolau1999, Meersseman1998, Meersseman2000}. The transversely K\"ahler form is the restriction to the smooth projectivized quadric model of the Fubini--Study form on $\mathbf P^{n-1}$ \cite{MeerssemanVerjovsky2004,PanovUstinovskiyVerbitsky2016}.

For LVM manifolds satisfying the rationality condition (K), Meersseman
and Verjovsky identified the canonical foliation as a Seifert fibration
over a projective simplicial toric variety. They also showed that every
projective simplicial toric variety arises as the base of such a fibration
\cite{MeerssemanVerjovsky2004}. Cupit-Foutou and Zaffran extended this
description to LVMB manifolds, whose toric bases are complete and
simplicial but need not be projective. When the base is nonprojective,
the foliation is not transversely K\"ahler and the manifold is not LVM
\cite{CupitFoutouZaffran2007}. For arbitrary LVMB manifolds, Ishida proved
that the canonical foliation is transversely K\"ahler if and only if the
manifold is LVM \cite{IshidaTransverse2017}. The canonical foliation also
gives results on complex subvarieties in both classes
\cite{LoebNicolau1999,Meersseman2000,PanovUstinovskiyVerbitsky2016}.

Under condition (K), when the fibration is a holomorphic principal torus
bundle over a smooth base, Panov--Ustinovsky computed the Dolbeault
cohomology of the total space using the Borel spectral sequence and
H\"ofer's results \cite{Hofer1993}. Its differential is the Chern class map
of the holomorphic torus bundle. Their calculation also gives Fr\"olicher
degeneration at $E_2$
\cite[Theorem~5.4 and Corollary~5.8]{PanovUstinovsky2012}.

Ishida developed torus-invariant transverse K\"ahler foliations in the more general and intrinsic setting of manifolds with maximal torus actions \cite{IshidaTransverse2017}.  When such a transverse K\"ahler structure exists, Ishida--Kasuya associate differential graded and differential bigraded models to it \cite{IshidaKasuya2019}. Battaglia--Zaffran constructed foliated LVMB manifolds from simplicial
fans, including nonrational fans, and computed basic Betti numbers
for shellable fans \cite{BattagliaZaffran2015}. Building on \cite{Ishida2018}, Ishida--Krutowski--Panov introduced transverse equivalence and computed the basic cohomology ring of the canonical holomorphic foliation for complex manifolds with maximal torus action \cite{IshidaKrutowskiPanov2022}.
Krutowski--Panov also computed the basic Dolbeault ring in the maximal-torus
setting and gave a bigraded model for ordinary Dolbeault cohomology
of complex moment-angle manifolds
\cite[Theorems~5.1 and~6.1]{KrutowskiPanov2021}.
For regular LVMB torus bundles, Thiella expresses the characteristic
class in Gale-dual lattice data and obtains a bigraded model
\cite[Theorem~1 and Proposition~2.8]{Thiella2026}.

We use these transverse models to identify the curvature map at every
realizable rank.  The resulting exact sequence explains when the
classical face-ring Tor groups compute Dolbeault cohomology.
The full-rank calculation uses Hochster's multigraded formula and the Tor description of moment-angle cohomology
\cite{Hochster1977,Panov2008,BuchstaberPanov2015}.
The squarefree comparison fixes the Dolbeault bidegrees; the analytic
proof constructs the same summands directly from holomorphic descent.
Our earlier Hodge-polynomial calculations
\cite{KatzarkovLeeLupercioMeersseman2025} are compared with the
full-rank polygon formula in Theorem~\ref{thm:polygonal-example}.
Rank-deficient examples were already known: Faucard's Example~3 is
computed in Proposition~\ref{prop:Faucard-pentagon-Hodge-table}
\cite{Faucard2024}.

Meersseman also studied holomorphic vector fields and deformations,
constructing parameter spaces of dimension $m(n-m-1)$ and obtaining
universality under additional hypotheses \cite{Meersseman2000}.
Biswas--Dumitrescu--Meersseman studied fundamental vector fields with
nonempty zero locus and the LVMB tangent bundle
\cite[Theorem~4]{BiswasDumitrescuMeersseman2020}. Madera computes the Kuranishi space of some LVM $3$-folds \cite{Madera2025}.
Sections~\ref{sec:tangent-sheaf} and~\ref{sec:deformations} compute the
resonant contributions to tangent cohomology and their role in
semiuniversality and the square's primary obstruction.

\subsection{Organization and remaining questions}

Part~I constructs the curvature map and computes its defect, the
polygonal Hodge polynomials, and the square family.  Parts~II and~III
give the two proofs of \eqref{eq:intro-main-formula}.  Part~IV treats
duality, examples, tangent cohomology, deformations, and fixed-rank
comparisons.  The appendix gives the hypergeometric form of the
polygonal coefficients.

Further questions concern fixed-rank variation for other three-dimensional
visible polytopes and a holomorphic family joining the two configurations
in the comparison over a six-dimensional visible polytope. Computing the Fr\"olicher differentials at deficient curvature rank requires a model retaining both $\partial$ and $\bar\partial$.
Determining the full Kuranishi germ of the square requires the higher obstruction equations. A further problem is to compute tangent-sheaf cohomology at deficient curvature rank.

\subsection{Notation}
Let $n$ be the total number of LVM weights and $m$ the dimension of the
complex parameter space of the LVM action.  The indispensable and visible
label sets are
\[
 D=\{a_0,\ldots,a_{k-1}\},
 \qquad k=|D|,
\]
and $[N]=\{1,\ldots,N\}$, respectively, where $N=n-k$.

In the polytopal case, $P=P_\Lambda$ denotes the simple $d$-polytope with
$N$ facets.  Its dual simplicial polytope and boundary complex are
\[
 Q=P^\vee,
 \qquad
 K=\partial Q,
\]
with vertex set $[N]$.  For $I\subseteq[N]$, the symbol $K_I$ denotes the
induced full subcomplex of $K$ on $I$.  Reduced cohomology is augmented
throughout:
\[
 K_\varnothing=\{\varnothing\},
 \qquad
 \widetilde H^{-1}(K_\varnothing;\C)=\C,
\]
and
\[
 \widetilde b_j(K_I)=\dim_\C\widetilde H^j(K_I;\C).
\]

Minimal stability means $k=m+1$.  In this case,
\[
 m=N-d,
 \qquad
 n=2N-d+1,
 \qquad
 \dim_\C\XL=N.
\]
The configuration has \emph{full rank} when the indispensable block
$B:\C^m\to\C^m$ is invertible.  On this locus,
\[
 \GammaL=B^{-1}(2\pi i\Z^m)\cong\Z^m.
\]
The complex dual of the deck lattice is denoted by $\GammaLCdual$.

\begin{remark}[Full-rank notation]
\label{rem:rank-firewall}
For $\det B\neq0$, write $\GammaL=B^{-1}(2\pi i\Z^m)$ and use the
associated deck characters and discrete quotient. For deficient rank,
use the quotient by the connected $\C^m$-action
\[
 \XL\cong
 \bigl(U_K\times(\C^*)^m\bigr)/\C^m
\]
of Theorem~\ref{thm:global-ghost-quotient}.
\end{remark}

By Dolbeault's theorem \cite{Dolbeault1953}, the cohomology of
$\bar\partial$ on smooth $(p,q)$-forms agrees with sheaf cohomology.
We write these groups and their generating series as
\[
 H_{\bar\partial}^{p,q}(\XL)
 :=H^q(\XL,\Omega_{\XL}^p),
 \qquad
 h^{p,q}(\XL)
 :=\dim_\C H_{\bar\partial}^{p,q}(\XL),
\]
and
\[
 \operatorname{Dol}_{\XL}(u,v)
 :=\sum_{p,q}h^{p,q}(\XL)u^pv^q,
 \qquad
 \operatorname{Poin}_Y(t)
 :=\sum_{\ell\geq0}b_{\ell}(Y)t^{\ell}.
\]
The Hochster polynomial is
\[
 \mathscr H_K(u,v)
 :=
 \sum_{I\subseteq[N]}
 u^{|I|}
 \sum_{a=0}^{d}
 \widetilde b_{a-1}(K_I)v^a.
\]
In the two full-rank proofs, $p$, $q$, $a$, and $s$ denote respectively the
holomorphic, sheaf-cohomological, curvature or Cox--\v{C}ech, and deck
exterior degrees; $r\in\Z^N$ is a Laurent exponent.  A zero-sector class in
\[
 \widetilde H^{a-1}(K_I;\C)\otimes\Lambda^s\GammaLCdual
\]
therefore has bidegree
\[
 p=|I|,
 \qquad
 q=a+s.
\]

\subsection{Hypotheses}
Section~\ref{sec:LVM-foundations} begins with admissible configurations and
introduces stability in its last subsection. From Section~\ref{sec:affine-action-matrix} onward, we assume that the configuration is minimally stable and polytopal, with
$d\geq1$ and $K=\partial Q$ the boundary of a $d$-dimensional simplicial
polytope.  A \emph{full-rank}
statement also assumes $\det B\neq0$.  Statements at arbitrary rank concern every rank realized by an admissible configuration in this class;
Theorem~\ref{thm:real-indispensable-frame} gives the realizability
restriction.

For $d\geq1$, a simple $d$-polytope has at least $d+1$ facets, and hence
\[
 m=N-d\geq1.
\]
We use the conventions
\[
 \Lambda^s\C^m=0
 \quad(s<0\text{ or }s>m),
 \qquad
 \binom ms=0
 \quad(s<0\text{ or }s>m),
\]
and all reduced simplicial cohomology is augmented, including
\(
\widetilde H^{-1}(K_\varnothing;\C)=\C
\).
Since $\dim_\C\XL=N$, the groups
\(
H^q(\XL,\Omega_{\XL}^p)
\)
vanish outside $0\leq p,q\leq N$.

\Needspace{12\baselineskip}
\part{Curvature and Dolbeault cohomology}
\section{LVM configurations and visible Gale geometry}
\label{sec:LVM-foundations}

We recall the quotient construction and the convex conditions on its
weights.  Indispensable coordinates will provide an affine chart;
the remaining coordinates will index the visible polytope.

We use the classical LVM/LVMB conventions of
\cite{LopezMedranoVerjovsky1997,Meersseman2000,Bosio2001}; the
polytopal and moment-angle interpretation follows
\cite{BosioMeersseman2006,Tambour2012}.

Throughout this section,
\[
 \Lambda=(\Lambda_1,\ldots,\Lambda_n),\qquad
 \Lambda_i\in\C^m\simeq\R^{2m},
\]
and convex hulls are taken over the real numbers.  The bracket
$\langle-,-\rangle$ is the standard complex-bilinear pairing on $\C^m$;
its real part is regarded as a real bilinear pairing whenever real affine
geometry is used.

\begin{definition}
The configuration $\Lambda$ is \emph{admissible} if
\[
 0\in\operatorname{conv}(\Lambda_1,\ldots,\Lambda_n)
\]
and
\[
 0\notin\operatorname{conv}\{\Lambda_i:i\in I\}
 \qquad\text{for every }I\subseteq[n]\text{ with }1\leq|I|\leq2m.
\]
The second condition is weak hyperbolicity.
\end{definition}

\begin{remark}
By Carath\'eodory's theorem, the origin lies in the convex hull of at most
$2m+1$ weights. Weak hyperbolicity excludes supports of cardinality at most
$2m$, so every minimal supporting set has $2m+1$ elements.  In particular,
$n\geq2m+1$.
\end{remark}

\begin{lemma}\label{lem:small-supports}
If
\[
 0\in\operatorname{conv}\{\Lambda_i:i\in I\},
\]
then $|I|\geq 2m+1$.
\end{lemma}

\begin{proof}
A supporting set of at most $2m$ weights would contradict weak hyperbolicity.
\end{proof}

\begin{proposition}\label{prop:affine-spanning}
An admissible configuration affinely spans $\C^m\simeq\R^{2m}$.
Equivalently, the augmented vectors
\[
 (1,\Lambda_i)\in\R^{2m+1}
\]
span $\R^{2m+1}$.

Every subconfiguration whose convex hull contains the origin contains
$2m+1$ affinely independent weights whose convex hull contains the origin
in its relative interior.
\end{proposition}

\begin{proof}
If the affine span had real dimension $r<2m$, Carath\'eodory's theorem in
that affine span would place the origin in the convex hull of at most
$r+1\leq2m$ weights, contradicting Lemma~\ref{lem:small-supports}.

For the second assertion, choose a minimal supporting subconfiguration.
Minimality gives affine independence and positive coefficients. Its
cardinality is at most $2m+1$ by Carath\'eodory and at least $2m+1$ by
Lemma~\ref{lem:small-supports}.
\end{proof}

\begin{corollary}
The projective infinitesimal $\C^m$-action determined by $\Lambda$ is
effective.
\end{corollary}

\begin{proof}
If $T\in\C^m$ induces the zero projective vector field, then
$\langle\Lambda_i,T\rangle$ is independent of $i$. Hence
\[
 \langle\Lambda_i-\Lambda_1,T\rangle=0
\]
for every $i$. Proposition~\ref{prop:affine-spanning} implies $T=0$.
\end{proof}

Define
\[
 V_\Lambda=
 \left\{
 [z]\in\PP^{n-1}:
 0\in\operatorname{conv}\{\Lambda_i:z_i\neq0\}
 \right\}.
\]
The holomorphic action is
\[
 T\cdot[z_1:\cdots:z_n]
 =
 [e^{\langle\Lambda_1,T\rangle}z_1:\cdots:
 e^{\langle\Lambda_n,T\rangle}z_n].
\]
Its classical LVM transversal is
\[
 \mathscr T_\Lambda=
 \left\{
 [z]\in\PP^{n-1}:
 \sum_{i=1}^n\Lambda_i|z_i|^2=0
 \right\}.
\]

\begin{theorem}[Classical LVM quotient]
\label{thm:classical-LVM-quotient}
The holomorphic $\C^m$-action on $V_\Lambda$ is free and proper.  The
orbit space
\[
 \XL=V_\Lambda/\C^m
\]
is a compact Hausdorff complex manifold, the quotient map is a
holomorphic principal $\C^m$-bundle, and the transversal
$\mathscr T_\Lambda$ maps diffeomorphically onto $\XL$.
\end{theorem}

\begin{proof}
This is the classical LVM minimization theorem
\cite[Section~12]{BosioMeersseman2006}; we recall the minimization argument.  For $[z]\in V_\Lambda$, put
\[
 F_z(T)=\sum_{i\in\operatorname{supp}(z)}
 |z_i|^2e^{2\operatorname{Re}\langle\Lambda_i,T\rangle},
 \qquad T\in\C^m\cong\R^{2m}.
\]
The origin is an interior point of the convex hull of the support:
Proposition~\ref{prop:affine-spanning} supplies a full-dimensional
supporting simplex inside it.  Hence $F_z$ is proper and strictly convex.
Its unique critical point is exactly the unique parameter for which the
orbit of $[z]$ meets $\mathscr T_\Lambda$.  The positive-definite Hessian
and the implicit-function theorem give smooth dependence on $[z]$.
Consequently the action map
\[
 \C^m\times\mathscr T_\Lambda\longrightarrow V_\Lambda,
 \qquad (T,x)\longmapsto T\cdot x,
\]
is a diffeomorphism.  Translation on the first factor is free and proper,
so the original action is free and proper as well.  The remaining
assertions follow from the free proper holomorphic quotient theorem and
compactness of $\mathscr T_\Lambda$.
\end{proof}

\begin{proposition}\label{prop:LVM-dimension}
The equations defining $\mathscr T_\Lambda$ have real rank $2m$.
Consequently
\[
 \dim_\C \XL=n-m-1.
\]
\end{proposition}

\begin{proof}
At a point $[z]\in\mathscr T_\Lambda$, the weights indexed by the nonzero
coordinates affinely span $\R^{2m}$ by
Proposition~\ref{prop:affine-spanning}. On the unit sphere, radial
variations with total coefficient zero therefore make the differential of
\[
 z\longmapsto\sum_i\Lambda_i|z_i|^2
\]
surjective. The level set in $S^{2n-1}$ has real dimension
$2n-1-2m$. Quotienting by scalar $S^1$ gives
$2(n-m-1)$.
\end{proof}

\subsection{Visible and indispensable labels}

\begin{definition}
A label $a\in[n]$ is \emph{indispensable} if
\[
 0\notin\operatorname{conv}\{\Lambda_i:i\neq a\}.
\]
Let $D$ be the set of indispensable labels, let $k=|D|$, and let
\[
 V=[n]\setminus D,\qquad N=|V|=n-k.
\]
The labels in $V$ are called visible.
\end{definition}

\begin{lemma}
A label $a$ is indispensable if and only if $z_a\neq0$ at every point of
$V_\Lambda$.
\end{lemma}

\begin{proof}
If $a$ is indispensable, a point with $z_a=0$ has nonzero support contained
in $[n]\setminus\{a\}$ and cannot belong to $V_\Lambda$. The converse
follows by choosing a point supported on a convex relation in the complement
of $a$.
\end{proof}

Define
\[
 \mathcal K_\Lambda=
 \left\{
 J\subseteq[n]:
 0\in\operatorname{conv}\{\Lambda_i:i\notin J\}
 \right\}.
\]

\begin{proposition}\label{prop:Gale-dictionary}
The vertex set of $\mathcal K_\Lambda$ is $V$. Thus the indispensable
labels are precisely the ghost labels. For every $J\subseteq V$,
\[
 J\in\mathcal K_\Lambda
 \iff
 0\in\operatorname{conv}\{\Lambda_i:i\notin J\},
\]
where the complement is taken in the full label set.
\end{proposition}

\begin{proof}
The singleton $\{j\}$ belongs to $\mathcal K_\Lambda$ exactly when deletion
of $j$ preserves the Siegel condition, which is exactly the condition that
$j$ be visible.
\end{proof}

\begin{lemma}\label{lem:indispensable-bound}
If $n>2m+1$, then $k\leq2m$.
\end{lemma}

\begin{proof}
Choose a minimal supporting simplex
$J=\{j_0,\ldots,j_{2m}\}$ with
\[
 0=\sum_{i\in J}\alpha_i\Lambda_i,\qquad
 \alpha_i>0,\qquad \sum_{i\in J}\alpha_i=1.
\]
Choose $b\notin J$ and write
\[
 \Lambda_b=\sum_{i\in J}\beta_i\Lambda_i,\qquad
 \sum_{i\in J}\beta_i=1.
\]
At least one $\beta_i$ is positive. Set
\[
 t_0=\min_{\beta_i>0}\frac{\alpha_i}{\beta_i}.
\]
Then
\[
 0=t_0\Lambda_b+\sum_{i\in J}
 (\alpha_i-t_0\beta_i)\Lambda_i
\]
is a convex relation with at least one coefficient on $J$ equal to zero.
Thus some label in $J$ is not indispensable. Every label outside $J$ is
also nonindispensable, because deleting it leaves the supporting simplex
$J$. Hence at most $2m$ labels are indispensable.
\end{proof}

\begin{corollary}
If $n>2m+1$, then the full visible set does not belong to
$\mathcal K_\Lambda$.
\end{corollary}

\begin{proof}
Otherwise the indispensable weights would contain the origin in their
convex hull. Lemma~\ref{lem:small-supports} would give $k\geq2m+1$, while
Lemma~\ref{lem:indispensable-bound} gives $k\leq2m$.
\end{proof}

Let
\[
 \widetilde\Lambda:\R^n\longrightarrow\R^{2m+1},
 \qquad e_i\longmapsto(1,\Lambda_i).
\]
Proposition~\ref{prop:affine-spanning} gives
\[
 d:=\dim\ker\widetilde\Lambda=n-2m-1.
\]
In the polytopal LVM case, standard Gale duality identifies
$\mathcal K_\Lambda$ with the boundary complex
\[
 K=\partial Q
\]
of a $d$-dimensional simplicial convex polytope $Q=Q_\Lambda$ with
vertex set $V$.  We write $P=Q^\vee$ for its dual simple polytope.  In
particular,
\[
 N=n-k,\qquad d=n-2m-1,\qquad \dim K=d-1.
\]

\subsection{Stable excess and dimensions}

\begin{definition}
The configuration is \emph{stable} if $k>m$ and \emph{minimally stable} if
$k=m+1$. In the stable range, set
\[
 \epsilon_{\mathrm{st}}=k-m-1.
\]
\end{definition}

\begin{proposition}\label{prop:dimension-dictionary}
For a stable polytopal LVM manifold,
\[
 \dim_\C \XL=N+\epsilon_{\mathrm{st}},
 \qquad
 m=N-d+\epsilon_{\mathrm{st}}.
\]
Equivalently,
\[
 d=n-2m-1,\qquad
 N=n-k,\qquad
 \dim_\C \XL=n-m-1.
\]
\end{proposition}

\begin{proof}
Since $k=m+1+\epsilon_{\mathrm{st}}$,
\[
 N=n-k=n-m-1-\epsilon_{\mathrm{st}}.
\]
Proposition~\ref{prop:LVM-dimension} gives
\[
 \dim_\C \XL=n-m-1=N+\epsilon_{\mathrm{st}}.
\]
Moreover,
\[
 N-d=(n-k)-(n-2m-1)=2m+1-k=m-\epsilon_{\mathrm{st}},
\]
so $m=N-d+\epsilon_{\mathrm{st}}$.
\end{proof}

\begin{proposition}[Dimensions at minimal stability]
\label{cor:minimal-dimensions}
\label{prop:minimal-dimension-ledger}
In the minimally stable polytopal case,
\[
 m=N-d,\qquad
 k=N-d+1,\qquad
 n=2N-d+1,\qquad
 \dim_\C \XL=N.
\]
Moreover,
\[
 \widehat N:=N+m=n-1=2N-d,
 \qquad
 \dim_\R\XL=2N.
\]
For the visible moment-angle manifold,
\[
 \dim_\R\mathcal Z_K=N+d,
\]
so
\[
 \dim_\R(\mathcal Z_K\times T^m)
 =N+d+m=2N.
\]
At full rank, $\operatorname{rank}_\Z\GammaL=m=N-d$.
\end{proposition}

\begin{proof}
The first identities follow from Proposition~\ref{prop:dimension-dictionary}
with $\epsilon_{\mathrm{st}}=0$; those for $\widehat N$ and $\XL$ follow
by substitution.  The quotient-fan dimension,
computed after the marked fan is introduced, is
$\widehat N-2m=N-m=d$.  The standard dimension formula
for a moment-angle manifold over a simplicial $(d-1)$-sphere on $N$
vertices is $N+d$.  The remaining identities are immediate.
\end{proof}

\begin{corollary}\label{cor:parity-law}
For a fixed visible simplicial $d$-polytope $Q$ with $N$ vertices
(equivalently, a simple dual polytope $P$ with $N$ facets), a stable realization
with excess $\epsilon_{\mathrm{st}}$ satisfies
\[
 m=N-d+\epsilon_{\mathrm{st}},\qquad
 k=N-d+1+2\epsilon_{\mathrm{st}},\qquad
 n=2N-d+1+2\epsilon_{\mathrm{st}}.
\]
In particular,
\[
 k\equiv N-d+1\pmod2.
\]
Increasing $\epsilon_{\mathrm{st}}$ by one changes
\[
 (n,m,k)\longmapsto(n+2,m+1,k+2).
\]
\end{corollary}

\begin{remark}
We henceforth assume $d\geq1$ and write $P=Q^\vee$ and $K=\partial Q$. The
case $d=0$ requires a separate convention for the boundary of a
zero-dimensional polytope.
\end{remark}

\section{The affine Cox model and the quotient by \texorpdfstring{$\C^m$}{C-m}}
\label{sec:affine-ghost-quotient}

An indispensable coordinate is nonzero on the whole projective open
set.  Taking it as denominator gives a product of a coordinate-subspace
complement and a torus.  The quotient by the connected $\C^m$-action exists at every realizable rank; when $B$ is invertible, it also has a discrete
covering description.

Let
\[
 D=\{a_0,\ldots,a_{k-1}\}
\]
be the indispensable labels and let $V=[n]\setminus D$ be the visible
labels. Write $N=|V|$ and let $K$ be the visible Gale complex.

\subsection{The global affine product}

Since every indispensable coordinate is nonzero on $V_\Lambda$, the whole
admissible locus belongs to the chart $z_{a_0}\neq0$. Define
\[
 w_j=\frac{z_j}{z_{a_0}}\quad (1\leq j\leq N),
 \qquad
 t_\mu=\frac{z_{a_\mu}}{z_{a_0}}\in\C^*
 \quad(1\leq\mu\leq k-1).
\]
Set
\[
 Z(w)=\{j:w_j=0\},
 \qquad
 U_K=\{w\in\C^N:Z(w)\in K\}.
\]

\begin{lemma}[Connected Cox open set]
\label{lem:UK-connected}
The space $U_K$ is path connected.  Consequently, whenever $k\geq1$,
$V_\Lambda$ and $\XL$ are connected.
\end{lemma}

\begin{proof}
Every visible singleton is a face of $K$.  Hence the coordinate-subspace
arrangement removed from $\C^N$ to form $U_K$ has complex codimension at
least two.  Its complement is path connected.  The product description
of Theorem~\ref{thm:global-affine-product} and the surjective quotient map
then give the final assertion.
\end{proof}

\begin{theorem}[Global affine product]
\label{thm:global-affine-product}
Assume $k\geq1$ and choose $a_0\in D$.  The map
\[
 \Phi_{a_0}:V_\Lambda\longrightarrow
 U_K\times(\C^*)^{k-1},
 \qquad
 [z]\longmapsto(w,t),
\]
is a biholomorphism. Consequently,
\[
 V_\Lambda\cong U_K\times(\C^*)^{k-1}.
\]
In the minimally stable case $k=m+1$,
\[
 V_\Lambda\cong U_K\times(\C^*)^m.
\]
\end{theorem}

\begin{proof}
All indispensable coordinates are nonzero, so $\Phi_{a_0}$ is globally
defined. In the chosen chart the zero-coordinate set is the visible set
$Z(w)$. By the definition of the Gale complex,
\[
 [z]\in V_\Lambda
 \iff
 0\in\operatorname{conv}\{\Lambda_i:i\notin Z(w)\}
 \iff
 Z(w)\in K.
\]
The inverse is
\[
 (w,t)\longmapsto[w_1:\cdots:w_N:1:t_1:\cdots:t_{k-1}].
\]
\end{proof}

\begin{remark}
Projective normalization fixes one of the $k$ indispensable coordinates and
leaves $k-1$ torus coordinates.
Indeed,
\[
 N+(k-1)=n-1=\dim_\C V_\Lambda.
\]
\end{remark}

\begin{proposition}[Change of indispensable chart]
Choosing another indispensable coordinate as the projective denominator
changes the product coordinates by a holomorphic monomial transformation.
For example, if $a_\nu$ is chosen, then
\[
 w'_j=\frac{w_j}{t_\nu},\qquad
 t'_0=\frac1{t_\nu},\qquad
 t'_\mu=\frac{t_\mu}{t_\nu}\quad(\mu\neq\nu).
\]
This transformation preserves $U_K$.
On exponent lattices it is induced by an integral unimodular map.  It
carries $\mathfrak h_\Lambda$, $\mathfrak r_\Lambda$, the coordinate fan,
and the marked-fan quadruple constructed below to isomorphic data.
\end{proposition}

\begin{proof}
The formulas follow by dividing by
$z_{a_\nu}=t_\nu z_{a_0}$. Since $t_\nu\neq0$, the visible zero set is
unchanged.  The displayed monomials and their inverse have integral
exponents, so their maps on exponent lattices are mutually inverse over
$\Z$.  Equivariance carries the connected action algebra and its real part
to the corresponding subspaces, and the zero-set description carries each
coordinate cone and its quotient marking to the new chart.
\end{proof}

\subsection{The affine action matrix}
\label{sec:affine-action-matrix}

From now on assume minimal stability and write
\[
 D=\{a_0,a_1,\ldots,a_m\}.
\]
For $T\in\C^m$, define
\[
 A_j(T)=\langle\Lambda_j-\Lambda_{a_0},T\rangle,
 \qquad
 B_\mu(T)=\langle\Lambda_{a_\mu}-\Lambda_{a_0},T\rangle.
\]
These rows define
\[
 A:\C^m\longrightarrow\C^N,\qquad
 B:\C^m\longrightarrow\C^m,
\]
and
\[
 \Psi_\Lambda=
 \begin{pmatrix}A\\B\end{pmatrix}:
 \C^m\longrightarrow\C^{\widehat N}.
\]

\begin{proposition}\label{prop:affine-action}
Under Theorem~\ref{thm:global-affine-product}, the LVM action is
\[
 T\cdot(w,t)=\bigl(e^{A(T)}w,e^{B(T)}t\bigr).
\]
The map $\Psi_\Lambda$ is injective.
\end{proposition}

\begin{proof}
The action formula follows by taking ratios with the $a_0$-coordinate.
If $\Psi_\Lambda(T)=0$, then
\[
 \langle\Lambda_i-\Lambda_{a_0},T\rangle=0
\]
for every label $i$. Since the weights affinely span $\C^m$ as a real
vector space, $T=0$.
\end{proof}

\begin{proposition}[Intrinsic rank and chart invariance]
\label{prop:rank-chart-invariance}
Under the standing minimal-stability hypothesis, let $D$ be the
indispensable set and put
\[
 \rho_\C(D)=
 \dim_\C\operatorname{span}_\C
 \{\Lambda_a-\Lambda_b:a,b\in D\}.
\]
Then $\rho_\C(D)=\operatorname{rank}_\C B$; in particular, the rank is
independent of the ordering of $D$ and of the distinguished denominator.
It is invariant under invertible complex-affine changes of the weights.

More precisely, a change of indispensable chart replaces $B$ by
$B'=LB$ for some $L\in\operatorname{GL}_m(\Z)$.  At full rank the deck
lattice has the intrinsic description
\[
 \GammaL=
 \left\{T\in\C^m:
 \langle\Lambda_a-\Lambda_b,T\rangle\in2\pi i\Z
 \text{ for all }a,b\in D\right\},
\]
and both this subgroup and its action on $U_K$ are independent of the
chosen indispensable chart.
\end{proposition}

\begin{proof}
The first assertion is the definition of the row span of $B$.  Replacing
the base label by $a_\nu$ expresses the new difference basis of the
augmentation kernel
\[
 \ker\bigl(\Z^D\longrightarrow\Z\bigr)
\]
in terms of the old one, so $B'=LB$ with
$L\in\operatorname{GL}_m(\Z)$.  The displayed intrinsic condition is
therefore equivalent to $B(T)\in2\pi i\Z^m$, and
\[
 (B')^{-1}(2\pi i\Z^m)=B^{-1}(2\pi i\Z^m).
\]
In the chart based at $a_\nu$, each visible row is replaced by
$A'_j=A_j-B_\nu$.  For $\gamma\in\GammaL$ the factor
$e^{-B_\nu(\gamma)}$ is one, hence
$e^{A'_j(\gamma)}=e^{A_j(\gamma)}$.  This proves chart independence of
the deck action.  Complex-affine automorphisms preserve the complex span
of all indispensable differences.
\end{proof}

\begin{theorem}[Real indispensable frame and realizable rank bound]
\label{thm:real-indispensable-frame}
For an admissible configuration, every collection of indispensable
weights is real-affinely independent.  In the minimally stable case
$D=\{a_0,\ldots,a_m\}$, the differences
\[
 \Lambda_{a_1}-\Lambda_{a_0},\ldots,
 \Lambda_{a_m}-\Lambda_{a_0}
\]
are therefore real-linearly independent.  Consequently,
\[
 \operatorname{rank}_{\C}B
 \geq
 \left\lceil\frac m2\right\rceil,
 \qquad
 0\leq\delta=m-\operatorname{rank}_{\C}B
 \leq
 \left\lfloor\frac m2\right\rfloor.
\]
\end{theorem}

\begin{proof}
Choose a minimal supporting subconfiguration $S$.  Proposition~
\ref{prop:affine-spanning} supplies inside $S$ a $2m+1$-element
real-affinely independent support of the origin; minimality forces this
support to equal $S$.  Every indispensable label belongs to $S$:
otherwise the same convex relation would survive after that label were
deleted.  Hence the indispensable weights, being a subset of a real
affine simplex, are real-affinely independent.

Under minimal stability their $m$ differences from
$\Lambda_{a_0}$ are real-linearly independent.  If
$r=\operatorname{rank}_{\C}B$, their complex span has real dimension
$2r$ and contains their $m$-dimensional real span.  Thus $2r\geq m$,
which gives the two displayed inequalities.
\end{proof}

\begin{theorem}[Realizable ranks at fixed visible complex]
\label{thm:exact-realizable-rank-range}
Fix an admissible minimally stable configuration $\Lambda$ and its visible
complex $K$.  For every integer
\[
 \left\lceil\frac m2\right\rceil\leq r\leq m
\]
there is an admissible minimally stable configuration $\Lambda^{(r)}$ with
the same subsets of labels whose weights contain the origin in their convex hull, and hence the same indispensable set and visible complex $K$, such that
\[
 \operatorname{rank}_\C B^{(r)}=r.
\]
Polytopality is preserved.  Thus the displayed interval is the set
of geometrically realizable complex ranks, even with $K$ fixed.
\end{theorem}

\begin{proof}
Write $v_\mu=\Lambda_{a_\mu}-\Lambda_{a_0}$.  By
Theorem~\ref{thm:real-indispensable-frame}, the $v_\mu$ form a real
$m$-frame in the real vector space $\C^m$.  Put $s=m-r$; the inequalities
on $r$ give $0\leq s\leq r$.  The ordered list
\[
 e_1,\ldots,e_r,ie_1,\ldots,ie_s
\]
is a real $m$-frame in $\C^m$ whose complex span has dimension $r$.
Extend the real-linear isomorphism carrying $(v_1,\ldots,v_m)$ to this
frame to an element $F\in\operatorname{GL}_\R(\C^m)$, and set
$\Lambda_i^{(r)}=F(\Lambda_i)$.  For every subset $I$,
\[
 0\in\operatorname{conv}\{\Lambda_i:i\in I\}
 \quad\Longleftrightarrow\quad
 0\in\operatorname{conv}\{\Lambda_i^{(r)}:i\in I\}.
\]
Hence all admissibility, indispensability, and visible-face tests are
unchanged.  The complex rank of the transformed indispensable differences
is $r$ by construction.  If $\widetilde\Lambda$ denotes the real augmented
weight matrix, then
\[
 \widetilde\Lambda^{(r)}=\operatorname{diag}(1,F)\widetilde\Lambda,
 \qquad
 \ker\widetilde\Lambda^{(r)}=\ker\widetilde\Lambda.
\]
Thus the full Gale dual data and the abstract visible complex are
unchanged, and polytopality is preserved.
\end{proof}

\begin{remark}
Real-linear changes of the weights produce individual configurations
at each rank.
Invertible complex-affine changes preserve rank by
Proposition~\ref{prop:rank-chart-invariance}. A holomorphic family crossing a
rank drop is constructed for the square in
Theorem~\ref{thm:holomorphic-square-rank-drop}.
\end{remark}

\begin{example}[A singular indispensable block over the square]
\label{ex:singular-square}
Let $m=2$ and take the seven weights in $\C^2$
\[
\begin{aligned}
 \Lambda_1&=(1,1),
 &\Lambda_2&=(i,1),
 &\Lambda_3&=(-1-i,1),\\
 \Lambda_4&=\left(\tfrac12,-1+i\right),
 &\Lambda_5&=\left(-\tfrac12,-1+i\right),\\
 \Lambda_6&=\left(\tfrac i2,-2-i\right),
 &\Lambda_7&=\left(-\tfrac i2,-2-i\right).
\end{aligned}
\]
Put
\[
 D=\{1,2,3\},\qquad E=\{4,5\},\qquad F=\{6,7\}.
\]
In this example the visible labels are $4,5,6,7$; in the standing convention
they are relabelled $1,2,3,4$.
This is an admissible minimally stable polytopal LVM configuration with
visible complex
\[
 K=\partial\Delta^1*\partial\Delta^1,
\]
the boundary of a square, and its indispensable block has
\[
 \operatorname{rank}_{\C}B=1<2=m.
\]
In particular, the bound in
Theorem~\ref{thm:real-indispensable-frame} is attained.
\end{example}

\begin{proof}
Consider a convex relation for the origin and let $S,E_0,F_0$ be the
sums of its coefficients on $D,E,F$, respectively.  The second complex
coordinate and the normalization give
\[
 S-E_0-2F_0=0,\qquad E_0-F_0=0,
 \qquad S+E_0+F_0=1,
\]
hence
\[
 S=\frac35,\qquad E_0=F_0=\frac15.
\]
After normalizing the coefficients on $D$, their first-coordinate convex
combination is a point $x+iy$.  Since the $E$-weights contribute only to
the real part of the visible first coordinate and the $F$-weights only to
its imaginary part, normalization by $S=3/5$ gives
$|x|,|y|\leq1/6$.  Its barycentric
coordinates in the triangle with vertices $1,i,-1-i$ are
\[
 \frac{1+2x-y}{3},\qquad
 \frac{1-x+2y}{3},\qquad
 \frac{1-x-y}{3},
\]
and all three are strictly positive on that closed square.  Every convex
relation therefore uses all labels in $D$ and at least one label from
each of $E$ and $F$.  Its support has at least five elements, so weak
hyperbolicity holds because $2m=4$.

For each choice of $e\in E$ and $f\in F$, the following table gives a strictly positive relation.  Each row lists
the coefficients on $(\Lambda_1,\Lambda_2,\Lambda_3,\Lambda_e,\Lambda_f)$:
\[
\begin{array}{c|c}
 (e,f)&\text{coefficients}\\ \hline
 (4,6)&(1/6,1/6,4/15,1/5,1/5)\\
 (4,7)&(1/10,3/10,1/5,1/5,1/5)\\
 (5,6)&(3/10,1/10,1/5,1/5,1/5)\\
 (5,7)&(7/30,7/30,2/15,1/5,1/5).
\end{array}
\]
For any prescribed visible label, a row using the other member of its
pair is a positive relation avoiding that label.  Thus deleting any
visible label preserves the Siegel condition, whereas the preceding
argument shows that every convex relation uses all of $D$.  Exactly the
labels $1,2,3$ are therefore indispensable, and $k=3=m+1$.  A visible
subset is a face precisely when its complement retains at least one label
of each of $E$ and $F$.
Therefore
\[
 K=\{J\subseteq E\sqcup F:E\not\subseteq J, F\not\subseteq J\}
   =\partial\Delta^1*\partial\Delta^1.
\]
Finally, with $a_0=1,a_1=2,a_2=3$,
\[
 \Lambda_2-\Lambda_1=(-1+i,0),\qquad
 \Lambda_3-\Lambda_1=(-2-i,0).
\]
These two nonzero rows span one complex line, so
$\operatorname{rank}_{\C}B=1$ and $\delta=1$.
\end{proof}

\begin{remark}[The rank assertion in Bosio--Meersseman]
In the proof of Bosio--Meersseman, Lemma~12.1 \cite{BosioMeersseman2006}, the lattice construction uses complex rank $m$ for the indispensable differences. The preceding example has $m=2$ and complex rank one. Weak hyperbolicity gives
the real affine independence of
Theorem~\ref{thm:real-indispensable-frame}; the lattice and affine-chart
argument requires the additional hypothesis $\det B\neq0$.  The quotient by the connected $\C^m$-action exists at every admissible rank by Theorem~\ref{thm:classical-LVM-quotient}.
Faucard's Example~3 already exhibits an admissible $(3,9,4)$
configuration whose four indispensable columns have complex rank two and
explicitly perturbs them to rank three
\cite[Example~3]{Faucard2024}.  Theorems~\ref{thm:exact-realizable-rank-range}
and~\ref{thm:curvature-exact-sequence} determine the realizable ranks and the
curvature defect; Section~\ref{sec:universal-curvature-formula} computes the
resulting Dolbeault groups.
\end{remark}

\subsection{Ghost vertices}

Let $G=\{a_1,\ldots,a_m\}$. Define $\widehat K$ on the ambient label set
$[N]\sqcup G$ by declaring its faces to be exactly the faces of $K$.
Thus every label in $G$ is a ghost vertex.

\begin{proposition}\label{prop:ghost-open-set}
One has
\[
 U_{\widehat K}=U_K\times(\C^*)^m.
\]
Moreover,
\[
 \mathcal Z_{\widehat K}\cong\mathcal Z_K\times(S^1)^m.
\]
\end{proposition}

\begin{proof}
A zero set belongs to $\widehat K$ exactly when it contains no ghost label
and its visible part belongs to $K$. The moment-angle statement follows
from the same observation for the polyhedral product $(D^2,S^1)^{\widehat
K}$.
\end{proof}

\subsection{The quotient by \texorpdfstring{$\C^m$}{C-m}}

Set
\[
 \mathfrak h_\Lambda=\operatorname{im}\Psi_\Lambda
 \subseteq\C^{\widehat N}
\]
and define
\[
 \rho_\Lambda(T)=\exp(\Psi_\Lambda T),\qquad
 \mathcal C_\Lambda=\rho_\Lambda(\C^m)=\exp(\mathfrak h_\Lambda).
\]

\begin{theorem}[Global ghost-vertex presentation]
\label{thm:global-ghost-quotient}
Assume minimal stability.  The homomorphism
$\rho_\Lambda:\C^m\to(\C^*)^{N+m}$ is a proper holomorphic embedding,
and its image $\mathcal C_\Lambda$ is a closed complex Lie subgroup
isomorphic to $(\C^m,+)$.  Its action on $U_{\widehat K}$ is free and
proper, with compact quotient.  After choosing the distinguished
indispensable label $a_0$, there is a biholomorphism
\[
 \XL\cong
 U_{\widehat K}/\C^m
 =
 U_{\widehat K}/\mathcal C_\Lambda,
\]
where
\[
 T\cdot z=\exp(\Psi_\Lambda T)z.
\]
This presentation is valid at every geometrically realizable complex rank
of $B$.
\end{theorem}

\begin{proof}
The biholomorphism
\[
 V_\Lambda\overset{\sim}{\longrightarrow}
 U_K\times(\C^*)^m=U_{\widehat K}
\]
is equivariant by Proposition~\ref{prop:affine-action}. It therefore
transfers freeness, properness, compactness of the quotient, and the
complex structure from Theorem~\ref{thm:classical-LVM-quotient}.
Injectivity of $\rho_\Lambda$ follows from freeness.  The orbit of
$\mathbf1=(1,\ldots,1)$ is $\mathcal C_\Lambda$; its orbit map is proper
with trivial stabilizer, hence a proper holomorphic embedding with closed
image.  The equivariant biholomorphism identifies the orbit equivalence
relations and their analytic quotients.
\end{proof}

\subsection{The full-rank slice}

\begin{theorem}[Regular holomorphic Cox covering]
\label{thm:full-rank-Cox-quotient}
Assume $\det B\neq0$ and set
\[
 \GammaL=B^{-1}(2\pi i\Z^m).
\]
Then $\GammaL\cong\Z^m$ acts freely and properly discontinuously on $U_K$
by
\[
 \gamma\cdot w=e^{A(\gamma)}w,
\]
and
\[
 \XL\cong_{\mathrm{bihol}} U_K/\GammaL.
\]
Moreover, $U_K\to\XL$ is a regular holomorphic covering whose deck group
is $\GammaL$.
\end{theorem}

\begin{proof}
The map
\[
 T\longmapsto e^{B(T)}
\]
is onto $(\C^*)^m$ and has kernel $\GammaL$. Hence every orbit in
$U_K\times(\C^*)^m$ meets the slice $t=\mathbf1$. Two points of the slice
belong to the same orbit exactly when they differ by an element of
$\GammaL$, and freeness follows from freeness of the connected action.

For proper discontinuity, let $C_1,C_2\subset U_K$ be compact and put
$\widetilde C_i=C_i\times\{\mathbf1\}$.  Properness of the connected
$\C^m$-action makes
\[
 \{T\in\C^m:T\widetilde C_1\cap\widetilde C_2\neq\varnothing\}
\]
compact.  Its intersection with the discrete subgroup $\GammaL$ is
finite.

To identify the analytic quotient, choose a simply connected neighbourhood
$W\subset(\C^*)^m$ with a holomorphic branch of $\log$.  The gauge
\[
 (w,t)\longmapsto e^{-A(B^{-1}\log t)}w
\]
is holomorphic on $U_K\times W$ and sends every connected orbit to the
slice $t=\mathbf1$.  Two such gauges differ on overlaps by the action of an
element of $\GammaL$.  They are therefore local covering charts for
$U_K\to U_K/\GammaL$ and identify this quotient biholomorphically with the
quotient by the connected $\C^m$-action. Since $U_K$ is connected, the quotient action gives
the full deck group and the covering is regular.
\end{proof}

\begin{proposition}
If $\operatorname{rank}_\C B<m$, then the slice $t=\mathbf1$ does not meet
every orbit, and the kernel of $T\mapsto e^{B(T)}$ contains the
positive-dimensional vector group $\ker B$.  In particular, the preceding
full-rank slice construction does not reduce the quotient to a discrete
action on $U_K$.
\end{proposition}

\begin{proof}
The image of $T\mapsto e^{B(T)}$ has Lie algebra
$\operatorname{im}B\subsetneq\C^m$, so it cannot equal $(\C^*)^m$; and
$\ker B$ is a positive-dimensional subgroup of its kernel.
\end{proof}

\section{The marked fan and the canonical foliation}
\label{sec:marked-fan-foliation}

Project the coordinate fan along the real action algebra.  The
resulting fan describes the geometry transverse to the canonical
holomorphic foliation.  Its ray relations will give the linear
relations in basic cohomology.

Put $\widehat N=N+m=n-1$ and
\[
 \mathfrak h_\Lambda=\operatorname{im}\Psi_\Lambda
 \subseteq\C^{\widehat N},\qquad
 \mathfrak r_\Lambda=\operatorname{Re}(\mathfrak h_\Lambda)
 \subseteq\R^{\widehat N}.
\]
The coordinate fan of $U_{\widehat K}$ is
\[
 \Sigma_{\widehat K}
 =
 \left\{
 \R_{\geq0}\langle e_i:i\in I\rangle:I\in\widehat K
 \right\}.
\]

\begin{proposition}\label{prop:real-part-injective}
The restriction
\[
 \operatorname{Re}|_{\mathfrak h_\Lambda}:
 \mathfrak h_\Lambda\longrightarrow\R^{\widehat N}
\]
is injective. Hence
\[
 \dim_\R\mathfrak r_\Lambda=2m
\]
and
\[
 (\mathfrak r_\Lambda)_\C
 =
 \mathfrak h_\Lambda\oplus\overline{\mathfrak h_\Lambda}.
\]
\end{proposition}

\begin{proof}
Write $\xi=\Psi_\Lambda(T)$ and suppose $\operatorname{Re}\xi=0$.  Then
\[
 \operatorname{Re}\langle\Lambda_i-\Lambda_{a_0},T\rangle=0
 \qquad\text{for every label }i.
\]
The real affine differences of the weights span $\C^m$ as a real vector
space by Proposition~\ref{prop:affine-spanning}.  Nondegeneracy of the real
part of the complex-bilinear pairing therefore gives $T=0$, hence $\xi=0$.

A nonzero real vector in $\mathfrak h_\Lambda$ would produce a nonzero
vector with zero real part after multiplication by $i$. Thus
$\mathfrak h_\Lambda$ contains no nonzero real vector. It follows that
$\mathfrak h_\Lambda\cap\overline{\mathfrak h_\Lambda}=0$, and the final
identity follows by dimensions.
\end{proof}

Let
\[
 \varpi_\Lambda:\R^{\widehat N}\longrightarrow
 \mathfrak t_\Lambda=\R^{\widehat N}/\mathfrak r_\Lambda
\]
be the quotient map.
We normalize the coordinate-torus exponential by
\[
 \exp_T:\R^{\widehat N}\longrightarrow T^{\widehat N},
 \qquad
 x\longmapsto(e^{2\pi i x_1},\ldots,e^{2\pi i x_{\widehat N}}),
\]
so that its kernel is $\Z^{\widehat N}$.

\begin{proposition}\label{prop:cone-transversality}
For every $I\in\widehat K$,
\[
 \mathfrak r_\Lambda\cap\R^I=0.
\]
Consequently, $\varpi_\Lambda$ is injective on every cone of
$\Sigma_{\widehat K}$.
\end{proposition}

\begin{proof}
Suppose $0\neq x\in\mathfrak r_\Lambda\cap\R^I$, and choose
$\xi\in\mathfrak h_\Lambda$ with $\operatorname{Re}\xi=x$. Let $z$ have
zero-coordinate set exactly $I$. Outside $I$, the orbit
$\exp(t\xi)z$ has constant coordinate moduli; on $I$ the coordinates remain
zero. Hence the orbit is precompact, contradicting properness.
\end{proof}

Since $m=N-d$,
\[
 \dim_\R\mathfrak t_\Lambda
 =
 (N+m)-2m=d.
\]

\begin{theorem}[Complete marked fan]
\label{thm:complete-marked-fan}
Put
\[
 a_i=\varpi_\Lambda(e_i),\qquad
 \widetilde\Gamma_\Lambda=
 \varpi_\Lambda(\Z^{\widehat N})\subset\mathfrak t_\Lambda.
\]
The collection
\[
 \Sigma_\Lambda
 =
 \varpi_\Lambda(\Sigma_{\widehat K})
\]
is a complete simplicial normal fan in $\mathfrak t_\Lambda$.  Its rays
are
\[
 \rho_j=\R_{\geq0}a_j,\qquad 1\leq j\leq N,
\]
and its underlying simplicial complex is $K$.  The map
\[
 \widetilde\lambda_\Lambda(\rho_j)=a_j
\]
makes
\[
 (\mathfrak t_\Lambda,\widetilde\Gamma_\Lambda,
   \Sigma_\Lambda,\widetilde\lambda_\Lambda)
\]
a complete marked fan.  The ray marking is indexed by the visible labels; ghost labels remain
in the redundant ambient presentation.
\end{theorem}

\begin{proof}
The projected-fan correspondence and its LVM polytopality criterion are
due to Battisti \cite[Theorems~2.2 and~3.10]{Battisti2013}.
We express them here in the later marked-fan notation.  The free proper exponential action of
Theorem~\ref{thm:global-ghost-quotient} makes the projected coordinate
cones into a simplicial fan, with their intersections equal to common
faces, by \cite[Theorem~10.3]{Panov2025}.  Compactness of the quotient
then makes this fan complete by \cite[Theorem~10.5]{Panov2025}.
Proposition~\ref{prop:cone-transversality} also verifies cone-wise
injectivity directly.  The one-dimensional coordinate cones occur exactly
at the visible labels, so the underlying complex is $K$; the ghost vectors remain part of the ambient presentation, with rays
indexed by the visible labels.

For LVM data the projected fan is normal by the LVM normal-fan
construction \cite[Construction~13.4]{Panov2025}.  Finally,
$\widetilde\Gamma_\Lambda$ is a finitely generated subgroup spanning
$\mathfrak t_\Lambda$, and the displayed ray marking is precisely the
quotient marking of
\cite[Definition~5.5 and Construction~5.6]{IshidaKrutowskiPanov2022}.
This subgroup may be nondiscrete. The group $\GammaL$ acts on the Cox
cover, whereas $\widetilde\Gamma_\Lambda$ records the projected fan.
\end{proof}

\begin{theorem}[Moment-angle topology and smooth quadrics]
\label{thm:smooth-moment-angle-model}
For a minimally stable polytopal LVM manifold,
\[
 \XL\cong_{T^{\widehat N},\mathrm{homeo}}\mathcal Z_{\widehat K}
 \cong_{T^{\widehat N},\mathrm{homeo}}\mathcal Z_K\times T^m.
\]
Moreover, $\XL$ is $T^{\widehat N}$-equivariantly diffeomorphic to a
nondegenerate Hermitian-quadrics realization
$\mathcal Z_{\widehat K}^{\mathrm{quad}}$ whose underlying topological
$T^{\widehat N}$-space is $\mathcal Z_{\widehat K}$.
In particular,
\[
 \dim_\R\XL
 =\dim_\R\mathcal Z_K+m
 =(N+d)+m
 =2N.
\]
\end{theorem}

\begin{proof}
The equivariant homeomorphism with the polyhedral-product moment-angle
space is the exponential-action realization
\cite[Theorem~10.6]{Panov2025}.  Each ghost label contributes an $S^1$
factor by \cite[Construction~10.8]{Panov2025}.  Since
$\Sigma_\Lambda$ is normal, the smooth identification with a nondegenerate
Hermitian-quadrics realization follows from
\cite[Theorem~10.11]{Panov2025}.  The orbit-intersection and radial gauges
in these constructions commute with coordinate rotations, which gives the
stated $T^{\widehat N}$-equivariance.  Apply
Theorems~\ref{thm:global-ghost-quotient} and
\ref{thm:complete-marked-fan} and
Proposition~\ref{prop:ghost-open-set}.  The dimension identity uses
$m=N-d$.
\end{proof}

The coordinate torus $T^{\widehat N}$ descends effectively to $\XL$.
Indeed, its kernel is $T^{\widehat N}\cap\mathcal C_\Lambda$.  The Lie
algebras intersect trivially by
Proposition~\ref{prop:real-part-injective}, so this kernel is discrete;
it is finite because it lies in a compact torus, and it is trivial because
$\mathcal C_\Lambda\cong\C^m$ has no nontrivial finite subgroup.  Its action
is maximal in the sense that there is a point $x$ with
\[
 \dim T^{\widehat N}+\dim(T^{\widehat N})_x=\dim_\R \XL.
\]
Indeed, choose a facet $I\in K$ and a point represented by a vector whose
zero set is exactly $I$.  At the infinitesimal level, a coordinate-torus
vector stabilizes its quotient class precisely when it differs from a
vector in $\mathfrak h_\Lambda$ by a vector supported on $I$.  Taking real
parts and using Proposition~\ref{prop:cone-transversality}, followed by
Proposition~\ref{prop:real-part-injective}, shows that the
$\mathfrak h_\Lambda$ vector is zero.  Thus the stabilizer Lie algebra is
$\R^I$ and has dimension $|I|=d$.  Consequently
\[
 \dim T^{\widehat N}+d=(N+m)+d=2N=\dim_\R\XL.
\]
This is the maximal-torus structure classified in
\cite[Theorem~7.9]{Ishida2019}.

The canonical foliation is generated by
\[
 H_{\mathcal F,\Lambda}=\exp_T(\mathfrak r_\Lambda)\subseteq T^{\widehat N}.
\]
The orbits of this immersed subgroup define the foliation. Its image
in the compact torus may be nonclosed.

\begin{proposition}[Canonical central foliation]
\label{prop:canonical-central-foliation}
The canonical foliation $\mathcal F_\Lambda$ is a locally free central
holomorphic foliation of complex dimension $m$. It is transversely
K\"ahler.
\end{proposition}

\begin{proof}
Its complex leaf algebra is
\[
 (\mathfrak r_\Lambda)_\C/\mathfrak h_\Lambda
 \cong\overline{\mathfrak h_\Lambda},
\]
which has dimension $m$.  We verify the central-foliation identification
directly.  In the complex Lie algebra
\[
 \C^{\widehat N}/\mathfrak h_\Lambda
 =\operatorname{Lie}\bigl((\C^*)^{\widehat N}/
   \exp(\mathfrak h_\Lambda)\bigr),
\]
the compact coordinate-torus algebra is represented by classes $[iv]$
with $v\in\R^{\widehat N}$.  Such a class belongs also to its complex
rotation precisely when
\[
 [iv]=[-w]\quad\text{for some }w\in\R^{\widehat N},
 \qquad\text{equivalently}\qquad w+iv\in\mathfrak h_\Lambda.
\]
Such a relation implies $w\in\mathfrak r_\Lambda$; multiplying it by
$-i$ also gives $v\in\mathfrak r_\Lambda$.  Conversely, if
$v\in\mathfrak r_\Lambda$, choose $v+iy\in\mathfrak h_\Lambda$; then
$i(v+iy)=-y+iv\in\mathfrak h_\Lambda$ supplies the required $w=-y$.
Thus the intersection of the torus algebra with its complex rotation is
exactly the image of $i\mathfrak r_\Lambda$, and its orbit foliation is
$\mathcal F_\Lambda$.  This is the central foliation in the sense of
\cite[Section~2]{IshidaKasuya2019}.

Under the complex moment-angle presentation of
Theorem~\ref{thm:global-ghost-quotient}, this is also the canonical
foliation of \cite[Construction~4.2]{IshidaKrutowskiPanov2022}; the
discrete-stabilizer statement there gives local freeness.  Finally, the
transverse K\"ahler assertion for this LVM foliation is
\cite[Theorem~7]{Meersseman2000}; see also the Fubini--Study description
in \cite{MeerssemanVerjovsky2004}.  In the present fan notation it is
the normal-fan case of
\cite[Proposition~4.4]{PanovUstinovskiyVerbitsky2016}, with normality
supplied by Theorem~\ref{thm:complete-marked-fan}.
\end{proof}

\section{The transverse Dolbeault model}
\label{sec:imported-model-legality}

The canonical foliation is central and transversely K\"ahler at every
realizable rank. The transverse Dolbeault models therefore apply throughout
this range.

\begin{proposition}[Applicability at every realizable rank]
\label{prop:imported-model-legality}
For every minimally stable polytopal LVM datum, at every geometrically
realizable value of $\operatorname{rank}_{\C}B$, the following hold.
\begin{enumerate}[label=(\roman*)]
\item $\XL$ is a compact connected complex moment-angle manifold in the
      presentation
      \[
       \XL\cong U_{\widehat K}/\exp(\mathfrak h_\Lambda),
      \]
      with $\dim_\C\mathfrak h_\Lambda=m$ and
      $\operatorname{Re}|_{\mathfrak h_\Lambda}$ injective.
\item The coordinate torus acts effectively and maximally, and it has a
      point with trivial stabilizer.
\item The canonical foliation $\mathcal F_\Lambda$ is locally free,
      central, and transversely K\"ahler of complex dimension $m$.
\item The basic-cohomology theorems
      \cite[Theorem~5.10]{IshidaKrutowskiPanov2022} and
      \cite[Theorems~4.12 and~5.1]{KrutowskiPanov2021}, and the ordinary
      Dolbeault-model theorems
      \cite[Theorems~1.2 and~4.13]{IshidaKasuya2019} and
      \cite[Theorem~6.1]{KrutowskiPanov2021}, all apply to this
      presentation.
\end{enumerate}

\end{proposition}

\begin{proof}
The quotient by the connected $\C^m$-action, compactness, and connectedness are
Theorems~\ref{thm:global-ghost-quotient} and
\ref{thm:global-affine-product} together with
Lemma~\ref{lem:UK-connected}.  Proposition~\ref{prop:affine-action} gives
the injectivity of $\Psi_\Lambda$, hence
$\dim_\C\mathfrak h_\Lambda=m$, while
Proposition~\ref{prop:real-part-injective} gives the required real-part
condition.  Thus the displayed quotient is exactly the complex
moment-angle presentation used in
\cite[Section~4.1]{KrutowskiPanov2021}.  Its parameter count is
\[
 \frac{\widehat N-d}{2}
 =\frac{N+m-d}{2}
 =m,
\]
where the last equality is $m=N-d$.

Effectivity and maximality were proved immediately before
Proposition~\ref{prop:canonical-central-foliation}.  A point represented by a vector with every coordinate nonzero has stabilizer
$T^{\widehat N}\cap\exp(\mathfrak h_\Lambda)$. This stabilizer is trivial by
the effectivity argument above.  Proposition~\ref{prop:canonical-central-foliation}
gives the foliation hypotheses.  The complete normal fan in
Theorem~\ref{thm:complete-marked-fan} supplies the polytopal hypothesis
for the direct complex moment-angle results.

These hypotheses concern the quotient, the action subspace $\mathfrak
h_\Lambda$, its real-part map, the canonical foliation, and the projected fan.
They hold at every realizable rank of the projection $B$ onto the
indispensable coordinates.
\end{proof}

\section{The basic Stanley--Reisner ring}
\label{sec:basic-SR-ring}

The basic ring is the face ring of $K$ modulo the linear relations of
the projected rays.  We establish this presentation and the regularity
of its parameters before identifying the curvature differential.

Let
\[
 \mathfrak t_\Lambda
 =
 \R^{\widehat N}/\mathfrak r_\Lambda
\]
and let
\[
 a_j=\varpi_\Lambda(e_j)\in\mathfrak t_\Lambda,
 \qquad 1\leq j\leq N,
\]
be the visible ray vectors of the complete simplicial fan
$\Sigma_\Lambda$.

Set
\[
 S=\C[v_1,\ldots,v_N],
 \qquad
 \C[K]=S/I_K,
 \qquad
 \deg v_j=1.
\]
Here the Stanley--Reisner ideal is
\[
 I_K=\left(\prod_{j\in I}v_j:\ I\subseteq[N],\ I\notin K\right),
\]
the squarefree monomial ideal generated by the nonfaces of $K$.

Since
\[
 \mathfrak h_\Lambda+\overline{\mathfrak h_\Lambda}
 =
 (\mathfrak r_\Lambda)_\C,
\]
the quotient map induces
\[
 \varpi_\Lambda^*:
 (\mathfrak t_\Lambda)_\C^\vee
 \overset{\sim}{\longrightarrow}
 \mathcal L_\Lambda
 =
 \operatorname{Ann}_\C
 (\mathfrak h_\Lambda+\overline{\mathfrak h_\Lambda}).
\]

For $u\in(\mathfrak t_\Lambda)_\C^\vee$, define
\[
 \theta_u=\sum_{j=1}^Nu(a_j)v_j.
\]
Let $J_\Lambda$ be the ideal generated by all $\theta_u$.

\begin{proposition}\label{prop:visible-relation-space}
The map
\[
 (\mathfrak t_\Lambda)_\C^\vee\longrightarrow S_1,
 \qquad
 u\longmapsto\theta_u,
\]
is injective. In particular, the visible linear-relation space has dimension
$d$.
\end{proposition}

\begin{proof}
The rays of the complete fan $\Sigma_\Lambda$ span
$\mathfrak t_\Lambda$. If $\theta_u=0$, then $u(a_j)=0$ for every ray
vector, and hence $u=0$.
\end{proof}

Choose a basis $u_1,\ldots,u_d$ of
$(\mathfrak t_\Lambda)_\C^\vee$ and put
\[
 \theta_\alpha=\sum_{j=1}^Nu_\alpha(a_j)v_j.
\]

\begin{proposition}\label{prop:facet-rank}
For every facet
\[
 \sigma=\{i_1,\ldots,i_d\}\in K,
\]
the matrix
\[
 \bigl(u_\alpha(a_{i_\beta})\bigr)_{\alpha,\beta}
\]
is invertible.
\end{proposition}

\begin{proof}
The vectors $a_{i_1},\ldots,a_{i_d}$ generate a maximal simplicial cone in
the $d$-dimensional fan $\Sigma_\Lambda$. Hence they form a basis of
$\mathfrak t_\Lambda$.
\end{proof}

\begin{theorem}\label{thm:lsop}
The sequence
\[
 \theta_1,\ldots,\theta_d
\]
is a linear system of parameters for $\C[K]$. Therefore
\[
 R_\Lambda
 =
 \frac{\C[K]}{(\theta_1,\ldots,\theta_d)}
 =
 \frac{\C[v_1,\ldots,v_N]}{I_K+J_\Lambda}
\]
is finite-dimensional.
\end{theorem}

\begin{proof}
By Proposition~\ref{prop:facet-rank}, the restrictions of the
$\theta_\alpha$ to every facet generate the polynomial algebra on that
facet. The standard facet criterion for systems of parameters in a
Stanley--Reisner ring therefore applies.

Equivalently, suppose a point $x\in\operatorname{Spec}\C[K]$ is a common
zero of all $\theta_u$. Its nonzero-coordinate support $I(x)$ is a face of
$K$, and
\[
 u\left(\sum_{j\in I(x)}x_ja_j\right)=0
\]
for every $u$. Thus
\[
 \sum_{j\in I(x)}x_ja_j=0.
\]
The ray vectors of a face of a simplicial fan are linearly independent, so
$x=0$. Hence the quotient has Krull dimension zero.
\end{proof}

\begin{theorem}\label{thm:basic-Artinian-Gorenstein}
The sequence $\theta_1,\ldots,\theta_d$ is regular, and $R_\Lambda$ is an
Artinian Gorenstein algebra. If
\[
 h_K(t)=h_0+h_1t+\cdots+h_dt^d
\]
is the $h$-polynomial of $K$, then
\[
 \operatorname{Hilb}(R_\Lambda;t)=h_K(t)
\]
and
\[
 \dim_\C R_\Lambda^p=h_p(K).
\]
Multiplication induces perfect pairings
\[
 R_\Lambda^p\otimes R_\Lambda^{d-p}
 \longrightarrow
 R_\Lambda^d\cong\C.
\]
\end{theorem}

\begin{proof}
The boundary of a simplicial polytope is a homology sphere. Reisner's
criterion implies that $\C[K]$ is Cohen--Macaulay
\cite{Reisner1976}; the homology-sphere Gorenstein* property and the
duality of its Artinian reduction are standard
\cite[Chapter~II]{Stanley1996}. Every homogeneous system of parameters in
a Cohen--Macaulay ring is regular, and its Artinian reduction has Hilbert
series equal to the $h$-polynomial. The Gorenstein property gives the
perfect pairings.
\end{proof}

\begin{proposition}\label{prop:degree-one-basic-ring}
One has
\[
 \dim_\C R_\Lambda^1=N-d=m.
\]
Equivalently, in Dolbeault bidegrees,
\[
 \dim_\C R_\Lambda^{1,1}=m.
\]
\end{proposition}

\begin{proof}
The Stanley--Reisner ideal has no linear terms, so
$\dim\C[K]_1=N$. Proposition~\ref{prop:visible-relation-space} gives
$\dim(J_\Lambda)_1=d$. Therefore
\[
 \dim R_\Lambda^1=N-d=m.
\]
Equivalently, $h_1(K)=N-d$ and
Theorem~\ref{thm:basic-Artinian-Gorenstein} applies.
\end{proof}

\begin{remark}
The ideal $J_\Lambda$ is canonical, although a list of generators depends
on a basis of $(\mathfrak t_\Lambda)_\C^\vee$. Completeness of the fan also gives injectivity of visible projection
on $\mathcal L_\Lambda$.
\end{remark}

\section{The basic Dolbeault ring and the abstract Hirsch model}
\label{sec:basic-Dolbeault-Hirsch}

Basic Dolbeault cohomology is concentrated in bidegrees $(a,a)$.
The finite model for ordinary Dolbeault cohomology adjoins leafwise
one-forms, whose differential is their basic curvature class.

Write $(x_1,\ldots,x_{\widehat N})=(v_1,\ldots,v_N,u_1,\ldots,u_m)$
for the visible and ghost polynomial variables.  Let $\widehat J_\Lambda$
be the ideal generated by
$\sum_{i=1}^{\widehat N}\ell_i x_i$ for
$\ell=(\ell_i)\in\mathcal L_\Lambda$, in the polynomial ring on all
visible and ghost labels.

\begin{theorem}[Basic Dolbeault ring]
\label{thm:basic-Dolbeault-ring}
There is an isomorphism
\[
 H_{\mathcal F_\Lambda}^{*,*}(\XL)
 \cong
 R_\Lambda
 =
 \frac{\C[v_1,\ldots,v_N]}{I_K+J_\Lambda}.
\]
Every generator has bidegree $(1,1)$, and the basic Dolbeault cohomology is
zero off the diagonal.

Equivalently, adjoining the ghost variables gives the presentation
\[
 H_{\mathcal F_\Lambda}^{*,*}(\XL)
 \cong
 \frac{\C[v_1,\ldots,v_N,u_1,\ldots,u_m]}
 {I_{\widehat K}+\widehat J_\Lambda},
\]
because every ghost variable belongs to $I_{\widehat K}$.
\end{theorem}

\begin{proof}
The hypotheses of the cited model theorems hold by
Proposition~\ref{prop:imported-model-legality}.
Apply the basic-cohomology presentation for maximal torus actions
\cite[Theorem~5.10]{IshidaKrutowskiPanov2022} and its basic Dolbeault
refinement \cite[Theorem~5.1]{KrutowskiPanov2021} to the complete marked
fan of Theorem~\ref{thm:complete-marked-fan}.  Equivalently, apply the
direct complex moment-angle statement
\cite[Theorem~4.12]{KrutowskiPanov2021}.  These results give the
visible-ray ring $\C[v_1,\ldots,v_N]/(I_K+J_\Lambda)$ directly.  In real basic cohomology, each $v_j$ has cohomological degree two; after complexification, the Dolbeault refinement places it in
bidegree $(1,1)$, which is internal degree one in our convention.  Its
linear ideal is determined by the annihilator of
$(\mathfrak r_\Lambda)_\C$, which is $\mathcal L_\Lambda$.  To obtain the second presentation, adjoin the ghost variables and
use the fact that every ghost singleton is a nonface, so those variables
vanish in the Stanley--Reisner quotient.
\end{proof}

\begin{theorem}[Abstract Dolbeault--Hirsch model]
\label{thm:abstract-Hirsch-model}
There are complex vector spaces
\[
 W_\Lambda^{1,0},\qquad W_\Lambda^{0,1},
 \qquad
 \dim_\C W_\Lambda^{1,0}
 =
 \dim_\C W_\Lambda^{0,1}
 =
 m,
\]
and a finite differential bigraded algebra
\[
 \mathcal B_\Lambda
 =
 R_\Lambda\otimes
 \Lambda(W_\Lambda^{1,0}\oplus W_\Lambda^{0,1})
\]
which is connected to the Dolbeault complex of $\XL$ by a zigzag of
quasi-isomorphisms of differential bigraded algebras.
Moreover,
\[
 \dim_\C\mathcal B_\Lambda
 =\left(\sum_{a=0}^{d}h_a(K)\right)2^{2m},
 \qquad
 \mathcal B_\Lambda^{p,q}=0
 \quad\text{unless}\quad 0\leq p,q\leq N.
\]
Its differential satisfies
\[
 \bar dR_\Lambda=0,\qquad
 \bar dW_\Lambda^{1,0}\subseteq R_\Lambda^{1,1},
 \qquad
 \bar dW_\Lambda^{0,1}=0.
\]
\end{theorem}

\begin{proof}
Proposition~\ref{prop:imported-model-legality} verifies the hypotheses of
the general Dolbeault-model theorem of Ishida--Kasuya
\cite[Theorem~1.2, equivalently Theorem~4.13]{IshidaKasuya2019}.
Independently, the complex moment-angle presentation permits the direct
application of \cite[Theorem~6.1]{KrutowskiPanov2021}.  In the notation
of that theorem, each of $W^{1,0}$ and $W^{0,1}$ has dimension
$(\widehat N-d)/2=m$.  The direct theorem gives
$\bar dW^{0,1}=0$.  Alternatively, the general theorem gives
\[
 \bar dW^{0,1}\subseteq
 H_{\mathcal F_\Lambda}^{0,2}(\XL).
\]
The latter group vanishes by
Theorem~\ref{thm:basic-Dolbeault-ring}.
\end{proof}

\begin{remark}
The finite model is defined up to weak equivalence. Choosing invariant
connection forms identifies its generator spaces with the dual leaf algebras.  The concrete
identification
\[
 W_\Lambda^{1,0}
 \cong
 \operatorname{Ann}(\mathfrak h_\Lambda)/\mathcal L_\Lambda
\]
and the formula
\[
 [\mu]\longmapsto
 \left[\sum_{j=1}^N\mu_jv_j\right]
\]
for the transgression are established in Sections~\ref{sec:Cartan-generators}
and~\ref{sec:curvature-map}.
\end{remark}

\section{The Cartan generator spaces}
\label{sec:Cartan-generators}

The leafwise generators are dual to the holomorphic and
antiholomorphic leaf algebras.  We express these dual spaces as
quotients of coordinate covectors, so that their curvature can be
computed from $A$ and $B$.

Recall
\[
 E_\Lambda=\C^{\widehat N},\qquad
 \mathfrak h_\Lambda=\operatorname{im}\Psi_\Lambda\subseteq E_\Lambda,
\]
and
\[
 \mathcal L_\Lambda
 =
 \operatorname{Ann}_{E_\Lambda^\vee}
 (\mathfrak h_\Lambda+\overline{\mathfrak h_\Lambda}).
\]
Recall that
\[
 (\mathfrak r_\Lambda)_\C
 =
 \mathfrak h_\Lambda\oplus\overline{\mathfrak h_\Lambda}.
\]

Since $\operatorname{Re}|_{\mathfrak h_\Lambda}$ is a real-linear
isomorphism onto $\mathfrak r_\Lambda$, the leaf algebra has complex
structure
\[
 J_{\mathcal F}(\operatorname{Re}z)=\operatorname{Im}z
 =\operatorname{Re}(-iz),\qquad z\in\mathfrak h_\Lambda.
\]
With this convention, the holomorphic and antiholomorphic tangent algebras
of the canonical foliation are
\[
 \mathfrak f_\Lambda^{1,0}
 =
 (\mathfrak r_\Lambda)_\C/\mathfrak h_\Lambda
 \cong\overline{\mathfrak h_\Lambda},
\]
and
\[
 \mathfrak f_\Lambda^{0,1}
 =
 (\mathfrak r_\Lambda)_\C/\overline{\mathfrak h_\Lambda}
 \cong\mathfrak h_\Lambda.
\]

\begin{definition}
Define
\[
 W_\Lambda^{1,0}
 =
 \frac{\operatorname{Ann}(\mathfrak h_\Lambda)}
 {\mathcal L_\Lambda},
 \qquad
 W_\Lambda^{0,1}
 =
 \frac{\operatorname{Ann}(\overline{\mathfrak h_\Lambda})}
 {\mathcal L_\Lambda}.
\]
\end{definition}

\begin{theorem}\label{thm:Cartan-spaces}
Restriction of covectors gives exact sequences
\[
 0\longrightarrow\mathcal L_\Lambda
 \longrightarrow\operatorname{Ann}(\mathfrak h_\Lambda)
 \longrightarrow\overline{\mathfrak h_\Lambda}^{\,\vee}
 \longrightarrow0
\]
and
\[
 0\longrightarrow\mathcal L_\Lambda
 \longrightarrow\operatorname{Ann}(\overline{\mathfrak h_\Lambda})
 \longrightarrow\mathfrak h_\Lambda^\vee
 \longrightarrow0.
\]
Consequently,
\[
 W_\Lambda^{1,0}
 \cong
 \overline{\mathfrak h_\Lambda}^{\,\vee}
 \cong
 (\mathfrak f_\Lambda^{1,0})^\vee,
\]
and
\[
 W_\Lambda^{0,1}
 \cong
 \mathfrak h_\Lambda^\vee
 \cong
 (\mathfrak f_\Lambda^{0,1})^\vee.
\]
In particular,
\[
 \dim_\C W_\Lambda^{1,0}
 =
 \dim_\C W_\Lambda^{0,1}
 =
 m.
\]
\end{theorem}

\begin{proof}
A covector in $\operatorname{Ann}(\mathfrak h_\Lambda)$ restricts to a
functional on $\overline{\mathfrak h_\Lambda}$. The kernel consists of
the covectors annihilating both summands, namely $\mathcal L_\Lambda$.
Every functional on $\overline{\mathfrak h_\Lambda}$ extends to
the direct sum
$\mathfrak h_\Lambda\oplus\overline{\mathfrak h_\Lambda}$ after being
set equal to zero on $\mathfrak h_\Lambda$, and then extends to
$E_\Lambda$.
The second sequence is the conjugate of the first.
\end{proof}

\begin{proposition}[Deck-dual identification at full rank]
\label{prop:deck-Cartan-identification}
Assume that $B$ is invertible and set
\[
 V_\Gamma=\GammaL\otimes_{\Z}\C.
\]
The inclusion $\GammaL\subset\C^m$ and the action map
$\Psi_\Lambda:\C^m\to\mathfrak h_\Lambda$ induce canonical
isomorphisms
\[
 V_\Gamma\cong\C^m\cong\mathfrak h_\Lambda.
\]
Consequently,
\[
 W_\Lambda^{0,1}
 \cong
 V_\Gamma^\vee
 \cong
 \GammaLCdual.
\]
\end{proposition}

\begin{proof}
The vectors $B^{-1}(2\pi i e_1),\ldots,B^{-1}(2\pi i e_m)$ form a
complex basis of $\C^m$, so complexification of the lattice inclusion is
an isomorphism.  The map $\Psi_\Lambda$ is an isomorphism from $\C^m$
onto its image $\mathfrak h_\Lambda$.  Dualize and use
Theorem~\ref{thm:Cartan-spaces}.
\end{proof}

Writing a covector as
\[
 \mu=(u,\eta)
 \in(\C^N)^\vee\oplus(\C^m)^\vee,
\]
one has
\[
 \operatorname{Ann}(\mathfrak h_\Lambda)
 =
 \{(u,\eta):A^\vee u+B^\vee\eta=0\},
\]
and
\[
 \mathcal L_\Lambda
 =
 \left\{
 (u,\eta):
 \begin{array}{l}
 A^\vee u+B^\vee\eta=0,\\
 \overline A^{\,\vee}u+\overline B^{\,\vee}\eta=0
 \end{array}
 \right\}.
\]

\begin{lemma}\label{lem:pure-ghost-generators}
The visible projection
\[
 \operatorname{Ann}(\mathfrak h_\Lambda)\longrightarrow(\C^N)^\vee
\]
has kernel
\[
 \{0\}\oplus\ker(B^\vee).
\]
Since the visible projection of $\mathcal L_\Lambda$ is injective, there is
a canonical embedding
\[
 \ker(B^\vee)=(\operatorname{coker}B)^\vee
 \hookrightarrow W_\Lambda^{1,0}.
\]
\end{lemma}

\begin{proof}
The condition $(0,\eta)\in\operatorname{Ann}(\mathfrak h_\Lambda)$ is
exactly $B^\vee\eta=0$. The second assertion follows from injectivity of visible projection on the basic relation space.
\end{proof}

\subsection{Realization by invariant connection forms}

We now realize the complex structure $J_{\mathcal F}$ defined above by
invariant connection forms on $\XL$.

\begin{lemma}[Type-compatible invariant connection]
\label{lem:type-compatible-connection}
The complex structure induced on the tangent distribution of
$\mathcal F_\Lambda$ is $J_{\mathcal F}$.  Moreover,
\[
 (\mathfrak r_\Lambda,J_{\mathcal F})^{1,0}
 =\overline{\mathfrak h_\Lambda},
 \qquad
 (\mathfrak r_\Lambda,J_{\mathcal F})^{0,1}
 =\mathfrak h_\Lambda.
\]
There is a $T^{\widehat N}$-invariant connection
\[
 \omega_\nabla\in\Omega^1(\XL;\mathfrak r_\Lambda)
\]
for the locally free infinitesimal $\mathfrak r_\Lambda$-action such that
\[
 \iota_{X_v}\omega_\nabla=v,
 \qquad
 \omega_\nabla\circ J=J_{\mathcal F}\circ\omega_\nabla.
\]
Consequently, after complexification,
\[
 \omega_\nabla^{1,0}
 \in\Omega^{1,0}(\XL;\overline{\mathfrak h_\Lambda}),
 \qquad
 \omega_\nabla^{0,1}
 \in\Omega^{0,1}(\XL;\mathfrak h_\Lambda).
\]
\end{lemma}

\begin{proof}
For $v\in\mathfrak r_\Lambda$, the unique element of
$\mathfrak h_\Lambda$ with real part $v$ is
$z=v+iJ_{\mathcal F}v$.  With the period-one exponential convention, the
infinitesimal compact-orbit vector $X_v$ is represented in
$\C^{\widehat N}/\mathfrak h_\Lambda$ by $[2\pi iv]$.
Since $v+iJ_{\mathcal F}v\in\mathfrak h_\Lambda$,
\[
 J[2\pi iv]=[-2\pi v]=[2\pi iJ_{\mathcal F}v].
\]
Thus $JX_v=X_{J_{\mathcal F}v}$.  Directly from the definition,
$\mathfrak h_\Lambda$ and $\overline{\mathfrak h_\Lambda}$ are the
$(-i)$- and $(+i)$-eigenspaces, respectively, of the complexification of
$J_{\mathcal F}$.

Average a Hermitian metric over $T^{\widehat N}$.  The tangent
distribution of $\mathcal F_\Lambda$ is $J$-stable, and therefore so is
its orthogonal complement.  The orthogonal projection onto the leaf
distribution commutes with $J$.  Composing it with the inverse of
\[
 \mathfrak r_\Lambda\longrightarrow T\mathcal F_\Lambda,
 \qquad v\longmapsto X_v,
\]
gives the asserted connection and its type decomposition.
\end{proof}

For
\[
 \mu\in\operatorname{Ann}(\mathfrak h_\Lambda),
\]
put
\[
 \lambda_\mu
 =
 \mu|_{(\mathfrak r_\Lambda)_\C}.
\]
This covector vanishes on $\mathfrak h_\Lambda$, so it is the extension
by zero of
\[
 \mu|_{\overline{\mathfrak h_\Lambda}}
 \in\overline{\mathfrak h_\Lambda}^{\,\vee}.
\]
Define
\[
 \alpha_{\nabla,\mu}
 =
 \langle\lambda_\mu,\omega_{\nabla,\C}\rangle
 =
 \left\langle
  \mu|_{\overline{\mathfrak h_\Lambda}},
  \omega_\nabla^{1,0}
 \right\rangle
 \in\Omega^{1,0}(\XL).
\]
Let $W_{\nabla,\Lambda}^{1,0}$ be the span of these forms, and define
$W_{\nabla,\Lambda}^{0,1}$ by conjugation.

\begin{theorem}\label{thm:connection-Cartan-identification}
Evaluation on leaf directions gives isomorphisms
\[
 W_{\nabla,\Lambda}^{1,0}
 \cong
 (\mathfrak f_\Lambda^{1,0})^\vee
 \cong
 W_\Lambda^{1,0},
\]
and
\[
 W_{\nabla,\Lambda}^{0,1}
 \cong
 (\mathfrak f_\Lambda^{0,1})^\vee
 \cong
 W_\Lambda^{0,1}.
\]
These spaces realize the generators of the Ishida--Kasuya Hirsch model.
Under these identifications its holomorphic generator differential is
\[
 \alpha_{\nabla,\mu}
 \longmapsto
 [\bar\partial\alpha_{\nabla,\mu}]_{\mathrm{bas}}.
\]
\end{theorem}

\begin{proof}
The form $\alpha_{\nabla,\mu}$ depends on the class of $\mu$ modulo
$\mathcal L_\Lambda$, since this relation space annihilates
$(\mathfrak r_\Lambda)_\C$.  For
$\zeta\in\overline{\mathfrak h_\Lambda}$, let
$X_\zeta^{1,0}$ be the normalized complexified leaf vector characterized
by
\[
 \omega_{\nabla,\C}(X_\zeta^{1,0})=\zeta.
\]
Then
\[
 \alpha_{\nabla,\mu}(X_\zeta^{1,0})=\mu(\zeta),
\]
so evaluation recovers
$\mu|_{\overline{\mathfrak h_\Lambda}}$.  The component forms are
invariant and have constant contractions with leaf vector fields.
For every $v\in(\mathfrak r_\Lambda)_\C$, Cartan's formula gives
\[
 \iota_{X_v}d\alpha_{\nabla,\mu}=0,
 \qquad
 \mathcal L_{X_v}d\alpha_{\nabla,\mu}=0,
\]
so their differentials are basic.  The resulting real connection
components have the nondegenerate contraction pairing required in
\cite[Proposition~4.9]{IshidaKasuya2019}.  After taking types, the
explicit formulation of
\cite[Theorem~1.2, equivalently Theorem~4.13]{IshidaKasuya2019}
identifies these spaces with the generator spaces in the
Dolbeault--Hirsch model and sends a generator $\alpha$ to the basic
Dolbeault class $[\bar\partial\alpha]_{\mathrm{bas}}$.
\end{proof}

\begin{proposition}[Connection independence]
For two type-compatible invariant connections, corresponding generator
forms differ by basic forms. Therefore the basic class
\[
 [\bar\partial\alpha_{\nabla,\mu}]_{\mathrm{bas}}
\]
is independent of the connection.
\end{proposition}

\begin{proof}
The difference of two connection forms vanishes on the leaf distribution
and is invariant, hence basic. Applying $\bar\partial_{\mathrm{bas}}$ changes the
curvature by an exact basic form.
\end{proof}

\subsection{Curvature by visible projection}

For
\[
 \mu=(u,\eta)\in\operatorname{Ann}(\mathfrak h_\Lambda),
\]
define
\[
 c_\Lambda^{\mathrm{alg}}([\mu])
 =
 \left[\sum_{j=1}^Nu_jv_j\right]
 \in R_\Lambda^{1,1}.
\]

\begin{proposition}
The map
\[
 c_\Lambda^{\mathrm{alg}}:
 W_\Lambda^{1,0}\longrightarrow R_\Lambda^{1,1}
\]
is well defined.
\end{proposition}

\begin{proof}
Changing $\mu$ by an element of $\mathcal L_\Lambda$ changes its visible
coefficient vector by one of the linear relations generating
$J_\Lambda$.
\end{proof}

\section{The explicit curvature map}
\label{sec:curvature-map}

A coordinate covector defines a scalar connection form.  Its
curvature class is the corresponding coordinate linear form in basic
cohomology.  Restricting this characteristic map to holomorphic
leafwise covectors gives $c_\Lambda$.  We fix the normalization first.

Put
\[
 \widehat N=N+m=n-1,\qquad
 E_\Lambda=\C^{\widehat N},\qquad
 \mathfrak r_\Lambda=\operatorname{Re}(\mathfrak h_\Lambda).
\]
We use the coordinate-torus exponential fixed in
Section~\ref{sec:marked-fan-foliation},
\[
 \exp_T(x)=(e^{2\pi ix_1},\ldots,e^{2\pi ix_{\widehat N}}),
\]
and the Cartan differential $d_{\mathrm C}=d-\iota$.  This is the
Cartan model of equivariant cohomology: invariant differential forms are
adjoined to polynomial functions on the torus Lie algebra, and a connection
sends the polynomial generators to curvature classes through the Cartan
characteristic homomorphism.  Let
$\varepsilon_i\in E_\Lambda^\vee$ be the positive $i$th coordinate
polynomial.  We normalize
\[
 x_i\in H^2(BT^{\widehat N};\C)
\]
as the classes of the $\varepsilon_i$, and use the same symbols for their
images in equivariant cohomology.  Equivalently, the corresponding
character satisfies
\[
 \chi_i(\exp_Tx)=e^{2\pi ix_i}.
\]
If the principal connection is written in the conventional
$i\R$-valued normalization, its curvature is
$2\pi i\,d\langle\varepsilon_i,\omega_\nabla\rangle$; division by
$2\pi i$ gives the real period-one representative used here.  Hence the
Cartan homomorphism sends
$\varepsilon_i|_{\mathfrak r_\Lambda}$ to
$d\langle\varepsilon_i|_{\mathfrak r_\Lambda},\omega_\nabla\rangle$,
with coefficient $+1$ in this normalization.  The visible images of the
$x_i$ in basic Dolbeault cohomology are denoted by $v_i$.

The dual of
\[
 0\longrightarrow\mathfrak r_\Lambda
 \longrightarrow\R^{\widehat N}
 \overset{\varpi_\Lambda}{\longrightarrow}
 \mathfrak t_\Lambda
 \longrightarrow0
\]
is
\[
 0\longrightarrow
 (\mathfrak t_\Lambda)_\C^\vee
 \overset{\varpi_\Lambda^\vee}{\longrightarrow}
 E_\Lambda^\vee
 \longrightarrow
 (\mathfrak r_\Lambda)_\C^\vee
 \longrightarrow0.
\]
Its left-hand image is
\[
 \mathcal L_\Lambda
 =
 \operatorname{Ann}
 (\mathfrak h_\Lambda+\overline{\mathfrak h_\Lambda}).
\]

\subsection{The Cartan characteristic map}

The coordinate-equivariant cohomology is
\[
 H_{T^{\widehat N}}^*(\mathcal Z_{\widehat K};\C)
 \cong\C[\widehat K],
\]
and the ambient coordinate map is
\[
 \chi_{\mathrm{eq}}:E_\Lambda^\vee\longrightarrow\C[\widehat K]_1,
 \qquad
 (\mu_1,\ldots,\mu_{\widehat N})
 \longmapsto
 \sum_{i=1}^{\widehat N}\mu_ix_i.
\]

Basic cohomology is the Cartan base change
\[
 H_{\mathcal F_\Lambda}^*(\XL)
 \cong
 \C\otimes_{S(\mathfrak t_{\Lambda,\C}^\vee)}
 H_{T^{\widehat N}}^*(\mathcal Z_{\widehat K}),
\]
so
\[
 H_{\mathcal F_\Lambda}^{*,*}(\XL)
 \cong
 \frac{\C[\widehat K]}
 {\varpi_\Lambda^\vee(\mathfrak t_{\Lambda,\C}^\vee)}
 \cong
 R_\Lambda.
\]

Let $\rho_\nabla$ be the morphism obtained by restricting the
$T^{\widehat N}$-Cartan model to $\mathfrak r_\Lambda$ and then applying
the Cartan homomorphism determined by the connection of
Lemma~\ref{lem:type-compatible-connection}.  The construction is the
one used in the Cartan/Weil comparison of
\cite[Section~2.2, Lemmas~3.1--3.2, and the proof of
Theorem~3.4]{IshidaKrutowskiPanov2022}.

\begin{proposition}[Chern--Weil characteristic square]
\label{prop:CW-coordinate-formula}
For
\[
 \widetilde\lambda
 =
 \sum_{i=1}^{\widehat N}
 \widetilde\lambda_i\varepsilon_i
 \in E_\Lambda^\vee,
\]
one has
\[
 \rho_\nabla\left(
  \sum_{i=1}^{\widehat N}
  \widetilde\lambda_ix_i
 \right)
 =
 \left[
  d\left\langle
   \widetilde\lambda|_{(\mathfrak r_\Lambda)_\C},
   \omega_{\nabla,\C}
  \right\rangle
 \right]_{\mathrm{bas}}.
\]
Equivalently, the following square commutes:
\[
\begin{array}{ccc}
 E_\Lambda^\vee
 &\longrightarrow&H_{T^{\widehat N}}^2(\XL;\C)\\[0.3em]
 \big\downarrow\mathrm{res}&&\big\downarrow\rho_\nabla\\[0.3em]
 (\mathfrak r_\Lambda)_\C^\vee
 &\overset{\mathrm{CW}_\nabla}{\longrightarrow}&
 H_{\mathcal F_\Lambda}^2(\XL;\C),
\end{array}
\]
where the upper map sends $\widetilde\lambda$ to
$\sum_i\widetilde\lambda_ix_i$.
\end{proposition}

\begin{proof}
Restriction of Cartan models sends the polynomial
$\widetilde\lambda\in E_\Lambda^\vee$ to its restriction to
$(\mathfrak r_\Lambda)_\C$.  Since the leaf algebra is abelian, the
Cartan homomorphism associated with $\omega_\nabla$ sends a degree-two
polynomial $\lambda\in(\mathfrak r_\Lambda)_\C^\vee$ to
\[
 \lambda(d\omega_{\nabla,\C})
 =
 d\langle\lambda,\omega_{\nabla,\C}\rangle.
\]
Naturality of restriction and of the Cartan homomorphism gives the
commutative square.  The $S(E_\Lambda^\vee)$-linear equivariant formality
comparison sends $\widetilde\lambda$ to its coordinate linear form
$\sum_i\widetilde\lambda_ix_i$.  A different lift of a covector on
$(\mathfrak r_\Lambda)_\C$ differs by an element of
\[
 \mathcal L_\Lambda
 =
 \varpi_\Lambda^\vee
 ((\mathfrak t_\Lambda)_\C^\vee),
\]
whose coordinate linear form vanishes after the Cartan base change.
\end{proof}

\subsection{Holomorphic transgression}

Let
\[
 [\mu]\in W_\Lambda^{1,0},
 \qquad
 \mu=(u,\eta)\in\operatorname{Ann}(\mathfrak h_\Lambda).
\]
The covector $\lambda_\mu=\mu|_{(\mathfrak r_\Lambda)_\C}$ is the
canonical extension by zero of the corresponding holomorphic leafwise
covector, and its scalar connection form is $\alpha_{\nabla,\mu}$.

\begin{lemma}[Degree-two basic Hodge edge]
\label{lem:degree-two-basic-edge}
The basic Fr\"olicher edge morphism is an isomorphism
\[
 \epsilon^{1,1}:
 H_{\mathcal F_\Lambda}^2(\XL;\C)
 \overset{\sim}{\longrightarrow}
 H_{\mathcal F_\Lambda}^{1,1}(\XL).
\]
If $\alpha\in\Omega^{1,0}(\XL)$ and $d\alpha$ is basic, then
\[
 \epsilon^{1,1}([d\alpha]_{\mathrm{bas}})
 =
 [\bar\partial\alpha]_{\mathrm{bas}}.
\]
\end{lemma}

\begin{proof}
The basic $\partial\bar\partial$ lemma for the transversely K\"ahler foliation implies that
the basic Fr\"olicher spectral sequence degenerates at $E_1$
\cite[Theorem~4.12]{IshidaKasuya2019}, and
Theorem~\ref{thm:basic-Dolbeault-ring} gives
\[
 H_{\mathcal F_\Lambda}^{2,0}(\XL)
 =
 H_{\mathcal F_\Lambda}^{0,2}(\XL)
 =0.
\]
Thus total basic degree two has only the $(1,1)$ graded piece.  Write
$d\alpha=\partial\alpha+\bar\partial\alpha$.  Modulo the second step of
the Hodge filtration, the class of $d\alpha$ is represented by its
$(1,1)$ component $\bar\partial\alpha$.  This is exactly the basic
Fr\"olicher edge morphism.
\end{proof}

\begin{theorem}[Explicit Chern--Weil/Dolbeault transgression]
\label{thm:explicit-curvature}
For
\[
 [\mu]=[(u,\eta)]\in W_\Lambda^{1,0},
\]
one has
\[
 c_\Lambda([\mu])
 =
 \left[
 \sum_{j=1}^Nu_jv_j
 \right]
 \in R_\Lambda^{1,1}.
\]
This identity holds at every geometrically realizable rank of $B$.
\end{theorem}

\begin{proof}
Apply Proposition~\ref{prop:CW-coordinate-formula} to
$\mu\in E_\Lambda^\vee$.  Its restriction to
$(\mathfrak r_\Lambda)_\C$ is $\lambda_\mu$, and the corresponding scalar
connection form is $\alpha_{\nabla,\mu}\in\Omega^{1,0}(\XL)$. Hence
\[
 [d\alpha_{\nabla,\mu}]_{\mathrm{bas}}
 =
 \left[
  \sum_{i=1}^{\widehat N}\mu_ix_i
 \right].
\]
Apply Lemma~\ref{lem:degree-two-basic-edge}.  The ghost coordinate classes
vanish in $\C[\widehat K]/\widehat J_\Lambda$, while the visible classes
are $v_1,\ldots,v_N$.  Therefore
\[
 [\bar\partial\alpha_{\nabla,\mu}]_{\mathrm{bas}}
 =
 \left[
  \sum_{j=1}^Nu_jv_j
 \right],
\]
which is the transferred Ishida--Kasuya differential by
Theorem~\ref{thm:connection-Cartan-identification}.
\end{proof}

\begin{remark}
When the canonical foliation is a holomorphic principal torus bundle over a smooth toric base, this is its first Chern class map.  The coordinate-projection
description agrees with the independent computation of
Panov--Ustinovsky, Theorem~5.4 and Lemma~5.5
\cite{PanovUstinovsky2012}; the argument above establishes the formula at every realizable rank.  On the regular LVMB
torus-bundle locus, compare also Thiella's characteristic-class formula
and bigraded model
\cite[Theorem~1 and Proposition~2.8]{Thiella2026}.
\end{remark}

\begin{corollary}
The map $c_\Lambda$ is induced by visible projection:
\[
 \operatorname{Ann}(\mathfrak h_\Lambda)
 \longrightarrow(\C^N)^\vee,
 \qquad
 (u,\eta)\longmapsto u,
\]
followed by the quotient
\[
 (\C^N)^\vee\longrightarrow
 R_\Lambda^{1,1}
 =
 (\C^N)^\vee/(J_\Lambda)_1.
\]
\end{corollary}

\subsection{The explicit finite model}

\begin{theorem}[Global Dolbeault--Cartan model]
\label{thm:global-Dolbeault-Cartan}
Let
\[
 \mathcal D_\Lambda
 =
 R_\Lambda\otimes
 \Lambda(W_\Lambda^{1,0}\oplus W_\Lambda^{0,1})
\]
with
\[
 \deg v_j=(1,1),\qquad
 \deg W_\Lambda^{1,0}=(1,0),\qquad
 \deg W_\Lambda^{0,1}=(0,1).
\]
Define
\[
 \bar dR_\Lambda=0,\qquad
 \bar dW_\Lambda^{0,1}=0,
\]
and
\[
 \bar d[(u,\eta)]
 =
 \left[
 \sum_{j=1}^Nu_jv_j
 \right].
\]
Then
\[
 \mathcal D_\Lambda
 \simeq
 (\Omega^{*,*}(\XL),\bar\partial)
\]
as differential bigraded algebras.
\end{theorem}

\begin{proof}
The abstract weak equivalence is Theorem~\ref{thm:abstract-Hirsch-model}.
The generator spaces were identified in
Theorem~\ref{thm:Cartan-spaces}, and
Theorem~\ref{thm:explicit-curvature} identifies their differential.
\end{proof}

\section{The curvature exact sequence}
\label{sec:curvature-exact-sequence}

The visible-projection formula for $c_\Lambda$ reduces its kernel
and cokernel to linear algebra.  We first restrict visible projection to the annihilator of the action subspace, then pass to the quotient by the basic linear relations.

Write
\[
 \mathsf U=\C^m,\qquad
 \mathsf V=\C^N,\qquad
 \mathsf G=\C^m,
\]
and
\[
 \Psi_\Lambda=
 \begin{pmatrix}A\\B\end{pmatrix}:
 \mathsf U\longrightarrow\mathsf V\oplus\mathsf G.
\]
Write
\[
 K_B=\ker B,\qquad
 C_B=\operatorname{coker}B,\qquad
 \delta=m-\operatorname{rank}_\C B.
\]

Recall the annihilator description from
Section~\ref{sec:Cartan-generators}: a covector $(u,\eta)$ annihilates
$\mathfrak h_\Lambda$ precisely when $A^\vee u+B^\vee\eta=0$.  Let
\[
 \pi_{\mathrm{vis}}:
 \operatorname{Ann}(\mathfrak h_\Lambda)\longrightarrow\mathsf V^\vee,
 \qquad
 (u,\eta)\longmapsto u.
\]

Define
\[
 \beta_0:\mathsf V^\vee\longrightarrow K_B^\vee,
 \qquad
 \beta_0(u)=(A^\vee u)|_{K_B}.
\]

\begin{theorem}[The annihilator exact sequence]
\label{thm:precurvature-sequence}
There is an exact sequence
\[
 0\longrightarrow C_B^\vee
 \longrightarrow\operatorname{Ann}(\mathfrak h_\Lambda)
 \overset{\pi_{\mathrm{vis}}}{\longrightarrow}
 \mathsf V^\vee
 \overset{\beta_0}{\longrightarrow}
 K_B^\vee
 \longrightarrow0,
\]
where $C_B^\vee$ is identified with $\ker(B^\vee)$ and maps by
$\eta\mapsto(0,\eta)$.
\end{theorem}

\begin{proof}
Lemma~\ref{lem:pure-ghost-generators} identifies the kernel of
$\pi_{\mathrm{vis}}$ with $C_B^\vee=\ker B^\vee$.

If $(u,\eta)$ annihilates $\mathfrak h_\Lambda$, then
$A^\vee u=-B^\vee\eta$, which vanishes on $\ker B$. Thus
$\operatorname{im}\pi_{\mathrm{vis}}\subseteq\ker\beta_0$.

Conversely, if $(A^\vee u)|_{\ker B}=0$, then
\[
 A^\vee u\in\operatorname{Ann}(\ker B)=\operatorname{im}B^\vee.
\]
Choose $\eta$ with $B^\vee\eta=-A^\vee u$. Then
$(u,\eta)\in\operatorname{Ann}(\mathfrak h_\Lambda)$.

Finally, $A|_{\ker B}$ is injective: if $Ax=Bx=0$, then
$\Psi_\Lambda x=0$, so $x=0$. Hence its dual $\beta_0$ is surjective.
\end{proof}

\begin{lemma}[Visible-relation quotient]
\label{lem:visible-relation-quotient}
The visible projection restricts to an isomorphism
\[
 \pi_{\mathrm{vis}}:\mathcal L_\Lambda
 \overset{\sim}{\longrightarrow}(J_\Lambda)_1.
\]
Moreover, $(J_\Lambda)_1\subseteq\ker\beta_0$.
\end{lemma}

\begin{proof}
Under the canonical identification
\[
 (\mathfrak t_\Lambda)_\C^\vee
 \overset{\varpi_\Lambda^\vee}{\cong}
 \mathcal L_\Lambda,
\]
visible projection sends a covector $q$ to the coefficient vector of its
linear basic relation
\[
 \theta_q=\sum_{j=1}^Nq(a_j)v_j.
\]
Its image is therefore $(J_\Lambda)_1$.  Its injectivity is Proposition~\ref{prop:visible-relation-space}.  If
$\ell=(u,\eta)\in\mathcal L_\Lambda$, then
\[
 A^\vee u+B^\vee\eta=0,
\]
so $(A^\vee u)|_{K_B}=0$.  Hence $\beta_0(u)=0$.
\end{proof}

Thus $\beta_0$ descends to
\[
 \beta_\Lambda:R_\Lambda^{1,1}\longrightarrow K_B^\vee,
 \qquad
 \beta_\Lambda([u])=(A^\vee u)|_{K_B}.
\]

\begin{theorem}[Curvature exact sequence]
\label{thm:curvature-exact-sequence}
There is a canonical exact sequence
\[
 0\longrightarrow C_B^\vee
 \longrightarrow W_\Lambda^{1,0}
 \overset{c_\Lambda}{\longrightarrow}
 R_\Lambda^{1,1}
 \overset{\beta_\Lambda}{\longrightarrow}
 K_B^\vee
 \longrightarrow0.
\]
The maps are
\[
 \eta\longmapsto[(0,\eta)],
 \qquad
 [(u,\eta)]\longmapsto[u],
 \qquad
 [u]\longmapsto(A^\vee u)|_{K_B}.
\]
\end{theorem}

\begin{proof}
By Theorem~\ref{thm:explicit-curvature}, the middle arrow induced by
$\pi_{\mathrm{vis}}$ is the Hirsch differential $c_\Lambda$.
Lemma~\ref{lem:visible-relation-quotient} identifies the acyclic
two-term subcomplex
\[
 0\longrightarrow0\longrightarrow
 \mathcal L_\Lambda
 \overset{\sim}{\longrightarrow}
 (J_\Lambda)_1\longrightarrow0\longrightarrow0
\]
inside the annihilator exact sequence.  Quotienting by it gives the
commutative diagram
\[
\begin{array}{ccccccccccc}
0&\longrightarrow&C_B^\vee&\longrightarrow&
\operatorname{Ann}(\mathfrak h_\Lambda)&\longrightarrow&
\mathsf V^\vee&\longrightarrow&K_B^\vee\longrightarrow0\\
&&\big\Vert&&\big\downarrow&&\big\downarrow&&\big\Vert\\
0&\longrightarrow&C_B^\vee&\longrightarrow&
W_\Lambda^{1,0}&\overset{c_\Lambda}{\longrightarrow}&
R_\Lambda^{1,1}&\overset{\beta_\Lambda}{\longrightarrow}&
K_B^\vee\longrightarrow0 .
\end{array}
\]
Explicitly,
\[
 \ker\beta_\Lambda
 =\ker\beta_0/(J_\Lambda)_1
 =\operatorname{im}\pi_{\mathrm{vis}}/(J_\Lambda)_1
 =\operatorname{im}c_\Lambda,
\]
and surjectivity of $\beta_\Lambda$ is inherited from $\beta_0$.
\end{proof}

\begin{proposition}[Indispensable-chart covariance]
\label{prop:curvature-chart-covariance}
The map $c_\Lambda$, its visible-coordinate formula, and every arrow in
Theorem~\ref{thm:curvature-exact-sequence} are natural under a change of
distinguished indispensable label.  Thus the curvature exact sequence depends only on the labelled LVM configuration, independently of the chosen indispensable chart.
\end{proposition}

\begin{proof}
Choose $a_\nu$, with $a_0$ replaced as denominator.  On exponent spaces the
monomial chart transition is an integral isomorphism
\[
 P_\nu:E_\Lambda\longrightarrow E'_\Lambda
\]
determined by
\[
 \xi'_j=\xi_j-\xi_{a_\nu},\qquad
 \xi'_{a_0}=-\xi_{a_\nu},\qquad
 \xi'_{a_\mu}=\xi_{a_\mu}-\xi_{a_\nu}
 \quad(\mu\neq\nu).
\]
It satisfies $\Psi'_\Lambda=P_\nu\Psi_\Lambda$ and carries
$\mathfrak h_\Lambda$, $\mathfrak r_\Lambda$, and
$\mathcal L_\Lambda$ to their primed counterparts.  A covector class is
carried by $\mu'=(P_\nu^{-1})^\vee\mu$.  Naturality of the characteristic
square in Proposition~\ref{prop:CW-coordinate-formula} gives
\[
 P_\nu^*\!\left(c'_\Lambda([\mu'])\right)=c_\Lambda([\mu]).
\]
In coordinates, $P_\nu^*v'_j=v_j-x_{a_\nu}=v_j$, because the denominator
coordinate is ghost and its class vanishes in the basic ring.

After ordering the new ghost coordinates, the two blocks satisfy
\[
 B'=LB,\qquad L\in\operatorname{GL}_m(\Z),
 \qquad
 A'_j=A_j-B_\nu .
\]
Therefore $\ker B'=\ker B$ on the common action-parameter space, while
\[
 \operatorname{coker}B\longrightarrow\operatorname{coker}B',
 \qquad [g]\longmapsto[Lg],
\]
is an isomorphism.  Finally, $A'|_{\ker B}=A|_{\ker B}$ because the row
$B_\nu$ vanishes on $\ker B$.  Hence the maps $\beta_\Lambda$ and the
kernel--cokernel identifications also commute with the chart transition.
\end{proof}

\begin{corollary}
There are canonical identifications
\[
 \ker c_\Lambda\cong C_B^\vee,
 \qquad
 \operatorname{coker}c_\Lambda\cong K_B^\vee.
\]
Consequently,
\[
 \dim\ker c_\Lambda
 =
 \dim\operatorname{coker}c_\Lambda
 =
 \delta
\]
and
\[
 \operatorname{rank}c_\Lambda
 =
 \operatorname{rank}B.
\]
\end{corollary}

\begin{corollary}[Curvature full-rank criterion]
\label{cor:curvature-full-rank-criterion}
The following are equivalent:
\[
 B\text{ is invertible},
 \qquad
 \delta=0,
 \qquad
 c_\Lambda\text{ is injective},
 \qquad
 c_\Lambda\text{ is surjective},
 \qquad
 c_\Lambda\text{ is an isomorphism}.
\]
\end{corollary}

\subsection{Holomorphic forms}

\begin{theorem}\label{thm:holomorphic-forms-defect}
For $0\leq p\leq N$,
\[
 H^0(\XL,\Omega_{\XL}^p)
 \cong
 \Lambda^pC_B^\vee.
\]
Hence
\[
 h^{p,0}(\XL)=\binom{\delta}{p}.
\]
\end{theorem}

\begin{proof}
Choose a splitting
\[
 W_\Lambda^{1,0}
 =
 \ker c_\Lambda\oplus W'
\]
with $c_\Lambda|_{W'}$ injective. In Dolbeault bidegree $(p,0)$ there are
no incoming boundaries.  On a summand containing $b>0$ factors from
$W'$, its coefficient-degree-zero part is the exterior comultiplication
\[
 \Lambda^bW'\longrightarrow c_\Lambda(W')\otimes\Lambda^{b-1}W'.
\]
It is injective: after identifying $c_\Lambda(W')$ with $W'$, composition
with exterior multiplication is multiplication by $b$.  Hence the kernel
is $\Lambda^p\ker c_\Lambda$. Apply
Theorem~\ref{thm:curvature-exact-sequence}.
\end{proof}

Since $W_\Lambda^{0,1}$ is closed,
\[
 H_{\bar\partial}^{0,q}(\XL)\cong\Lambda^qW_\Lambda^{0,1},
 \qquad
 h^{0,q}(\XL)=\binom mq.
\]

\begin{theorem}\label{thm:first-Hodge-row}
For $0\leq q\leq N$,
\[
 H_{\bar\partial}^{1,q}(\XL)
 \cong
 \left(C_B^\vee\otimes\Lambda^qW_\Lambda^{0,1}\right)
 \oplus
 \left(K_B^\vee\otimes
 \Lambda^{q-1}W_\Lambda^{0,1}\right).
\]
Consequently,
\[
h^{1,q}(\XL)
 =
 \delta\binom{m+1}{q}.
\]
Here exterior powers and binomial coefficients outside their natural
ranges are zero.
\end{theorem}

\begin{proof}
The holomorphic curvature complex has
\[
 H^{1,0}=\ker c_\Lambda,\qquad
 H^{1,1}=\operatorname{coker}c_\Lambda,
\]
and no other cohomology in holomorphic degree one. Tensor with the closed
exterior algebra on $W_\Lambda^{0,1}$.
\end{proof}

\begin{corollary}
The Dolbeault--Hochster formula~\eqref{eq:intro-main-formula} can hold only
when $\delta=0$.
\end{corollary}

\begin{proof}
Formula~\eqref{eq:intro-main-formula} gives $h^{1,0}=0$, since every one-vertex induced
subcomplex is a point. Theorem~\ref{thm:holomorphic-forms-defect} gives
$h^{1,0}=\delta$.
\end{proof}

\section{Dolbeault cohomology at arbitrary rank}
\label{sec:universal-curvature-formula}

The closed antiholomorphic generators can be separated from the
finite model.  The remaining differential is the Koszul differential
of the curvature map.

Define the holomorphic curvature complex
\[
 \mathscr K_\Lambda
 =
 \left(
 R_\Lambda\otimes\Lambda W_\Lambda^{1,0},
 \bar d
 \right),
\]
where
\[
 \bar d
 \left(
 r\otimes\xi_1\wedge\cdots\wedge\xi_\ell
 \right)
 =
 \sum_{i=1}^{\ell}
 (-1)^{i-1}
 r\,c_\Lambda(\xi_i)
 \otimes
 \xi_1\wedge\cdots\widehat{\xi_i}\cdots\wedge\xi_\ell.
\]
Write
\[
 \mathcal H_\Lambda^{p,a}
 =
 H^{p,a}(\mathscr K_\Lambda),
 \qquad
 \kappa_\Lambda^{p,a}
 =
 \dim_\C\mathcal H_\Lambda^{p,a}.
\]

\begin{proposition}
For every fixed holomorphic degree $p$,
\[
 \mathscr K_\Lambda^{p,a}
 =
 R_\Lambda^a
 \otimes
 \Lambda^{p-a}W_\Lambda^{1,0}.
\]
Thus the complex in holomorphic degree $p$ is finite, with
\[
 \max(0,p-m)\leq a\leq\min(p,d).
\]
\end{proposition}

\begin{proof}
An element of
$R_\Lambda^a\otimes\Lambda^\ell W_\Lambda^{1,0}$
has bidegree $(a+\ell,a)$. Set $p=a+\ell$.
\end{proof}
We set $\mathcal H_\Lambda^{p,a}=0$ whenever
\[
 a<0,
 \qquad
 a>d,
 \qquad
 p-a<0,
 \text{ or }
 p-a>m.
\]
With this convention, all sums below are finite and their bounds are implicit.

Since the antiholomorphic generators are closed,
\[
 \mathcal D_\Lambda
 =
 \mathscr K_\Lambda
 \otimes
 (\Lambda W_\Lambda^{0,1},0).
\]

\begin{theorem}[Dolbeault cohomology at arbitrary rank]
\label{thm:universal-Hodge-formula}
For $0\leq p,q\leq N$,
\[
 H_{\bar\partial}^{p,q}(\XL)
 \cong
 \bigoplus_{s=0}^{m}
 \mathcal H_\Lambda^{p,q-s}
 \otimes
 \Lambda^sW_\Lambda^{0,1}.
\]
Consequently,
\[
 h^{p,q}(\XL)
 =
 \sum_{s=0}^{m}
 \binom{m}{s}\,
 \kappa_\Lambda^{p,q-s}.
\]
\end{theorem}

\begin{proof}
Apply the K\"unneth isomorphism to the tensor-product decomposition of the
global Dolbeault--Cartan model.
\end{proof}

Define
\[
 \operatorname{Curv}_\Lambda(u,v)
 =
 \sum_{p,a}\kappa_\Lambda^{p,a}u^pv^a.
\]

\begin{corollary}
The Dolbeault polynomial is
\[
 \operatorname{Dol}_{\XL}(u,v)
 =
 (1+v)^m\operatorname{Curv}_\Lambda(u,v).
\]
\end{corollary}

\subsection{The defect factor}

Put
\[
 Z_\Lambda=\ker c_\Lambda=C_B^\vee,
 \qquad
 U_\Lambda=\operatorname{im}c_\Lambda\subseteq R_\Lambda^{1,1},
\]
and let
\[
 \overline W_\Lambda^{1,0}=W_\Lambda^{1,0}/Z_\Lambda.
\]
The curvature map induces a canonical isomorphism
\[
 \overline c_\Lambda:
 \overline W_\Lambda^{1,0}\overset{\sim}{\longrightarrow}U_\Lambda.
\]
Define the intrinsic reduced curvature complex by
\[
 \mathscr K_{\Lambda,\mathrm{red}}
 =
 R_\Lambda\otimes
 \Lambda\overline W_\Lambda^{1,0},
 \qquad
 \bar d\,\overline w=\overline c_\Lambda(\overline w).
\]

\begin{proposition}
\label{prop:defect-exterior-factor}
Every linear splitting of
$W_\Lambda^{1,0}\twoheadrightarrow\overline W_\Lambda^{1,0}$ induces an
isomorphism of differential bigraded algebras
\[
 \mathscr K_\Lambda
 \cong
 (\Lambda Z_\Lambda,0)
 \otimes
 \mathscr K_{\Lambda,\mathrm{red}}.
\]
Consequently,
\[
 H(\mathscr K_\Lambda)
 \cong
 \Lambda Z_\Lambda
 \otimes
 H(\mathscr K_{\Lambda,\mathrm{red}}).
\]
\end{proposition}

\begin{proof}
Choose
$W_\Lambda^{1,0}=Z_\Lambda\oplus W'$.
The exterior-algebra isomorphism induced by this direct sum intertwines
the differentials: the curvature differential vanishes on $Z_\Lambda$,
while $W'\to\overline W_\Lambda^{1,0}$ identifies the restriction of
$c_\Lambda$ with $\overline c_\Lambda$.  This proves the tensor
decomposition.  If another complement is the graph of
$f:W'\to Z_\Lambda$, then $w\mapsto w+f(w)$ is a chain isomorphism because
$c_\Lambda f=0$.
\end{proof}

\begin{remark}
The exterior algebra $\Lambda Z_\Lambda=\Lambda C_B^\vee$, the quotient
complex $\mathscr K_{\Lambda,\mathrm{red}}$, and the subspace
$U_\Lambda\subset R_\Lambda^1$ are intrinsic.  The tensor decompositions of complexes and cohomology depend on the chosen complement to $Z_\Lambda$; the dimension factorization is independent of this choice.
\end{remark}

If
\[
 \operatorname{Curv}_{\Lambda,\mathrm{red}}(u,v)
 =
 \sum_{p,a}
 \dim H^{p,a}(\mathscr K_{\Lambda,\mathrm{red}})
 u^pv^a,
\]
then
\[
 \operatorname{Dol}_{\XL}(u,v)
 =
 (1+u)^\delta(1+v)^m
 \operatorname{Curv}_{\Lambda,\mathrm{red}}(u,v).
\]

\begin{proposition}[Exterior factors and the curvature image]
\label{prop:position-sensitive-readout}
The closed antiholomorphic generators $W_\Lambda^{0,1}$ give the factor
$(1+v)^m$ at every rank.  The curvature kernel gives $(1+u)^\delta$.
The integers $m$ and $\delta$ determine all $h^{p,0}(\XL)$ and the entire row $h^{1,q}(\XL)$.
The remaining factor
$\operatorname{Curv}_{\Lambda,\mathrm{red}}$ is the Koszul-homology
polynomial of the pair
\[
 \bigl(R_\Lambda,U_\Lambda\subset R_\Lambda^1\bigr).
\]
The differential depends on multiplication in $R_\Lambda$ and on the inclusion $U_\Lambda\subset R_\Lambda^1$.  For polygons, Theorem~\ref{thm:polygonal-rigidity}
shows that its cohomology dimensions depend only on $m$ and $\delta$.
\end{proposition}

\begin{proof}
The exterior factors and low rows follow from
Theorem~\ref{thm:universal-Hodge-formula},
Theorem~\ref{thm:holomorphic-forms-defect},
Theorem~\ref{thm:first-Hodge-row}, and
Proposition~\ref{prop:defect-exterior-factor}.  Via
$\overline c_\Lambda$, the reduced complex is the Koszul complex of the
inclusion $U_\Lambda\hookrightarrow R_\Lambda$; its differential uses
multiplication by elements of $U_\Lambda\subset R_\Lambda^1$.
\end{proof}

\begin{corollary}[Exact defect order]
\label{cor:exact-defect-order}
One has
\[
 \operatorname{Dol}_{\XL}(u,0)=(1+u)^\delta,
 \qquad
 \operatorname{ord}_{u=-1}\operatorname{Dol}_{\XL}(u,v)=\delta.
\]
In particular, positive defect is equivalent to divisibility of the
Dolbeault polynomial by $1+u$.
\end{corollary}

\begin{proof}
The first identity is
Theorem~\ref{thm:holomorphic-forms-defect}.  The defect factorization
gives divisibility by $(1+u)^\delta$.  No higher power can divide the full
polynomial, because its specialization at $v=0$ has exact order
$\delta$ at $u=-1$.
\end{proof}

\begin{corollary}
If $m>0$, then
\[
 \chi(\XL,\Omega_{\XL}^p)=0
\]
for every $p$. If $\delta>0$, then
\[
 \sum_p(-1)^ph^{p,q}(\XL)=0
\]
for every $q$.
\end{corollary}

\begin{proof}
Evaluate the Dolbeault polynomial at $v=-1$ and, respectively, at $u=-1$.
\end{proof}

\subsection{The total \texorpdfstring{$E_1$}{E1} dimension}

\begin{proposition}
One has
\[
 \sum_{p,q}h^{p,q}(\XL)
 =
 2^m\dim_\C H(\mathscr K_\Lambda)
 =
 2^{m+\delta}
 \dim_\C H(\mathscr K_{\Lambda,\mathrm{red}}).
\]
\end{proposition}

\begin{proof}
Sum the formula of Theorem~\ref{thm:universal-Hodge-formula} over $p$ and $q$.  The closed
exterior algebra on the $m$-dimensional space
$W_\Lambda^{0,1}$ contributes $2^m$.  The defect factorization then
contributes the additional factor $2^\delta$.
\end{proof}

By Theorem~\ref{thm:smooth-moment-angle-model},
\[
 \XL\cong_{\mathrm{homeo}}\mathcal Z_K\times T^m,
\]
so
\[
 \sum_{\ell}b_{\ell}(\XL)
 =
 2^m\sum_{\ell}b_{\ell}(\mathcal Z_K).
\]

\begin{theorem}[Curvature criterion for Fr\"olicher degeneration]
\label{thm:curvature-Frolicher-criterion}
The Fr\"olicher spectral sequence degenerates at $E_1$ if and only if
\[
 \dim_\C H(\mathscr K_\Lambda)
 =
 \sum_{\ell}b_{\ell}(\mathcal Z_K).
\]
\end{theorem}

\begin{proof}
A finite-dimensional spectral sequence degenerates at $E_1$ exactly when
the total dimensions of $E_1$ and $E_\infty$ agree. Use the two preceding
formulas and cancel $2^m$.
\end{proof}

\begin{definition}
The \emph{curvature Fr\"olicher excess} is
\[
 \mathfrak F(\Lambda)
 =
 \dim_\C H(\mathscr K_\Lambda)
 -
 \sum_{\ell}b_{\ell}(\mathcal Z_K).
\]
\end{definition}

Then
\[
 \mathfrak F(\Lambda)\geq0,
 \qquad
 \mathfrak F(\Lambda)=0
 \iff
 E_1=E_\infty,
\]
and
\[
\dim E_1-\dim E_\infty
 =
 2^m\mathfrak F(\Lambda).
\]

\begin{lemma}[The first Betti number]
\label{lem:all-rank-b1}
At every geometrically realizable rank,
\[
 b_1(\XL)=m.
\]
\end{lemma}

\begin{proof}
The moment-angle Hochster decomposition has
$b_1(\mathcal Z_K)=0$: the empty support and every singleton support
have zero reduced cohomology in the degree required to contribute to
$H^1$, and larger supports have degree too large.  At every admissible rank, the homeomorphism $\XL\cong\mathcal Z_K\times T^m$ and the K\"unneth formula give
$b_1(\XL)=m$.
\end{proof}

\begin{theorem}[Fr\"olicher obstruction at deficient curvature rank]
\label{thm:singular-Frolicher-obstruction}
If $\delta>0$, then the Fr\"olicher spectral sequence of $\XL$ does not
degenerate at $E_1$.  Equivalently,
\[
 \delta>0\quad\Longrightarrow\quad\mathfrak F(\Lambda)>0.
\]
The obstruction is already visible in total degree one:
\[
 h^{1,0}(\XL)+h^{0,1}(\XL)=m+\delta>b_1(\XL)=m.
\]
\end{theorem}

\begin{proof}
Theorem~\ref{thm:holomorphic-forms-defect} gives
$h^{1,0}=\delta$, while the closed antiholomorphic generators give
$h^{0,1}=m$.  Apply Lemma~\ref{lem:all-rank-b1}.  Degeneration at $E_1$
would identify the total-degree-one dimension of $E_1$ with $b_1$, which
is impossible when $\delta>0$.  The statement about
$\mathfrak F(\Lambda)$ follows from
Theorem~\ref{thm:curvature-Frolicher-criterion}.
\end{proof}

\begin{remark}
At full rank, $c_\Lambda$ identifies $W_\Lambda^{1,0}$ with
$R_\Lambda^{1,1}$.  In Part~II,
$H(\mathscr K_\Lambda)$ is identified with Stanley--Reisner Tor and yields
the Dolbeault--Hochster formula.
\end{remark}

\begin{corollary}[Failure of the \texorpdfstring{$\partial\bar\partial$}{ddbar} lemma]
\label{cor:all-rank-ddbar-failure}
Every minimally stable polytopal LVM manifold with $m>0$ fails the
$\partial\bar\partial$ lemma.  At full curvature rank, its Fr\"olicher spectral sequence degenerates at $E_1$, and
\[
 \sum_{p+q=j}h^{p,q}(\XL)=b_j(\XL)
 \qquad\text{for every }j.
\]
\end{corollary}

\begin{proof}
The model at arbitrary rank and
Theorem~\ref{thm:real-indispensable-frame} give
\[
 h^{1,0}(\XL)=\delta
 \leq\left\lfloor\frac m2\right\rfloor
 <m=h^{0,1}(\XL).
\]
The $\partial\bar\partial$ lemma would force the reverse inequality.
Indeed, if $\alpha$ is a $\bar\partial$-closed $(0,1)$-form, then
$\partial\alpha$ is $d$-closed and $\partial$-exact.  The lemma would
give $\partial\alpha=\partial\bar\partial f$ for a smooth function $f$.
Thus $\alpha-\bar\partial f$ would be a $d$-closed representative of
$[\alpha]$, whose conjugate is a holomorphic one-form.  Conjugation
would therefore give an antilinear surjection from the space of closed holomorphic
one-forms onto $H^{0,1}(\XL)$, contradicting the displayed dimensions.

At full rank, the Betti-number equality and $E_1$-degeneration follow
from Theorem~\ref{thm:exact-full-rank-criterion}.
\end{proof}

\section{Polygonal Hodge diamonds}
\label{sec:singular-readout-tests}

We compute the rank-one square and Faucard's rank-two pentagon, then derive
the formula for an arbitrary polygon.  The common reason for rigidity
is the perfect multiplication pairing in degree one.

\subsection{The rank-one square}

For Example~\ref{ex:singular-square}, write $x,y$ for the common classes
of the two pairs of opposite visible variables.  Directly from the
projected relation space and the two minimal nonfaces,
\[
 R_\Lambda
 \cong
 \frac{\C[x,y]}{(x^2,y^2)}.
\]
Here $m=2$, $\delta=1$, and the curvature image is the line
\[
 U_\Lambda
 =
 \C\bigl((3+i)x+(-2+i)y\bigr).
\]

\begin{proposition}[Rank-one square Hodge diamond]
\label{prop:singular-square-Hodge-table}
For the rank-one square,
\[
 \operatorname{Curv}_{\Lambda,\mathrm{red}}(u,v)
 =1+uv+u^2v+u^3v^2
\]
and
\[
 \operatorname{Dol}_{\XL}(u,v)
 =(1+u)(1+v)^2(1+uv+u^2v+u^3v^2).
\]
Its Hodge rows $(h^{p,0},\ldots,h^{p,4})$ are
\[
\begin{array}{c|ccccc}
p\backslash q&0&1&2&3&4\\ \hline
0&1&2&1&0&0\\
1&1&3&3&1&0\\
2&0&2&4&2&0\\
3&0&1&3&3&1\\
4&0&0&1&2&1
\end{array}.
\]
Moreover,
\[
 \mathfrak F(\Lambda)=4,
 \qquad
 \dim E_1=32,
 \qquad
 \dim E_\infty=16.
\]
\end{proposition}

\begin{proof}
Let $\ell=(3+i)x+(-2+i)y$.  Multiplication by $\ell$ is injective
from $R_\Lambda^0$ to $R_\Lambda^1$ and surjective from
$R_\Lambda^1$ to $R_\Lambda^2=\C xy$.  The one-generator reduced
Koszul complex therefore has one-dimensional homology in bidegrees
\[
 (p,a)=(0,0),(1,1),(2,1),(3,2),
\]
which gives the reduced polynomial.  The closed kernel generator and the
two closed antiholomorphic generators supply the factors $(1+u)$ and
$(1+v)^2$.  Expanding gives the table.  Finally,
$\mathcal Z_K\cong S^3\times S^3$, so
$\sum b_\ell(\mathcal Z_K)=4$, while
$\dim H(\mathscr K_\Lambda)=8$.  Apply the definition of
$\mathfrak F$ and the factor $2^m=4$.
\end{proof}

\begin{proposition}[A full-rank realization of the same visible square]
\label{prop:full-rank-square-companion}
The visible square of Example~\ref{ex:singular-square} also admits a full-rank realization.  For that realization,
\[
 \operatorname{Dol}_{X_{\square}^{\mathrm{fr}}}(u,v)
 =(1+v)^2(1+2u^2v+u^4v^2)
\]
and its Hodge rows are
\[
\begin{array}{c|ccccc}
p\backslash q&0&1&2&3&4\\ \hline
0&1&2&1&0&0\\
1&0&0&0&0&0\\
2&0&2&4&2&0\\
3&0&0&0&0&0\\
4&0&0&1&2&1.
\end{array}
\]
Thus the full-rank and rank-deficient realizations have the same visible complex
and different Hodge diamonds; the difference appears in bidegree $(1,0)$.
\end{proposition}

\begin{proof}
In the notation of Example~\ref{ex:singular-square}, put
\[
 v_1=\Lambda_2-\Lambda_1=(-1+i,0),\qquad
 v_2=\Lambda_3-\Lambda_1=(-2-i,0).
\]
The ordered list
\[
 v_1,\ v_2,\ (0,1),\ (0,i)
\]
is a real basis of $\C^2$.  Let $F\in\operatorname{GL}_{\R}(\C^2)$ be
the real-linear map sending this basis respectively to
\[
 (1,0),\ (0,1),\ (i,0),\ (0,i),
\]
and set $\Lambda_j^{\mathrm{fr}}=F(\Lambda_j)$.  Because $F$ is invertible
and fixes the origin, it preserves every convex-support test, hence
admissibility, minimal stability, indispensability, polytopality, and the
visible square.  Its indispensable difference matrix is the identity,
so it has complex rank $2=m$.  The full-rank Dolbeault--Hochster formula
and the two opposite pairs in the square give the displayed polynomial;
expansion gives the table.
\end{proof}

Explicitly, if $z_1=x+iy$ and $z_2=a+ib$, the map used in the proof is
\[
 F(z_1,z_2)
 =\left(\frac{-x+2y}{3}+ia,
          -\frac{x+y}{3}+ib\right),
\]
whose real determinant is $-1/3$.

\subsection{Faucard's rank-two pentagon}

Consider the configuration from \cite[Example~3]{Faucard2024}, written
with its nine weights as columns:
\[
 \Lambda^{\mathrm F}
 =
 \begin{pmatrix}
 0&0&0&0&2&1+2i&-2+i&-2-i&1-2i\\
 -1-i&i&0&0&1&1&1&1&1\\
 1&1&-1-i&i&1&1&1&1&1
 \end{pmatrix}.
\]
The first four labels are indispensable.  Thus
$(m,n,k)=(3,9,4)$, the visible complex is the five-cycle, and
\[
 \operatorname{rank}_\C B=2,
 \qquad
 \delta=1.
\]
Let $x_1,\ldots,x_5$ denote the visible variables in the displayed order.
The visible linear relations are
\[
 -6x_1+8x_2-9x_3+7x_4=0,
 \qquad
 -12x_1+9x_2-4x_3+7x_5=0,
\]
and
\[
 I_K=(x_1x_2,x_1x_5,x_2x_3,x_3x_4,x_4x_5).
\]
Put $y=x_3$, $z=x_4$, and $w=x_5$.  Then
\[
 x_1=\tfrac76y-\tfrac32z+\tfrac43w,
 \qquad
 x_2=2y-2z+w.
\]
The Gorenstein multiplication pairing
$R_\Lambda^1\otimes R_\Lambda^1\to R_\Lambda^2$ is represented, up to
the choice of a nonzero top class, by
\[
 Q=
 \begin{pmatrix}
 4&0&-8\\
 0&4&0\\
 -8&0&7
 \end{pmatrix}
 \quad\text{in the basis }(y,z,w).
\]
The curvature image is the plane spanned by
\[
 \ell_1=(-3-4i)y+5z,
 \qquad
 \ell_2=(4-3i)y+5w.
\]
One has $\det Q=-144$, and the determinant of the restriction of $Q$ to
$(\ell_1,\ell_2)$ is $-5400-7200i$.

\begin{proposition}[Faucard-pentagon Hodge diamond]
\label{prop:Faucard-pentagon-Hodge-table}
For $\Lambda^{\mathrm F}$,
\[
 \operatorname{Curv}_{\Lambda^{\mathrm F},\mathrm{red}}(u,v)
 =1+uv+4u^2v+u^3v+u^4v^2
\]
and
\[
 \operatorname{Dol}_{X_{\Lambda^{\mathrm F}}}(u,v)
 =(1+u)(1+v)^3
 (1+uv+4u^2v+u^3v+u^4v^2).
\]
Its Hodge rows $(h^{p,0},\ldots,h^{p,5})$ are
\[
\begin{array}{c|cccccc}
p\backslash q&0&1&2&3&4&5\\ \hline
0&1&3&3&1&0&0\\
1&1&4&6&4&1&0\\
2&0&5&15&15&5&0\\
3&0&5&15&15&5&0\\
4&0&1&4&6&4&1\\
5&0&0&1&3&3&1
\end{array}.
\]
Moreover,
\[
 \mathfrak F(\Lambda^{\mathrm F})=4,
 \qquad
 \dim E_1=128,
 \qquad
 \dim E_\infty=96.
\]
\end{proposition}

\begin{proof}
The displayed relations follow by taking the annihilator of the real
action matrix and projecting to the five visible coordinates.  The five
quadratic nonfaces then give the displayed multiplication pairing.  The
kernel of $B$ is the first action-parameter line, so
$\beta_\Lambda:R_\Lambda^1\to(\ker B)^\vee$ has coefficient vector
\[
 (-2+i,-2-i,1-2i)
\]
in the basis $(y,z,w)$.  Its kernel is precisely the plane generated by
$\ell_1,\ell_2$, in agreement with
Theorem~\ref{thm:curvature-exact-sequence}.

The two-generator reduced Koszul chains have homology dimensions
\[
 1,1,4,1,1
\]
in bidegrees
\[
 (0,0),(1,1),(2,1),(3,1),(4,2),
\]
respectively.  This follows directly from the injectivity of
$U_\Lambda\hookrightarrow R_\Lambda^1$, the one-dimensional top degree,
and the nondegenerate multiplication pairing above.  The defect and
antiholomorphic factors give the polynomial and table.  Finally,
$\mathcal Z_K\cong\#_5(S^3\times S^4)$, so its total Betti dimension is
$12$, whereas $\dim H(\mathscr K_{\Lambda^{\mathrm F}})=16$.
The asserted excess and page dimensions follow.
\end{proof}

At full rank the pentagon polynomial is
\[
 (1+v)^3(1+5u^2v+5u^3v+u^5v^2),
\]
whose $(1,0)$ coefficient is zero, while the rank-two polynomial has coefficient one.

\subsection{Rigidity at fixed rank}
\label{sec:polygonal-rigidity}

The preceding calculations depend on a common feature of every polygon.
If $K=C_l$, then $N=l$, $d=2$, and $m=l-2$.  The basic ring has
\[
 R_\Lambda=R_\Lambda^0\oplus R_\Lambda^1\oplus R_\Lambda^2,
 \qquad
 (\dim R_\Lambda^0,\dim R_\Lambda^1,\dim R_\Lambda^2)=(1,m,1).
\]
Choose a nonzero class $\omega\in R_\Lambda^2$ and write
$xy=Q(x,y)\omega$ for $x,y\in R_\Lambda^1$.  By
Theorem~\ref{thm:basic-Artinian-Gorenstein}, $Q$ is a nondegenerate
symmetric bilinear form.  This pairing determines the ranks needed in
the curvature calculation.

\begin{theorem}[Polygonal rigidity]\label{thm:polygonal-rigidity}
Let $\XL$ be a minimally stable polytopal LVM manifold whose visible
complex is the $l$-cycle, $l\geq3$.  Put
\[
 m=l-2,\qquad \rho=\operatorname{rank}_{\C}B,\qquad \delta=m-\rho.
\]
Then
\begin{equation}\label{eq:polygonal-all-rank}
 \begin{aligned}
 \operatorname{Dol}_{\XL}(u,v)
 &=(1+u)^\delta(1+v)^m
 \left[1+v\sum_{p=1}^{\rho+1}b_{m,\rho,p}u^p
             +u^{\rho+2}v^2\right],\\
 b_{m,\rho,p}
 &=m\binom{\rho}{p-1}-\binom{\rho}{p}-\binom{\rho}{p-2}.
 \end{aligned}
\end{equation}
Binomial coefficients outside their usual ranges are zero.  In
particular, the complete Hodge diamond is determined by $l$ and $\rho$
throughout the realizable interval
$\lceil m/2\rceil\leq\rho\leq m$.
\end{theorem}

\begin{proof}
Identify $\overline W_\Lambda^{1,0}$ with
$U=U_\Lambda\subset R_\Lambda^1$ by the curvature map.
The curvature exact sequence gives $\dim U=\rho$.  The reduced
curvature complex in holomorphic degree $p$ has just three terms:
\begin{equation}\label{eq:polygonal-three-term}
 \Lambda^pU
 \xrightarrow{d_0}
 R_\Lambda^1\otimes\Lambda^{p-1}U
 \xrightarrow{d_1}
 \C\omega\otimes\Lambda^{p-2}U.
\end{equation}
Their curvature degrees are $a=0,1,2$, respectively.  The differentials
are
\[
 \begin{aligned}
 d_0(u_1\wedge\cdots\wedge u_p)
 &=\sum_{j=1}^p(-1)^{j-1}u_j\otimes
    u_1\wedge\cdots\widehat{u_j}\cdots\wedge u_p,\\
 d_1(x\otimes u_1\wedge\cdots\wedge u_{p-1})
 &=\sum_{j=1}^{p-1}(-1)^{j-1}Q(x,u_j)\omega\otimes
    u_1\wedge\cdots\widehat{u_j}\cdots\wedge u_{p-1}.
 \end{aligned}
\]
Symmetry of $Q$ gives $d_1d_0=0$.

For $1\leq p\leq\rho$, choose a projection $R_\Lambda^1\to U$.
Projecting the first factor after $d_0$ and then taking exterior
product gives $p$ times the identity on $\Lambda^pU$.  Hence $d_0$ is
injective.  For $2\leq p\leq\rho+1$, nondegeneracy of $Q$ gives a
surjection
\[
 R_\Lambda^1\longrightarrow U^\vee,\qquad
 x\longmapsto Q(x,-)|_U.
\]
The second differential is therefore surjective.  Explicitly, for a
basis wedge $u_J$ with $|J|=p-2$, choose a basis element $u_k$ outside
$J$ and an $x\in R_\Lambda^1$ whose pairing on $U$ is the dual
covector of $u_k$.  Then
$d_1(x\otimes u_k\wedge u_J)=\omega\otimes u_J$.

Thus, for $1\leq p\leq\rho+1$, the cohomology of
\eqref{eq:polygonal-three-term} lies in curvature degree one and has
dimension $b_{m,\rho,p}$.  In degree $p=0$ there is a single copy of
$\C$ in curvature degree zero.  In degree $p=\rho+2$ the only term
is $\C\omega\otimes\Lambda^\rho U$, in curvature degree two.  All
other holomorphic degrees vanish.  This proves
\[
 \operatorname{Curv}_{\Lambda,\mathrm{red}}(u,v)
 =1+v\sum_{p=1}^{\rho+1}b_{m,\rho,p}u^p+u^{\rho+2}v^2.
\]
The two closed exterior factors in
Proposition~\ref{prop:defect-exterior-factor} and
Theorem~\ref{thm:universal-Hodge-formula} now give
\eqref{eq:polygonal-all-rank}.  The realizable interval is
Theorem~\ref{thm:exact-realizable-rank-range}.
\end{proof}

The proof uses the perfect pairing on $R_\Lambda^1$ even when $Q|_U$
is degenerate.  Both the ring and the curvature subspace may vary;
their dimensions fix the ranks of the two arrows in
\eqref{eq:polygonal-three-term}.  Consequently every Hodge number is
constant in a family of the stated polygonal manifolds on which the
curvature rank is constant.

The formula recovers the examples above directly.  For
$(m,\rho)=(2,1)$ it gives
$1+uv+u^2v+u^3v^2$ as the reduced polynomial; for $(m,\rho)=(3,2)$
it gives $1+uv+4u^2v+u^3v+u^4v^2$.  At full rank, Pascal's identity
gives
\[
 b_{m,m,p}
 =(m+2)\binom m{p-1}-\binom{m+2}p.
\]
The coefficients at $p=1,m+1$ vanish.  Substitution of $l=m+2$
recovers the polygonal Hochster formula of
Theorem~\ref{thm:polygonal-example}, including the triangle, whose
proper-support contribution is zero.

\begin{corollary}[Change of curvature rank]
\label{cor:polygonal-rank-change}
Write $D_{m,\rho}(u,v)$ for the polynomial in
\eqref{eq:polygonal-all-rank}.  If $\rho$ and $\rho-1$ are both
geometrically realizable, then
\begin{equation}\label{eq:polygonal-rank-difference}
 D_{m,\rho-1}(u,v)-D_{m,\rho}(u,v)
 =(1+u)^{m-\rho}(1+v)^{m+1}
       \bigl(u+u^{\rho+1}v\bigr).
\end{equation}
Thus every Hodge number weakly increases as the rank decreases.
For defect $\delta$ the curvature Fr\"olicher excess is
\begin{equation}\label{eq:polygonal-Frolicher-excess}
 \mathfrak F(\Lambda)=4(2^\delta-1),
 \qquad
 \dim E_1-\dim E_\infty=2^{m+2}(2^\delta-1).
\end{equation}
\end{corollary}

\begin{proof}
Summing the middle coefficients gives
\[
 \sum_{p=1}^{\rho+1}b_{m,\rho,p}u^p
 =(mu-1-u^2)(1+u)^\rho+1+u^{\rho+2}.
\]
Consequently,
\[
 \frac{D_{m,\rho}(u,v)}{(1+v)^m}
 =v(mu-1-u^2)(1+u)^m
   +(1+u)^{m-\rho}(1+v)(1+u^{\rho+2}v).
\]
Subtracting the expressions at $\rho-1$ and $\rho$ proves
\eqref{eq:polygonal-rank-difference}, whose coefficients are
nonnegative.  The total dimension of the reduced curvature cohomology is
\[
 \operatorname{Curv}_{\Lambda,\mathrm{red}}(1,1)
 =(m-2)2^\rho+4.
\]
After restoring the defect factor, this becomes
\[
 \dim H(\mathscr K_\Lambda)=(m-2)2^m+2^{\delta+2}.
\]
At full rank the corresponding dimension is
$(m-2)2^m+4$, which equals the total Betti number of the fixed
moment-angle manifold by
Theorem~\ref{thm:curvature-moment-angle}.  Subtracting proves the formula for
$\mathfrak F$; the closed antiholomorphic factor contributes $2^m$
to the difference of page dimensions.
\end{proof}

Equation~\eqref{eq:polygonal-rank-difference} compares the cohomology
of configurations at the two ranks.  In a holomorphic family in which
such a rank drop occurs, it gives the exact Hodge-number jumps.
The next subsection constructs such a family for the square.
Tangent-sheaf cohomology also depends on the resonant exponents
(Theorem~\ref{thm:tangent-sector-formula}), as the square in
Section~\ref{sec:tangent-square} illustrates.

\subsection{A holomorphic family through the rank-one square}
\label{sec:holomorphic-square-family}

The rank change can be realized by a holomorphic variation of a single
weight.  We first record the parameter dependence of the classical
quotient construction.

\begin{lemma}[The quotient in a family]
\label{lem:parametric-LVM-quotient}
Let $\Lambda_1(t),\ldots,\Lambda_n(t)\in\C^m$ depend holomorphically
on $t$ in a disk $\Delta$.  Suppose every configuration is admissible
and the subsets whose convex hull contains the origin are independent
of $t$.  Write $V\subset\PP^{n-1}$ for the resulting common open set.
Then
\[
 T\cdot(t,[z])
 =\bigl(t,[e^{\langle\Lambda_1(t),T\rangle}z_1:\cdots:
             e^{\langle\Lambda_n(t),T\rangle}z_n]\bigr)
\]
is a free proper holomorphic $\C^m$-action on $\Delta\times V$.
Its quotient is a complex manifold, and the induced map
\[
 \pi:\mathscr X=(\Delta\times V)/\C^m\longrightarrow\Delta
\]
is a proper holomorphic submersion with fibre $X_{\Lambda(t)}$.
All its fibres are diffeomorphic.
\end{lemma}

\begin{proof}
For $(t,[z])\in\Delta\times V$, consider
\[
 F_{t,z}(T)=\sum_i|z_i|^2
 e^{2\operatorname{Re}\langle\Lambda_i(t),T\rangle},
 \qquad T\in\C^m\cong\R^{2m}.
\]
Admissibility places the origin in the interior of the convex hull
of the nonzero support.  As in
Theorem~\ref{thm:classical-LVM-quotient}, $F_{t,z}$ is proper and its
Hessian is positive definite.  Its unique critical point depends
smoothly on both $t$ and $[z]$ by the implicit-function theorem;
it is invariant under rescaling of $z$.

Define the relative transversal
\[
 \mathscr T=
 \left\{(t,[z])\in\Delta\times\PP^{n-1}:
       \sum_i\Lambda_i(t)|z_i|^2=0\right\}.
\]
Every point of $\mathscr T$ belongs to $\Delta\times V$.  At such a
point, the nonzero support affinely spans $\R^{2m}$.  Radial coordinate
variations therefore give rank $2m$ for the defining equations in the
projective directions.  Hence $\mathscr T$ is smooth and
$\mathscr T\to\Delta$ is a submersion.  The minimizing parameter gives
an equivariant diffeomorphism
\[
 \C^m\times\mathscr T\longrightarrow\Delta\times V,
 \qquad (T,(t,x))\longmapsto(t,T\cdot x).
\]
The action becomes translation on the first factor, proving joint
freeness and properness.

The free proper holomorphic quotient theorem gives the complex
structure on $\mathscr X$ and a holomorphic principal bundle
$\Delta\times V\to\mathscr X$.  Projection to $\Delta$ descends to a
holomorphic submersion, since its differential is surjective and
annihilates the orbit directions.  The transversal identifies
$\mathscr X$ smoothly with $\mathscr T$.  For a compact
$C\subset\Delta$, its inverse image is the closed zero set of the
normalized moment map in $C\times\PP^{n-1}$, and is compact.
Thus $\pi$ is proper.  Its fibres are the classical LVM quotients,
and Ehresmann's theorem gives their smooth local triviality.
\end{proof}

Keep the weights of Example~\ref{ex:singular-square} fixed except for
the third, and set
\[
 \Lambda_3(t)=(-1-i,1+t),\qquad
 \Delta=\{t\in\C:|t|<1/20\}.
\]
As before, put $D=\{1,2,3\}$, $E=\{4,5\}$, and $F=\{6,7\}$.

\begin{theorem}[A square family crossing the rank drop]
\label{thm:holomorphic-square-rank-drop}
These weights define a proper holomorphic submersion
$\pi:\mathscr X\to\Delta$ whose fibres $X_t$ are minimally stable
polytopal LVM fourfolds with visible complex
$\partial\Delta^1*\partial\Delta^1$.  The indispensable block is
\[
 B(t)=
 \begin{pmatrix}-1+i&0\\-2-i&t\end{pmatrix},
 \qquad \det B(t)=(-1+i)t.
\]
Consequently the curvature rank is one at $t=0$ and two elsewhere.
The Hodge polynomials are
\[
 \begin{aligned}
 \operatorname{Dol}_{X_0}(u,v)
   &=(1+u)(1+v)^2(1+uv+u^2v+u^3v^2),\\
 \operatorname{Dol}_{X_t}(u,v)
   &=(1+v)^2(1+2u^2v+u^4v^2),\qquad t\ne0,
 \end{aligned}
\]
and their difference is
\begin{equation}\label{eq:square-family-Hodge-jump}
 \operatorname{Dol}_{X_0}(u,v)-\operatorname{Dol}_{X_t}(u,v)
 =(1+v)^3(u+u^3v),\qquad t\ne0.
\end{equation}
Every fibre fails the $\partial\bar\partial$ lemma.  Its Fr\"olicher
spectral sequence degenerates at $E_1$ precisely when $t\ne0$.
\end{theorem}

\begin{proof}
We verify the convex supports throughout the stated disk.  Write
$t=a+ib$ and let $c_j$ be the coefficients of a convex relation for
the origin.  Put
\[
 S=c_1+c_2+c_3,\qquad P=c_4+c_5,\qquad Q=c_6+c_7.
\]
The second coordinate and the normalization give
\[
 S-P-2Q+ac_3=0,\qquad P-Q+bc_3=0,\qquad S+P+Q=1.
\]
Thus
\[
 P=\frac{1+(a-3b)c_3}{5},\qquad
 Q=\frac{1+(a+2b)c_3}{5}.
\]
Since $0\leq c_3\leq1$ and $|a|,|b|\leq1/20$, these imply
\[
 \frac4{25}\leq P,Q\leq\frac6{25},\qquad S\geq\frac{13}{25}.
\]
Let $x=(c_4-c_5)/2$ and $y=(c_6-c_7)/2$.  The first coordinate gives
\[
 c_1=\frac{S-2x+y}{3},\qquad
 c_2=\frac{S+x-2y}{3},\qquad
 c_3=\frac{S+x+y}{3}.
\]
Using $|x|\leq P/2$ and $|y|\leq Q/2$, we obtain
\[
 c_1,c_2\geq\frac4{75},\qquad c_3\geq\frac7{75}.
\]
Every convex relation therefore uses all three labels in $D$ and
at least one label from each of $E$ and $F$.

Conversely, each choice of one label from $E$ and one from $F$
supports a positive relation.  Denote their first coordinates by
$\sigma/2$ and $i\tau/2$, where $\sigma,\tau\in\{1,-1\}$.  Set
\[
 \begin{aligned}
 c_3&=\frac{6+\sigma+\tau}
 {30+(4-\sigma-\tau)a+(3\sigma-2\tau-2)b},\\
 P&=\frac{1+(a-3b)c_3}{5},\qquad
 Q=\frac{1+(a+2b)c_3}{5},\\
 c_1&=c_3-\frac{\sigma P}{2},\qquad
 c_2=c_3-\frac{\tau Q}{2}.
 \end{aligned}
\]
Assign $P,Q$ to the chosen visible labels and zero to the other two.
The denominator is at least $589/20$, so $0<c_3<1$.  Direct
substitution gives the normalized relation.  The preceding bounds on
$P,Q,S$ then give the same positive lower bounds on $c_1,c_2,c_3$.

It follows that a subset supports the origin exactly when it contains
$D$ and meets both $E$ and $F$.  In particular, no four labels support
the origin, exactly $D$ is indispensable, and the visible complex is
the square.  The normalized relation polytope is a quadrilateral,
with the four positive relations above as its vertices.
The common projective open set is
\[
 V=\left\{[z]\in\PP^6:
 z_1z_2z_3\ne0,\ (z_4,z_5)\ne(0,0),\
                         (z_6,z_7)\ne(0,0)\right\}.
\]
Lemma~\ref{lem:parametric-LVM-quotient} supplies the required family.
Its fibres have complex dimension $7-2-1=4$.

Taking indispensable differences gives the displayed matrix $B(t)$.
The two Hodge polynomials follow from
Theorem~\ref{thm:polygonal-rigidity} at $(m,\rho)=(2,1)$ and $(2,2)$;
Corollary~\ref{cor:polygonal-rank-change} gives
\eqref{eq:square-family-Hodge-jump}.  The
$\partial\bar\partial$ assertion is
Corollary~\ref{cor:all-rank-ddbar-failure}, and the Fr\"olicher
assertion follows from
Theorem~\ref{thm:exact-full-rank-criterion}.
\end{proof}

The smooth fibres have the common Poincar\'e polynomial
\[
 \operatorname{Poin}_{X_t}(s)=(1+s)^2(1+s^3)^2.
\]
Their total Betti number is $16$.  The total dimension of the
Dolbeault cohomology is $16$ away from the origin and $32$ at the
origin, where $\mathfrak F=4$.  Thus the sixteen additional
Dolbeault classes occur on a single fibre of a smooth holomorphic
family.  Equation~\eqref{eq:square-family-Hodge-jump} places eight
in holomorphic degree one and eight in holomorphic degree three.

\Needspace{12\baselineskip}
\part{Curvature and the Hochster formula}
\section{Curvature and Stanley--Reisner Tor}
\label{sec:Koszul-Tor-reduction}
\label{sec:Hochster-first-proof}

Assume throughout this part that $B$ is invertible.  Curvature gives
an isomorphism
\[
 c_\Lambda:W_\Lambda^{1,0}\overset{\sim}{\longrightarrow}R_\Lambda^1.
\]
By Theorem~\ref{thm:universal-Hodge-formula}, it remains to compute
\[
 \mathscr K_\Lambda
 =\left(R_\Lambda\otimes\Lambda W_\Lambda^{1,0},\bar d\right).
\]
We compare this algebra with the coordinate Koszul complex of
$\C[K]$.  Regularity of the basic linear relations gives the
comparison, and its squarefree components give the Hochster summands.

Put
\[
 S=\C[v_1,\ldots,v_N],
 \qquad
 \mathcal A_K=\C[K]=S/I_K,
\]
and let
\[
 L_\Lambda=(J_\Lambda)_1\subseteq S_1.
\]
Choose a basis
\[
 \theta_1,\ldots,\theta_d
\]
of $L_\Lambda$.  By Theorems~\ref{thm:lsop}
and~\ref{thm:basic-Artinian-Gorenstein}, this is a regular linear system
of parameters for $\mathcal A_K$, and
\[
 R_\Lambda
 =
 \mathcal A_K/(\theta_1,\ldots,\theta_d).
\]
For an algebra $M$, a finite-dimensional vector space $E$, and a linear
map $\varphi:E\to M$, write
\[
 \Kos_M(E,\varphi)
 =
 (M\otimes\Lambda E,d_\varphi),
 \qquad
 d_\varphi(e)=\varphi(e).
\]

\subsection{From curvature to Stanley--Reisner Tor}

\begin{proposition}[Basis-free curvature complex]
\label{prop:basis-free-curvature-complex}
There is a canonical dga isomorphism
\[
 \mathscr K_\Lambda
 \cong
 \Kos_{R_\Lambda}(R_\Lambda^1,\operatorname{id}).
\]
\end{proposition}

\begin{proof}
Transport the exterior generators through the curvature isomorphism
$c_\Lambda:W_\Lambda^{1,0}\to R_\Lambda^1$.  The defining identity
$\bar d(\xi)=c_\Lambda(\xi)$ becomes the identity map on $R_\Lambda^1$.
\end{proof}

Choose a basis
\[
 \bar x_1,\ldots,\bar x_m
\]
of $R_\Lambda^1$, choose linear lifts $x_1,\ldots,x_m\in S_1$, and
choose $\xi_j\in W_\Lambda^{1,0}$ with
$c_\Lambda(\xi_j)=\bar x_j$.  Since $d+m=N$, the ordered list
\[
 \theta_1,\ldots,\theta_d,x_1,\ldots,x_m
\]
is a basis of $S_1$.  A linear change of basis in $S_1$ identifies the
standard Koszul resolution on $v_1,\ldots,v_N$ with the Koszul complex on
this list.  Consequently,
\[
 \operatorname{Tor}^S_\bullet(\mathcal A_K,\C)
 \cong
 H_\bullet\Kos_{\mathcal A_K}(\theta,x).
\]

\begin{theorem}[Contraction of the l.s.o.p.]
\label{thm:Koszul-contraction}
The quotient map $\mathcal A_K\to R_\Lambda$ induces a dga
quasi-isomorphism
\[
 \Kos_{\mathcal A_K}(\theta,x)
 \overset{\simeq}{\longrightarrow}
 \Kos_{R_\Lambda}(\bar x_1,\ldots,\bar x_m).
\]
\end{theorem}

\begin{proof}
There is a natural factorization
\[
 \Kos_{\mathcal A_K}(\theta,x)
 \cong
 \Kos_{\mathcal A_K}(\theta)
 \otimes_{\mathcal A_K}
 \Kos_{\mathcal A_K}(x).
\]
Because $\theta_1,\ldots,\theta_d$ is a regular sequence, the augmentation
\[
 \Kos_{\mathcal A_K}(\theta)\longrightarrow R_\Lambda
\]
is a dga quasi-isomorphism: its positive Koszul homology vanishes and its
degree-zero homology is $R_\Lambda$.  The second factor is a bounded complex
of finite free $\mathcal A_K$-modules.  It is therefore $K$-flat, so tensoring
the augmentation with it preserves the quasi-isomorphism and the product.
The target is
\[
 R_\Lambda\otimes_{\mathcal A_K}\Kos_{\mathcal A_K}(x)
 =
 \Kos_{R_\Lambda}(\bar x_1,\ldots,\bar x_m).
\]
This is the usual change-of-rings contraction
\cite[Proposition~4.5.6 and Lemma~A.3.5]{BuchstaberPanov2015}, written
here as an explicit dga map in the degree-one normalization.
\end{proof}

\begin{theorem}[Curvature--Tor comparison]
\label{thm:curvature-Tor}
At full rank there is an isomorphism of graded-commutative
algebras
\[
 H(\mathscr K_\Lambda)
 \cong
 \operatorname{Tor}^{S}(\C[K],\C).
\]
\end{theorem}

\begin{proof}
The standard Koszul resolution computes the left side of
\[
 \operatorname{Tor}^{S}(\mathcal A_K,\C)
 \cong
 H\Kos_{\mathcal A_K}(\theta,x).
\]
Apply Theorem~\ref{thm:Koszul-contraction}, and then identify its target
with $\mathscr K_\Lambda$ by
Proposition~\ref{prop:basis-free-curvature-complex}.  All maps are dga
maps, so the product is preserved.
\end{proof}

\begin{remark}[Dependence on choices]
\label{rem:first-bridge-choice-control}
The dga identification in
Proposition~\ref{prop:basis-free-curvature-complex} is canonical.  The
comparison with the standard coordinate Koszul complex requires a basis of
$L_\Lambda$ and linear lifts of a basis of $R_\Lambda^1$.
The resulting isomorphism with $\operatorname{Tor}^S(\C[K],\C)$ preserves
internal degree and multiplication. Tor has its canonical support
multigrading. Transporting it to curvature homology depends on the chosen
comparison; the internal grading, algebra isomorphism type, and bigraded
dimensions are independent of that choice.
\end{remark}

\begin{corollary}
At full rank, the graded-algebra isomorphism class of the
curvature-homology algebra depends only on $K$.  It is independent of the
affine action matrices, the marked-fan realization, the chosen l.s.o.p., and
all auxiliary bases.
\end{corollary}

Give each polynomial variable internal degree $1$, and give every Koszul
generator homological degree $1$ and internal degree $1$.  An element of
\[
 R_\Lambda^a\otimes\Lambda^\ell W_\Lambda^{1,0}
\]
has curvature bidegree
\[
 (p,a)=(a+\ell,a).
\]
The corresponding Tor class therefore has
\[
 i=\ell=p-a,
 \qquad
 j=a+\ell=p.
\]

\begin{theorem}[Dolbeault--Tor grading]
\label{thm:Dolbeault-Tor-grading}
At full rank,
\[
 H^{p,a}(\mathscr K_\Lambda)
 \cong
 \operatorname{Tor}^{S}_{p-a}(\C[K],\C)_p.
\]
\end{theorem}

\begin{proof}
The linear change of polynomial coordinates and the contraction of the
regular sequence preserve internal degree.  The displayed translation is
therefore exactly $i=p-a$ and $j=p$.
\end{proof}

\subsection{The squarefree Hochster calculation}

Give $v_j$ and the corresponding standard Koszul generator $e_j$ the
same $\Z^N$-degree $\mathbf e_j$.  Thus
\[
 \Kos_S(\mathcal A_K)
 =
 \mathcal A_K\otimes\Lambda(e_1,\ldots,e_N),
 \qquad
 d(e_j)=v_j,
\]
decomposes as the algebraic direct sum of its multidegree subcomplexes.

\begin{lemma}[Nonsquarefree contraction]
\label{lem:nonsquarefree-contractible}
If $\alpha_j\geq2$ for some coordinate of
$\alpha\in\NN^N$, then the multidegree-$\alpha$ component of
$\Kos_S(\mathcal A_K)$ is contractible.
\end{lemma}

\begin{proof}
A basis element in multidegree $\alpha$ has the form
$v^\beta e_J$ with
\[
 \beta+\mathbf1_J=\alpha.
\]
Write every $e_J$ in increasing order, and fix $j$ with
$\alpha_j\geq2$.  Define
\[
 h_j(v^\beta e_J)=0
 \quad(j\in J),
\]
and, for $j\notin J$,
\[
 h_j(v^\beta e_J)
 =
 (-1)^{\#\{i\in J:i<j\}}
 v^{\beta-\mathbf e_j}e_{J\cup\{j\}},
\]
where the last exterior monomial is again increasingly ordered.
When $j\notin J$, one has $\beta_j=\alpha_j\geq2$, so division by $v_j$ preserves the support of the monomial.  Thus
$h_j$ is well defined in the Stanley--Reisner quotient.  The usual
Koszul sign calculation gives
\[
 dh_j+h_jd=\operatorname{id}.
\]
Indeed, the term in which $d$ removes the newly inserted $e_j$ is the
identity.  For every $k\in J$, the term that first removes $e_k$ and then
inserts $e_j$ has the opposite sign from the term that first inserts $e_j$
and then removes $e_k$; hence all cross terms cancel.  The displayed sign is the coefficient of the ordered monomial
$e_{J\cup\{j\}}$.  Equivalently, one may write the same homotopy as
$v^{\beta-\mathbf e_j}e_j\wedge e_J$ with coefficient $+1$.
\end{proof}

It follows that Tor is supported in squarefree multidegrees.  For
$I\subseteq[N]$, let $\mathscr C_I$ be the
multidegree-$\mathbf1_I$ component.  It has basis
\[
 v_\sigma e_{I\setminus\sigma},
 \qquad
 \sigma\in K_I.
\]
For the ordered set $I$, define
\[
 \epsilon_I(\sigma)
 =
 \sum_{j\in\sigma}\#\{i\in I:i<j\}.
\]
Let $\sigma^\vee$ denote the reduced cochain dual to the oriented simplex
$\sigma$, including the empty simplex in cochain degree $-1$, and set
\[
 \Phi_I(v_\sigma e_{I\setminus\sigma})
 =
 (-1)^{\epsilon_I(\sigma)}\sigma^\vee.
\]

\begin{theorem}[Squarefree sector theorem]
\label{thm:squarefree-sector}
The map $\Phi_I$ is an isomorphism of complexes after the degree
translation
\[
 t=|I|-i-1.
\]
Consequently,
\[
 H_i(\mathscr C_I)
 \cong
 \widetilde H^{|I|-i-1}(K_I;\C).
\]
\end{theorem}

\begin{proof}
Both complexes have one basis vector for each face $\sigma\in K_I$.  The
Koszul differential replaces
$v_\sigma e_{I\setminus\sigma}$ by the signed sum obtained by moving one
vertex $j\in I\setminus\sigma$ from the exterior factor to the monomial.
The reduced coboundary adds the same vertex to $\sigma$.  If
$j\notin\sigma$, put
\[
 r_I(j)=\#\{i\in I:i<j\},\qquad
 r_\sigma(j)=\#\{i\in\sigma:i<j\}.
\]
The Koszul sign for removing $e_j$ from $e_{I\setminus\sigma}$ is
$(-1)^{r_I(j)-r_\sigma(j)}$, while the reduced-coboundary incidence sign
for adjoining $j$ to $\sigma$ is $(-1)^{r_\sigma(j)}$.  Moreover,
\[
 \epsilon_I(\sigma\cup\{j\})-\epsilon_I(\sigma)
 =
 r_I(j).
\]
Consequently,
\[
 r_I(j)-r_\sigma(j)+\epsilon_I(\sigma\cup\{j\})
 \equiv
 r_\sigma(j)+\epsilon_I(\sigma)
 \pmod2,
\]
so the twist $(-1)^{\epsilon_I(\sigma)}$ intertwines the differentials.
Finally,
\[
 i=|I|-|\sigma|,
 \qquad
 t=|\sigma|-1=|I|-i-1.
\]
\end{proof}

For $I=\varnothing$, we use the augmented convention
\[
 K_\varnothing=\{\varnothing\},
 \qquad
 \widetilde H^{-1}(K_\varnothing;\C)=\C.
\]
Here the multidegree-zero Koszul sector is the single copy of $\C$ in
homological degree zero, so the augmented convention agrees with the chain map.

\begin{theorem}[Hochster formula \cite{Hochster1977}]
\label{thm:Hochster}
For every $I\subseteq[N]$,
\[
 \operatorname{Tor}^S_i(\C[K],\C)_{\mathbf1_I}
 \cong
 \widetilde H^{|I|-i-1}(K_I;\C).
\]
All nonsquarefree multidegrees vanish.  Therefore
\[
 \operatorname{Tor}^S_i(\C[K],\C)_j
 \cong
 \bigoplus_{|I|=j}
 \widetilde H^{j-i-1}(K_I;\C).
\]
\end{theorem}

\begin{proof}
The first assertion is Theorem~\ref{thm:squarefree-sector}; the
nonsquarefree assertion is
Lemma~\ref{lem:nonsquarefree-contractible}.  Summing the squarefree
multidegrees with $|I|=j$ gives the internal-degree formula.
This agrees with
\cite[Theorem~3.2.9]{BuchstaberPanov2015} after halving its internal
multigrading, $2\mathbf1_I\mapsto\mathbf1_I$, and hence $2j\mapsto j$.
\end{proof}

\subsection{The Dolbeault bidegrees}

Substituting $i=p-a$ and $j=p$ into Hochster's formula gives the complete
curvature calculation.

\begin{theorem}[Curvature--Hochster decomposition]
\label{thm:curvature-Hochster}
At full rank,
\[
 H^{p,a}(\mathscr K_\Lambda)
 \cong
 \bigoplus_{|I|=p}
 \widetilde H^{a-1}(K_I;\C).
\]
\end{theorem}

\begin{proof}
Combine Theorem~\ref{thm:Dolbeault-Tor-grading} with
Theorem~\ref{thm:Hochster}:
\[
 |I|-(p-a)-1=p-(p-a)-1=a-1.
\]
\end{proof}

\begin{theorem}[Dolbeault--Hochster theorem: first proof]
\label{thm:Dolbeault-Hochster-first}
Let $\XL$ be minimally stable and polytopal with $d\geq1$.  If $B$ is
invertible, then for $0\leq p,q\leq N$,
\[
 H^q(\XL,\Omega_{\XL}^p)
 \cong
 \bigoplus_{\substack{I\subseteq[N],\ |I|=p\\
                       0\leq a\leq d,\ 0\leq s\leq m\\
                       a+s=q}}
 \widetilde H^{a-1}(K_I;\C)
 \otimes
 \Lambda^s\GammaLCdual.
\]
Consequently,
\[
 h^{p,q}(\XL)
 =
 \sum_{|I|=p}
 \sum_{a=0}^{d}
 \widetilde b_{a-1}(K_I)
 \binom{m}{q-a}.
\]
\end{theorem}

\begin{proof}
Theorem~\ref{thm:universal-Hodge-formula} gives
\[
 H^q(\XL,\Omega_{\XL}^p)
 \cong
 \bigoplus_{s=0}^{m}
 H^{p,q-s}(\mathscr K_\Lambda)
 \otimes
 \Lambda^sW_\Lambda^{0,1}.
\]
Apply Theorem~\ref{thm:curvature-Hochster} with $a=q-s$, and then use
Proposition~\ref{prop:deck-Cartan-identification} to identify
$W_\Lambda^{0,1}$ with $\GammaLCdual$.  Taking dimensions gives the
second formula.
\end{proof}

\begin{corollary}[Dolbeault-polynomial factorization]
\label{cor:Dolbeault-polynomial-first}
At full rank,
\[
 \operatorname{Dol}_{\XL}(u,v)
 =
 (1+v)^m\mathscr H_K(u,v).
\]
\end{corollary}

\begin{proof}
Sum the Hodge-number formula over $p$ and $q$, and use
\[
 \sum_{s=0}^{m}\binom ms v^s=(1+v)^m.
\]
\end{proof}

\begin{theorem}[Equivalent full-rank conditions]
\label{thm:exact-full-rank-criterion}
For a minimally stable polytopal LVM manifold, the following are
equivalent:
\begin{enumerate}[label=\textup{(\roman*)}]
\item $B$ is invertible, equivalently $\delta=0$;
\item $c_\Lambda$ is an isomorphism;
\item $H^0(\XL,\Omega_{\XL}^1)=0$;
\item $\operatorname{Dol}_{\XL}(u,0)=1$;
\item
\(
 h^{1,0}(\XL)+h^{0,1}(\XL)=b_1(\XL);
\)
\item the polynomial identity
\[
 \operatorname{Dol}_{\XL}(u,v)
 =(1+v)^m\mathscr H_K(u,v)
\]
holds;
\item the numerical identities
\[
 h^{p,q}(\XL)
 =
 \sum_{|I|=p}
 \sum_{a=0}^{d}
 \widetilde b_{a-1}(K_I)
 \binom{m}{q-a}
 \qquad(0\leq p,q\leq N)
\]
hold;
\item the Fr\"olicher spectral sequence degenerates at $E_1$;
\item $\mathfrak F(\Lambda)=0$.
\end{enumerate}
When $\delta>0$, the identity $h^{1,0}=\delta$ detects the discrepancy with the Dolbeault--Hochster formula~\eqref{eq:intro-main-formula}, and $h^{1,0}+h^{0,1}=m+\delta>b_1$ shows that the Fr\"olicher spectral sequence does not degenerate at $E_1$.  Moreover,
\[
 \delta
 =\operatorname{ord}_{u=-1}\operatorname{Dol}_{\XL}(u,v).
\]
\end{theorem}

\begin{proof}
The equivalence of (i) and (ii) is
Corollary~\ref{cor:curvature-full-rank-criterion}.  Theorem~\ref{thm:holomorphic-forms-defect}
gives
\[
 H^0(\XL,\Omega_{\XL}^1)\cong\C^\delta,
\]
and Corollary~\ref{cor:exact-defect-order} gives (iv) and the final order
formula.  Also
\[
 h^{1,0}+h^{0,1}=\delta+m,
 \qquad
 b_1(\XL)=m
\]
by Lemma~\ref{lem:all-rank-b1}, proving the equivalence with (v).

If $B$ is invertible, Theorem~\ref{thm:Dolbeault-Hochster-first} gives
(vi) and (vii).
Conversely, either (vi) or (vii) gives zero in bidegree
$(1,0)$, because every one-vertex induced subcomplex is a point; hence
(iii) follows.  This proves the equivalence of (i)--(vii).

Conditions (viii) and (ix) are equivalent by
Theorem~\ref{thm:curvature-Frolicher-criterion}.  Theorem~\ref{thm:singular-Frolicher-obstruction}
shows that (viii) forces $\delta=0$.  Conversely, at full rank the
classical moment-angle Tor theorem and Hochster formula give
\[
 \operatorname{Poin}_{\mathcal Z_K}(t)=\mathscr H_K(t,t)
\]
\cite{BaskakovBuchstaberPanov2004,Panov2008,BuchstaberPanov2015}.
Together with the polynomial identity (vi) and the homeomorphism at every admissible rank
$\XL\cong\mathcal Z_K\times T^m$, this gives
\[
 \operatorname{Dol}_{\XL}(t,t)
 =\operatorname{Poin}_{\XL}(t).
\]
Thus the total-degree dimensions of $E_1$ and $E_\infty$ agree, and the
Fr\"olicher spectral sequence degenerates at $E_1$.
\end{proof}

\begin{remark}[Relation with the classical $E_2$ bound]
For the holomorphic principal torus bundles considered by Panov--Ustinovsky, the characteristic
differential in Panov--Ustinovsky's model is the curvature map
$c_\Lambda$ used here.  Their Corollary~5.8 gives the bound $E_2=E_\infty$
\cite[Corollary~5.8]{PanovUstinovsky2012}.  For these bundles, the $E_2$ bound and the inequality
$h^{1,0}+h^{0,1}=m+\delta>b_1=m$ imply that positive defect forces a nonzero
$d_1$ in total degree one. At full rank,
Theorem~\ref{thm:exact-full-rank-criterion} gives $E_1=E_\infty$.
\end{remark}

\Needspace{12\baselineskip}
\part{Holomorphic descent}
The second proof starts from a different question: after expanding
holomorphic forms on the Cox charts into Laurent sectors, which sectors
survive deck descent?
Its inputs are the quotient geometry of Part~I, Laurent expansions,
and the group--Koszul resolution of deck invariants.

Throughout this part the minimally stable polytopal configuration is
full rank, so
\[
 \pi:U_K\longrightarrow \XL,
 \qquad
 \GammaL=B^{-1}(2\pi i\Z^m)\cong\Z^m,
 \qquad
 m=N-d.
\]
We write $z_j$ for the visible affine Cox coordinate denoted $w_j$ in
Part~I.  The covering map is
\[
 q=\pi:U_K\longrightarrow\XL,
 \qquad z\longmapsto[(z,\mathbf1)],
\]
and, because it is a holomorphic covering,
\(
 \pi^*\Omega_{\XL}^p\cong\Omega_{U_K}^p.
\)
For a maximal face $\sigma\in K$, the Cox chart $U_\sigma$ allows the
coordinates indexed by $\sigma$ to vanish and requires every other
coordinate to be nonzero.  These charts form the finite Stein cover
$\mathcal U_K$.  On an intersection, a Laurent sector fixes the exponent
$r\in\Z^N$ and the logarithmic differential
$e_I=\bigwedge_{j\in I}dz_j/z_j$, with $I$ written in increasing order;
regularity determines whether $z^re_I$ occurs on that intersection.
Fix $0\leq p\leq N$ and put
\[
 C_p^\bullet
 =
 \check C^\bullet(\mathcal U_K,\Omega_{U_K}^p).
\]
Fix a total order on the finite set $K_{\max}$ and use throughout the
normalized alternating \v{C}ech complex: its degree-$a$ component is indexed
by strictly increasing tuples
\(
 \sigma_0<\cdots<\sigma_a.
\)
Thus $C_p^\bullet$ is bounded, and its signs are the ordinary oriented-simplex
signs.  Every section space carries its compact-open Fr\'echet topology, and
every finite \v{C}ech degree carries the corresponding finite-product topology.
The fixed-exponent contractions apply to the completed descent complex through
the finite-support quasi-isomorphism of
Theorem~\ref{thm:Euler-finite-support}. Each class then involves only finitely many character differences, so its contraction uses only finitely many scalar inverses.  The comparisons are
\[
 \begin{aligned}
 R\Gamma(\XL,\Omega_{\XL}^p)
 &\simeq\operatorname{Tot}\Kos_{\GammaL}^\bullet(C_p^\bullet)
 \simeq\bigoplus_{r\in\Z^N}^{\mathrm{alg}}\mathcal T_{p,r}^\bullet\\
 &\simeq\mathcal T_{p,0}^\bullet
 \simeq\bigoplus_{|I|=p}\widetilde C^{\bullet-1}(K_I;\C)
 \otimes\Lambda^\bullet\GammaLCdual.
 \end{aligned}
\]

\section{Cox--Čech sectors and analytic completion}
\label{sec:Cox-Cech-sectors}
The Cox charts and their intersections are products of copies of
$\C$ and $\C^*$.  Taylor--Laurent expansion therefore describes their
holomorphic forms explicitly.  We compute each fixed-exponent
complex and then specify the topology in which these expansions
converge. For the Cox construction and these charts, see
\cite{Cox1995,BuchstaberPanov2015,Panov2025}.

For $\tau\in K$, define
\[
 U_\tau
 =
 \{z\in U_K:z_j\neq0\text{ for }j\notin\tau\}.
\]

\begin{proposition}[Finite Stein--Leray cover]
\label{prop:finite-Stein-Leray-cover}
\[
 U_\tau
 \cong
 \C^{|\tau|}\times(\C^*)^{N-|\tau|}.
\]
The maximal faces give a finite Stein cover
\[
 \mathcal U_K=\{U_\sigma\}_{\sigma\in K_{\max}},
\]
and
\[
 U_{\sigma_0}\cap\cdots\cap U_{\sigma_a}
 =
 U_{\sigma_0\cap\cdots\cap\sigma_a}.
\]
\end{proposition}

\begin{proof}
Because $\tau$ is a face, every subset of $\tau$ is a face.  Hence on
$U_\tau$ the coordinates indexed by $\tau$ are arbitrary and all remaining
coordinates are nonzero, which gives the displayed product biholomorphism.
If $z\in U_K$, then $Z(z)\in K$ and is contained in a maximal face, so the
maximal charts cover.  The intersection formula follows directly from the
nonvanishing conditions.  Every nonempty finite intersection is therefore a
product of copies of $\C$ and $\C^*$, hence Stein.
\end{proof}

\begin{corollary}[Topological \v{C}ech model]
\label{cor:topological-Cech-model}
For $0\leq p\leq N$,
\[
 R\Gamma(U_K,\Omega_{U_K}^p)
 \simeq
 \check C^\bullet(\mathcal U_K,\Omega_{U_K}^p).
\]
The right-hand side is a bounded complex of Fr\'echet spaces with continuous
differential.
\end{corollary}

\begin{proof}
The sheaf $\Omega_{U_K}^p$ is locally free and therefore coherent.  Cartan's
theorem~B \cite{Cartan1953} makes it acyclic on every nonempty intersection
in Proposition~\ref{prop:finite-Stein-Leray-cover}; the finite Leray theorem
then gives the quasi-isomorphism.  Holomorphic restriction maps are continuous
for the compact-open topologies.  The finite cover and the normalized \v{C}ech convention give a bounded complex, each of whose terms is a finite product of Fr\'echet spaces.
\end{proof}

\subsection{Logarithmic Laurent expansions}

For
\[
 I=\{i_1<\cdots<i_p\},
\]
set
\[
 e_I
 =
 \bigwedge_{i\in I}\frac{dz_i}{z_i},
 \qquad
 z^re_I=z^{r-\mathbf1_I}dz_I.
\]
Here $e_I$ is a meromorphic logarithmic form on $U_K$, holomorphic on
the dense torus; the second equality is the definition of the individual
term on a chart.  All coordinates in this part are the $N$ visible Cox
coordinates.  Indispensable labels enter through the action and its deck lattice;
Laurent exponents are indexed by the visible coordinates.

\begin{lemma}[Regularity criterion]
\label{lem:Cox-regularity}
The term $z^re_I$ is holomorphic on $U_\tau$ if and only if
\[
 r_j-\mathbf1_I(j)\geq0
 \qquad(j\in\tau).
\]
\end{lemma}

\begin{proof}
In the standard frame $dz_I$, the coefficient is the Laurent monomial
$z^{r-\mathbf1_I}$.  Exactly the coordinate divisors indexed by $\tau$ meet
$U_\tau$, so their exponents must be nonnegative; the inverted coordinates
have arbitrary integral exponents.
\end{proof}

Let $\mathscr P_{\mathrm{vis}}^p(U_\tau)$ be the \v{C}ech coefficient space of
normally convergent
series
\[
 \sum_{\substack{r\in\Z^N\\|I|=p}}
 a_{r,I}z^re_I
\]
satisfying the inequalities of
Lemma~\ref{lem:Cox-regularity}.  It carries the topology of uniform
convergence of the coefficient functions in the standard frames $dz_I$ on
compact subsets.  For $\tau'\subseteq\tau$, restriction to
$U_{\tau'}\subseteq U_\tau$ preserves the coefficients.  The coefficient spaces and restriction maps define
$\mathscr P_{\mathrm{vis}}^p$ on this finite diagram of chart intersections.

\begin{proposition}[Logarithmic Laurent coordinates]
\label{prop:logarithmic-Laurent-coordinates}
Rewriting $dz_I=z_Ie_I$ gives a natural topological isomorphism
\[
 \Gamma(U_\tau,\Omega_{U_K}^p)
 \cong
 \mathscr P_{\mathrm{vis}}^p(U_\tau),
\]
compatible with restrictions. Hence
\[
 \check C^\bullet(\mathcal U_K,\Omega_{U_K}^p)
 \cong
 \check C^\bullet(\mathcal U_K,\mathscr P_{\mathrm{vis}}^p).
\]
\end{proposition}

\begin{proof}
Write a form uniquely as $\sum_{|I|=p}f_I(z)\,dz_I$.  On
$\C^{|\tau|}\times(\C^*)^{N-|\tau|}$, the multivariable Taylor--Laurent
theorem expands each $f_I$ uniquely in monomials $z^{r-\mathbf1_I}$,
normally on compact subsets, with precisely the inequalities in
Lemma~\ref{lem:Cox-regularity}.  Conversely, normal convergence gives a
holomorphic coefficient in every standard frame.  Cauchy coefficient
estimates and uniform convergence on smaller compacta make extraction and
reconstruction continuous inverses.  Uniqueness on the dense torus and
coefficientwise restriction prove compatibility with the \v{C}ech maps.
\end{proof}

\subsection{Fixed Laurent sectors}

For $r\in\Z^N$ and $|I|=p$, define
\[
 \Pole(r,I)
 =
 \{j:r_j<\mathbf1_I(j)\}.
\]
Then
\[
 z^re_I\text{ is regular on }U_\tau
 \iff
 \Pole(r,I)\cap\tau=\varnothing.
\]

Let $\mathscr P_{r,I}^\bullet$ be the coefficient subcomplex carried by
the single term $z^re_I$.

Let $\Delta_{\max}$ be the full simplex on the totally ordered vertex set
$K_{\max}$.  We use normalized simplicial cochains with the induced
orientation.  For $T\subseteq[N]$, define the pole subcomplex
\[
 [\sigma_0,\ldots,\sigma_a]\in\mathfrak D_T
 \iff
 T\cap(\sigma_0\cap\cdots\cap\sigma_a)\neq\varnothing.
\]

\begin{theorem}[Relative simplex model]
\label{thm:relative-sector}
There is a natural isomorphism
\[
 \mathscr P_{r,I}^\bullet
 \cong
 C^\bullet
 (\Delta_{\max},\mathfrak D_{\Pole(r,I)};\C).
\]
\end{theorem}

\begin{proof}
The coefficient line is present on a Čech simplex precisely when
$z^re_I$ is regular on its chart intersection, equivalently when that
simplex is outside $\mathfrak D_{\Pole(r,I)}$.  The set
$\mathfrak D_{\Pole(r,I)}$ is a subcomplex: if a label belongs to the
intersection of a tuple, it belongs to the intersection of each face of
that tuple.  Restrictions preserve the monomial and the normalized
\v{C}ech signs are exactly the oriented-simplex signs.  Declaring the
coefficient to be zero on the pole subcomplex therefore gives an isomorphism of complexes.
\end{proof}

For $j\in T$, let $\mathfrak D_j$ be the simplex on the maximal faces
containing $j$.

\begin{lemma}
For nonempty $J\subseteq T$,
\[
 \bigcap_{j\in J}\mathfrak D_j
\]
is nonempty if and only if $J\in K$, and in that case it is a simplex.
\end{lemma}

\begin{proof}
Its vertices are the maximal faces containing $J$.
\end{proof}

\begin{theorem}[Hochster nerve]
\label{thm:Hochster-nerve}
For nonempty $T$,
\[
 \mathfrak D_T=\bigcup_{j\in T}\mathfrak D_j
\]
 is a simplicial good cover with nerve $K_T$. Hence
\[
 \mathfrak D_T\simeq K_T.
\]
\end{theorem}

\begin{proof}
A simplex of $\Delta_{\max}$ lies in $\mathfrak D_T$ exactly when all its
vertices contain some common $j\in T$, so
$\mathfrak D_T=\bigcup_{j\in T}\mathfrak D_j$.  Every nonempty finite
intersection $\bigcap_{j\in J}\mathfrak D_j$ is the full simplex on the
maximal faces containing $J$, by the preceding lemma.  Thus this is a
finite cover by simplicial subcomplexes with all nonempty intersections
contractible.  The simplicial nerve lemma identifies its realization with
the nerve.  A subset $J\subseteq T$ spans that nerve exactly when $J\in K$,
so the nerve is the induced complex $K_T$.
\end{proof}

\begin{theorem}[Fixed Laurent-sector theorem]
\label{thm:fixed-Laurent-sector}
For every $r\in\Z^N$ and $|I|=p$,
\[
 H^a(\mathscr P_{r,I}^\bullet)
 \cong
 \widetilde H^{a-1}(K_{\Pole(r,I)};\C).
\]
\end{theorem}

\begin{proof}
If $\Pole(r,I)\neq\varnothing$, combine
Theorems~\ref{thm:relative-sector} and~\ref{thm:Hochster-nerve} with the
connecting isomorphism in the long exact sequence of the pair
$(\Delta_{\max},\mathfrak D_{\Pole(r,I)})$; the ambient simplex is contractible and the pole subcomplex is nonempty.  If
$\Pole(r,I)=\varnothing$, then $\mathfrak D_\varnothing$ has no nonempty
simplices (it is $\{\varnothing\}$ in the augmented convention), and the
nonaugmented relative complex is the cochain complex of
$\Delta_{\max}$.  Our reduced-cohomology convention is
\[
 K_\varnothing=\{\varnothing\},
 \qquad
 \widetilde H^{-1}(K_\varnothing)=\C.
\]
\end{proof}

\begin{corollary}
\label{cor:zero-sector-cohomology}
For $r=0$,
\[
 \Pole(0,I)=I,
\]
and therefore
\[
 H^a(\mathscr P_{0,I}^\bullet)
 \cong
 \widetilde H^{a-1}(K_I;\C).
\]
\end{corollary}

\subsection{Deck characters}

Every deck transformation is diagonal with nonzero coordinate multipliers;
it preserves zero sets and hence preserves each chart $U_\tau$ and every
intersection in $\mathcal U_K$.

Recall that $A_j$ is the $j$th row of the visible block $A$ in
Section~\ref{sec:affine-ghost-quotient}.  For $r\in\Z^N$, set
\[
 L_r(T)=\sum_{j=1}^Nr_jA_j(T),
 \qquad
 \chi_r(\gamma)=e^{L_r(\gamma)}.
\]

\begin{proposition}
The deck action on $\mathscr P_{r,I}^\bullet$ is through the character $\chi_r$:
\[
 \gamma^*(z^re_I)
 =
 \chi_r(\gamma)z^re_I.
\]
\end{proposition}

\begin{proof}
In the standard frame, diagonal scaling gives
\[
 \gamma^*\!\left(z^{r-\mathbf1_I}dz_I\right)
 =
 e^{\sum_j(r_j-\mathbf1_I(j))A_j(\gamma)}
 e^{\sum_{i\in I}A_i(\gamma)}
 z^{r-\mathbf1_I}dz_I
 =
 \chi_r(\gamma)z^re_I.
\]
Equivalently, $z^r$ has character $\chi_r$ and every logarithmic form
$dz_j/z_j$ is invariant.
\end{proof}

\begin{remark}
The analytic \v{C}ech complex is the completion of the algebraic
sum of the $\mathscr P_{r,I}^\bullet$ for normal convergence of
Laurent series.  We describe this completion below.
\end{remark}

\subsection{Analytic completion and unconditional sector projections}
\label{sec:analytic-Laurent-completion}

For $\tau\in K$ and $R\geq2$, define
\[
 \mathfrak A_{\tau,R}
 =
 \left\{
 z\in U_\tau:
 \begin{array}{ll}
 |z_j|\leq R,&j\in\tau,\\
 R^{-1}\leq|z_j|\leq R,&j\notin\tau
 \end{array}
 \right\}.
\]
The integer-radius family $\{\mathfrak A_{\tau,R}\}_{R\in\Z,\,R\geq2}$
is a compact exhaustion of $U_\tau$.  If
\(
 \omega=\sum_{|I|=p}f_I\,dz_I,
\)
put
\[
 M_{\tau,R}(\omega)
 =
 \sum_{|I|=p}\sup_{z\in\mathfrak A_{\tau,R}}|f_I(z)|.
\]
These seminorms generate the compact-open Fr\'echet topology on
$\Gamma(U_\tau,\Omega_{U_K}^p)$.

For an allowed pair $(r,I)$, define the shifted exponent
\[
 \widehat r^{\,I}=r-\mathbf1_I
\]
and the weight
\[
 \omega_{\tau,R}(r,I)
 =
 R^{\nu_\tau(r,I)},
\]
where
\[
 \nu_\tau(r,I)
 =
 \sum_{j\in\tau}\widehat r_j^{\,I}
 +
 \sum_{j\notin\tau}|\widehat r_j^{\,I}|.
\]

Let $\mathfrak E_{\tau,p}$ be the set of pairs with $|I|=p$ and
$\widehat r_j^{\,I}\geq0$ for $j\in\tau$.

\begin{definition}
The local analytic coefficient space is
\[
 \lambda_{\tau,p}^{\mathrm{an}}
 =
 \left\{
 (a_{r,I}):
 q_{\tau,R}(a)
 =
 \sum_{(r,I)\in\mathfrak E_{\tau,p}}
 |a_{r,I}|\omega_{\tau,R}(r,I)
 <\infty
 \text{ for every }R
 \right\}.
\]
It suffices to take integer $R\geq2$. With these countably many norms, this
space is the projective limit of the Banach spaces
\(
 \ell^1(\mathfrak E_{\tau,p},\omega_{\tau,R})
\)
and hence is a complete Fr\'echet space.
\end{definition}

\begin{lemma}\label{lem:weight-estimates}
If $2\leq R<L$, then
\[
 \sum_{(r,I)\in\mathfrak E_{\tau,p}}
 \frac{\omega_{\tau,R}(r,I)}
 {\omega_{\tau,L}(r,I)}
 <\infty.
\]
Moreover, for every $j$ there is $C_{R,L}$ such that
\[
 (|r_j|+1)\omega_{\tau,R}(r,I)
 \leq
 C_{R,L}\omega_{\tau,L}(r,I).
\]
\end{lemma}

\begin{proof}
Write $t=R/L<1$.  Translation $r\mapsto\widehat r^{\,I}$ identifies the
allowed exponents, for each $I$, with
$\Z_{\geq0}^{\tau}\times\Z^{[N]\setminus\tau}$.  Consequently the first
sum is exactly
\[
 \binom Np
 \left(\frac1{1-t}\right)^{|\tau|}
 \left(\frac{1+t}{1-t}\right)^{N-|\tau|},
\]
which is finite.  Moreover $|r_j|+1\leq
|\widehat r_j^{\,I}|+2$ and
$\nu_\tau(r,I)\geq|\widehat r_j^{\,I}|$.  Therefore
\[
 (|r_j|+1)t^{\nu_\tau(r,I)}
 \leq
 \sup_{n\geq0}(n+2)t^n<\infty,
\]
which is precisely the second estimate after multiplication by
$\omega_{\tau,L}(r,I)$.
\end{proof}

\begin{lemma}[Nuclear K\"othe coefficient space]
\label{lem:nuclear-Kothe-space}
The Fr\'echet space $\lambda_{\tau,p}^{\mathrm{an}}$ is nuclear.
\end{lemma}

\begin{proof}
For integer $2\leq R<L$, the diagonal inclusion
\[
 \ell^1(\mathfrak E_{\tau,p},\omega_{\tau,L})
 \longrightarrow
 \ell^1(\mathfrak E_{\tau,p},\omega_{\tau,R})
\]
is nuclear: in the normalized coordinate bases its nuclear norm is at
most
\[
 \sum_{(r,I)}\omega_{\tau,R}(r,I)/\omega_{\tau,L}(r,I),
\]
finite by Lemma~\ref{lem:weight-estimates}.  The projective limit is
therefore a nuclear Fr\'echet space by the K\"othe--Grothendieck
criterion \cite{Grothendieck1955}.
\end{proof}

\begin{theorem}[Local Köthe--Laurent theorem]
\label{thm:local-Kothe-Laurent}
Coefficient extraction and reconstruction give inverse topological
isomorphisms
\[
 \Gamma(U_\tau,\Omega_{U_K}^p)
 \cong
 \lambda_{\tau,p}^{\mathrm{an}}.
\]
\end{theorem}

\begin{proof}
If $a\in\lambda_{\tau,p}^{\mathrm{an}}$, then
\[
 \sum_{r,I}|a_{r,I}z^{\widehat r^{\,I}}|
 \leq
 q_{\tau,R}(a)
\]
coefficientwise on $\mathfrak A_{\tau,R}$, so reconstruction in the
standard frames $dz_I$ converges normally and
\[
 M_{\tau,R}(\omega)\leq q_{\tau,R}(a).
\]

Conversely, choose $L>R$ and apply the product Cauchy formula to the
Taylor variables and, according to the sign of each exponent, the outer
or inner circle in every Laurent variable.  All these integration tori
lie in $\mathfrak A_{\tau,L}$, and
\[
 |a_{r,I}|
 \leq
 M_{\tau,L}(\omega)
 \omega_{\tau,L}(r,I)^{-1}.
\]
Thus
\[
 q_{\tau,R}(a)
 \leq
 M_{\tau,L}(\omega)\cdot
 \sum_{(r,I)\in\mathfrak E_{\tau,p}}
 \frac{\omega_{\tau,R}(r,I)}{\omega_{\tau,L}(r,I)}.
\]
Lemma~\ref{lem:weight-estimates} proves membership and continuity of
coefficient extraction.  The first inequality proves continuity of
reconstruction, while uniqueness of Taylor--Laurent coefficients makes
the maps inverse.
\end{proof}

If $\tau'\subseteq\tau$, restriction preserves every coefficient and
\[
 \omega_{\tau',R}(r,I)=\omega_{\tau,R}(r,I)
\]
for pairs allowed on $U_\tau$.

Let $\mathfrak S_a$ be the finite set of normalized \v{C}ech simplices
$\boldsymbol\sigma=(\sigma_0<\cdots<\sigma_a)$ and set
\[
 \tau(\boldsymbol\sigma)
 =
 \bigcap_i\sigma_i.
\]
For
\[
 c_{\boldsymbol\sigma}
 =
 \sum_{r,I}c_{\boldsymbol\sigma;r,I}z^re_I,
\]
define
\[
 \mathfrak q_{a,R}(c)
 =
 \sum_{\boldsymbol\sigma\in\mathfrak S_a}
 \sum_{(r,I)\in\mathfrak E_{\tau(\boldsymbol\sigma),p}}
 |c_{\boldsymbol\sigma;r,I}|
 \omega_{\tau(\boldsymbol\sigma),R}(r,I).
\]

\begin{definition}
The analytically completed sector sum
\[
 \widehat{\bigoplus}_{r,I}^{\,\mathrm{an}}\mathscr P_{r,I}^a
\]
is the space of sector families with finite $\mathfrak q_{a,R}$ for every $R$.
\end{definition}

\begin{theorem}[Global analytic sector decomposition]
\label{thm:global-analytic-sector}
For every $a$,
\[
 C_p^a
 \cong
 \widehat{\bigoplus}_{r,I}^{\,\mathrm{an}}\mathscr P_{r,I}^a
\]
as Fréchet spaces. Compatibility with restriction gives
\[
 C_p^\bullet
 \cong
 \widehat{\bigoplus}_{r,I}^{\,\mathrm{an}}\mathscr P_{r,I}^\bullet.
\]
\end{theorem}

\begin{proof}
Apply
Theorem~\ref{thm:local-Kothe-Laurent}
to the finitely many Čech intersections. Restrictions are coefficientwise
and continuous.  Indeed, if $\tau'\subseteq\tau$, regularity on $U_\tau$
gives $\widehat r_j^{\,I}\geq0$ on $\tau\setminus\tau'$, so replacing
$\widehat r_j^{\,I}$ by its absolute value does not change the weight.
Thus
\[
 \mathfrak q_{a+1,R}(\delta c)
 \leq |K_{\max}|\,\mathfrak q_{a,R}(c).
\]
The local factors are nuclear by
Lemma~\ref{lem:nuclear-Kothe-space}, and finite products preserve
nuclearity.  The normalized complex vanishes in degrees
$a\geq|K_{\max}|$.
\end{proof}

\begin{theorem}[Unconditional sector projections]
\label{thm:unconditional-sector-projections}
For any set of sectors $\mathcal E$, the coefficient projection
$\Pi_\mathcal E$ is continuous and
\[
 \mathfrak q_{a,R}(\Pi_\mathcal E c)\leq \mathfrak q_{a,R}(c).
\]
It commutes with the Čech differential.
\end{theorem}

\begin{proof}
Deleting coefficients can only decrease every positive weighted
$\ell^1$ seminorm.  Since restrictions preserve the pair $(r,I)$,
coefficient deletion commutes with every alternating restriction sum.
\end{proof}

\begin{corollary}
The algebraic finite-sector subcomplex
\[
 C_{p,\mathrm{fin}}^\bullet
 =
 \bigoplus_{r,I}^{\mathrm{alg}}\mathscr P_{r,I}^\bullet
\]
is dense in $C_p^\bullet$.
\end{corollary}

\begin{proof}
Enumerate the countable set of pairs $(r,I)$ and truncate to its first
$n$ elements, simultaneously in every one of the finitely many \v{C}ech
components.  For each $a$ and $R$, the discarded tail tends to zero in
the weighted $\ell^1$ seminorm $\mathfrak q_{a,R}$.  These truncations
belong to the algebraic finite-sector subcomplex.
\end{proof}

\begin{remark}
Finite-sector cochains are dense in the \v{C}ech complex. After deck
descent, Theorem~\ref{thm:Euler-finite-support} shows that they also compute
cohomology in the corresponding group--Koszul total complexes.
\end{remark}

Let
\[
 \mathcal L_j
 =
 \mathcal L_{z_j\partial/\partial z_j}.
\]
Then
\[
 \mathcal L_j(z^re_I)=r_jz^re_I.
\]

\begin{proposition}
For $R<L$,
\[
 \mathfrak q_{a,R}(\mathcal L_jc)
 \leq
 C_{R,L}\mathfrak q_{a,L}(c).
\]
Hence every polynomial in the Euler operators is continuous.
\end{proposition}

\begin{proof}
On a sector the multiplier is $r_j$.  Summing the second estimate of
Lemma~\ref{lem:weight-estimates} over the finitely many \v{C}ech
components gives the displayed bound.  Iteration gives the estimate for products of Euler operators, and addition
gives it for polynomials.
\end{proof}

For $\gamma\in\GammaL$, define
\[
 M_\gamma
 =
 \max\left(
 1,e^{|\operatorname{Re}A_1(\gamma)|},\ldots,
 e^{|\operatorname{Re}A_N(\gamma)|}
 \right).
\]

\begin{proposition}
The deck action is continuous and
\[
 \mathfrak q_{a,R}(\gamma^*c)
 \leq
 M_\gamma^p
 \mathfrak q_{a,M_\gamma R}(c).
\]
\end{proposition}

\begin{proof}
On sector $(r,I)$, the action is multiplication by $\chi_r(\gamma)$.
For an allowed pair on $U_\tau$,
\(
 \sum_j|r_j|\leq p+\nu_\tau(r,I).
\)
Hence
\[
 |\chi_r(\gamma)|
 \leq
 M_\gamma^pM_\gamma^{\nu_\tau(r,I)}.
\]
The estimate follows by increasing the weight radius.
\end{proof}

\begin{proposition}[Continuous \v{C}ech product]
\label{prop:continuous-Cech-product}
The Alexander--Whitney \v{C}ech product followed by wedge,
\[
 C_p^a\times C_{p'}^b\longrightarrow C_{p+p'}^{a+b},
\]
is jointly continuous and satisfies
\[
 \mathfrak q_{a+b,R}(c\smile d)
 \leq
 \mathfrak q_{a,R}(c)\,\mathfrak q_{b,R}(d).
\]
On individual sectors, the product is zero unless $I\cap J=\varnothing$;
in that case
\[
 z^re_I\wedge z^se_J
 =
 (-1)^{\operatorname{inv}(I,J)}z^{r+s}e_{I\cup J},
\]
where
\(
 \operatorname{inv}(I,J)=\#\{(i,j)\in I\times J:i>j\},
\)
and $\chi_{r+s}=\chi_r\chi_s$.  Moreover,
\[
 \omega_{\tau,R}(r+s,I\cup J)
 \leq
 \omega_{\tau,R}(r,I)
 \omega_{\tau,R}(s,J).
\]
\end{proposition}

\begin{proof}
When $I\cap J=\varnothing$, the shifted exponent of the product is
\[
 \widehat{r+s}^{\,I\cup J}
 =
 \widehat r^{\,I}+\widehat s^{\,J}.
\]
The triangle inequality in every inverted coordinate, and equality in
every noninverted coordinate, give the displayed weight inequality.
Weighted $\ell^1$ convolution therefore has norm at most the product of
the two norms.  For the Alexander--Whitney product, restriction from each
front or back intersection to the total intersection preserves its
weight.  Summing over the normalized simplices gives the global estimate.
\end{proof}

\begin{theorem}[Completed Laurent-sector theorem]
\label{thm:completed-Laurent-sector}
The Cox--Čech complex admits the unconditional analytic decomposition
\[
 \check C^\bullet(\mathcal U_K,\Omega_{U_K}^p)
 \cong
 \widehat{\bigoplus}_{\substack{r\in\Z^N\\|I|=p}}^{\,\mathrm{an}}
 \mathscr P_{r,I}^\bullet
\]
as a bounded complex of nuclear Fréchet spaces. Coefficient extraction,
reconstruction, sector projections, Čech restrictions, Euler operators,
deck transformations, and multiplication are continuous.
\end{theorem}

\begin{proof}
Combine Corollary~\ref{cor:topological-Cech-model},
Theorems~\ref{thm:global-analytic-sector} and
\ref{thm:unconditional-sector-projections}, and
Proposition~\ref{prop:continuous-Cech-product}, together with the Euler
and deck estimates above.  Nuclearity and boundedness were proved in
Theorem~\ref{thm:global-analytic-sector}.
\end{proof}

The finite-support quasi-isomorphism after deck descent is proved in
Theorem~\ref{thm:Euler-finite-support}.

For each fixed exponent, cohomology is now expressed by an induced
complex.  We next determine the pole complexes of resonant exponents.

\section{Resonance and convex vanishing}
\label{sec:resonance-Gale}

We now compute the sectors whose deck character is trivial.  Such
an exponent extends to an integral affine relation among the weights.
A positive Gale normalization turns that relation into a height on a
convex realization of $K$.  The pole set consists of the vertices of negative height together with selected vertices of height zero. The induced subcomplex is contractible when the resonant exponent is nonzero.

Write the indispensable labels as
\[
 D=\{a_0,a_1,\ldots,a_m\}.
\]
For $r\in\Z^N$, set
\[
 L_r=\sum_{j=1}^Nr_jA_j,
 \qquad
 \chi_r(\gamma)=e^{L_r(\gamma)}.
\]

\begin{definition}
The resonance lattice is
\[
 M_{\mathrm{res}}
 =
 \{r\in\Z^N:\chi_r=1\text{ on }\GammaL\}.
\]
\end{definition}

\begin{theorem}[Integral resonance criterion]
\label{thm:integral-resonance}
An exponent $r$ is resonant if and only if there is a unique
$\ell\in\Z^m$ such that
\[
 L_r=\sum_{\mu=1}^m\ell_\mu B_\mu.
\]
\end{theorem}

\begin{proof}
If $\chi_r=1$, define
\[
 \alpha_r=L_r\circ B^{-1}.
\]
For every $k\in\Z^m$,
\[
 1=e^{\alpha_r(2\pi ik)}=e^{2\pi i\alpha_r(k)},
\]
so
\[
 \alpha_r(k)\in\Z.
\]
Thus $\alpha_r(z)=\sum_\mu\ell_\mu z_\mu$ for a unique
$\ell\in\Z^m$. The converse is immediate on
$\GammaL=B^{-1}(2\pi i\Z^m)$.
\end{proof}

Define the integral affine relation lattice
\[
 \operatorname{Rel}_{\Z}(\Lambda)
 =
 \left\{
 c\in\Z^{N+m+1}:
 \sum_i c_i=0,\
 \sum_i c_i\Lambda_i=0
 \right\}.
\]

For resonant $r$, define
\[
 c_j=r_j,\qquad
 c_{a_\mu}=-\ell_\mu,\qquad
 c_{a_0}=-\sum_jr_j+\sum_\mu\ell_\mu.
\]

\begin{theorem}[Resonance--relation equivalence]
\label{thm:resonance-relation}
The extension above gives a lattice isomorphism
\[
 M_{\mathrm{res}}
 \overset{\sim}{\longrightarrow}
 \operatorname{Rel}_{\Z}(\Lambda).
\]
Its inverse is restriction to the visible labels.
\end{theorem}

\begin{proof}
The resonance equation expands to
\[
 \sum_jr_j(\Lambda_j-\Lambda_{a_0})
 -
 \sum_\mu\ell_\mu
 (\Lambda_{a_\mu}-\Lambda_{a_0})
 =0,
\]
which is the affine relation for the displayed coefficients. Conversely,
every integral affine relation gives this resonance identity. Uniqueness
uses the invertibility of $B$.
\end{proof}

Write
\[
 \operatorname{Sat}_{\Z^N}(M_{\mathrm{res}})
 =
 \{r\in\Z^N:\text{some }k\geq1\text{ has }kr\in M_{\mathrm{res}}\}.
\]

\begin{proposition}
\[
 \operatorname{Sat}_{\Z^N}(M_{\mathrm{res}})
 =
 \{r:\chi_r\text{ has finite order}\}.
\]
\end{proposition}

\begin{proof}
The character $\chi_r$ has finite order if and only if
$\chi_{kr}=1$ for some $k>0$.
\end{proof}

\subsection{Positive normalization of the affine Gale transform}

Let
\[
 \widetilde\Lambda_{\R}:
 \R^{N+m+1}\longrightarrow\R^{2m+1},
 \qquad
 e_i\longmapsto
 (1,\operatorname{Re}\Lambda_i,\operatorname{Im}\Lambda_i),
\]
and put
\[
 \mathcal R_\Lambda=\ker\widetilde\Lambda_{\R}.
\]
For every label define
\[
 q_i=\varepsilon_i|_{\mathcal R_\Lambda}
 \in\mathcal R_\Lambda^\vee.
\]
Then
\[
 \sum_iq_i=0.
\]

The vectors $q_i$ are the raw affine Gale vectors.  Gale transforms and
the freedom of positive scaling in Gale diagrams are classical
\cite[Chapter~5]{Grunbaum2003}; the same positive scaling enters the
LVM--toric correspondence of \cite{MeerssemanVerjovsky2004}.
Positive dehomogenization turns these vectors into the vertices of
the convex realization of $K$ used below.

Let
\[
 \Lambda_{\R}:\R^{N+m+1}\longrightarrow\R^{2m},
 \qquad
 e_i\longmapsto(\operatorname{Re}\Lambda_i,
                 \operatorname{Im}\Lambda_i),
 \qquad
 \mathcal H_\Lambda=\ker\Lambda_{\R}.
\]
Then
\(
 \mathcal R_\Lambda=\mathcal H_\Lambda\cap\ker(\sum_i-)
\)
and $\dim\mathcal H_\Lambda=d+1$.

\begin{lemma}[Positive Gale normalization]
\label{lem:positive-Gale-calibration}
There is an element
\[
 h=(h_i)\in\mathcal H_\Lambda,
 \qquad h_i>0,
 \qquad \sum_i h_i=1.
\]
For any such $h$, put
\[
 \widehat q_i=\frac{q_i}{h_i}\in\mathcal R_\Lambda^\vee.
\]
Define the visible polytope by
\[
 Q_h=\operatorname{conv}\{\widehat q_j:j\in[N]\}.
\]
Then $Q_h$ is a $d$-dimensional simplicial convex polytope whose boundary
is identified with $K$ by the visible labels.  The points $\widehat q_j$,
$j\in[N]$, are precisely its vertices.  Every indispensable point
$\widehat q_{a_\mu}$ lies in $Q_h$ and is a nonvertex; hence adjoining
these marked points does not change the convex hull.
\end{lemma}

\begin{proof}
Admissibility makes the origin an interior point of the convex hull of the
weights.  For each label $i$, choose a convex relation for the origin in
which the coefficient of $\Lambda_i$ is positive; averaging these finitely
many relations and normalizing gives $h_i>0$, $\Lambda_{\R}h=0$, and
$\sum_i h_i=1$.

Choose a basis $c^{(1)},\ldots,c^{(d)}$ of $\mathcal R_\Lambda$.  Since
$\mathcal H_\Lambda=\R h\oplus\mathcal R_\Lambda$, the rows
\[
 h,c^{(1)},\ldots,c^{(d)}
\]
form a basis matrix for $\ker\Lambda_{\R}$.  Its $i$-th column is
\[
 h_i\bigl(1,\widehat q_i\bigr),
\]
because the coordinates of $q_i$ in the dual basis are
$c_i^{(1)},\ldots,c_i^{(d)}$.  All $h_i$ are positive, so homogeneous
Gale duality applies to the dehomogenized point configuration
$(\widehat q_i)$.  More explicitly, if $P_h$ is its augmented point
matrix and $D_h=\operatorname{diag}(h_i)$, then
\[
 P_hD_h=[h;c^{(1)};\cdots;c^{(d)}],
\]
so the columns $h_i\Lambda_i$ form a Gale transform of $P_h$.  Positive
radial rescaling of the weights preserves every zero-in-convex-hull test.
The homogeneous Gale face criterion, weak hyperbolicity,
Proposition~\ref{prop:affine-spanning}, and
Proposition~\ref{prop:Gale-dictionary} therefore identify the boundary
face complex with $K$.  Its vertices are indexed exactly by the visible labels and its nonvertices exactly the indispensable labels, proving the
assertions about the convex hull.
\end{proof}

Fix such a positive normalization $h$ for the remainder of this section.

\begin{remark}[Why the positive normalization is necessary]
Positive radial rescaling preserves all convex-support tests and
rescales the Gale vectors inversely.  Before normalization, a visible Gale vector $q_i$ may lie in the convex hull of the other Gale vectors. The halfspace theorem is applied to the normalized vertices $\widehat q_i=q_i/h_i$ of $Q_h$.
\end{remark}

For $c\in\mathcal R_\Lambda$, define
\[
 \phi_c(q)=q(c).
\]
Then
\[
 \phi_c(q_i)=c_i,
 \qquad
 \phi_c(\widehat q_i)=\frac{c_i}{h_i}.
\]

\begin{definition}
An affine function on $\mathcal R_\Lambda^\vee$ is
\emph{normalized with respect to $h$} if
\[
 \sum_i h_i\phi(\widehat q_i)=0.
\]
It is \emph{$h$-integral} if
\[
 h_i\phi(\widehat q_i)\in\Z
\]
for every visible and indispensable label.
\end{definition}

Since
\(
 \sum_i h_i\widehat q_i=\sum_iq_i=0
\)
and $\sum_i h_i=1$, the affine functions normalized with respect to $h$ are precisely the
linear functions.  Define
\[
 \operatorname{Ht}_{\Z}(Q_h,\widehat{\mathbf q};h)
 =
 \left\{
 \phi|_{Q_h}:
 \begin{array}{l}
 \phi\text{ is an affine function normalized with respect to $h$ on }
       \mathcal R_\Lambda^\vee,\\
 h_i\phi(\widehat q_i)\in\Z
       \text{ for every visible and indispensable label }i
 \end{array}
 \right\}.
\]
Because $Q_h$ is full-dimensional, restriction to $Q_h$ is injective on
linear functions.

\begin{theorem}[Relation--height equivalence]
\label{thm:relation-height}
For every positive normalization $h$, the map
$c\mapsto\phi_c|_{Q_h}$ gives a lattice isomorphism
\[
 \operatorname{Rel}_{\Z}(\Lambda)
 \overset{\sim}{\longrightarrow}
 \operatorname{Ht}_{\Z}(Q_h,\widehat{\mathbf q};h).
\]
\end{theorem}

\begin{proof}
The relation lattice is
$\mathcal R_\Lambda\cap\Z^{N+m+1}$.
Canonical biduality identifies $\mathcal R_\Lambda$ with the linear
functions on $\mathcal R_\Lambda^\vee$.  Since
$q_i=h_i\widehat q_i$, evaluation gives
\[
 c_i=\phi_c(q_i)=h_i\phi_c(\widehat q_i).
\]
This proves both directions of the weighted integrality condition.
\end{proof}

\begin{theorem}[Resonant exponents as Gale heights]
\label{thm:resonance-Gale}
For every positive normalization $h$, there is a lattice isomorphism
\[
 M_{\mathrm{res}}
 \overset{\sim}{\longrightarrow}
 \operatorname{Ht}_{\Z}(Q_h,\widehat{\mathbf q};h),
 \qquad r\longmapsto\phi_{r,h}.
\]
At every visible vertex,
\[
 h_j\phi_{r,h}(\widehat q_j)=r_j,
 \qquad
 \operatorname{sign}\phi_{r,h}(\widehat q_j)
 =\operatorname{sign}r_j.
\]
\end{theorem}

\begin{proof}
Compose Theorems~\ref{thm:resonance-relation} and
\ref{thm:relation-height}.  The value and sign statements follow from
$h_j>0$.
\end{proof}

\subsection{The two-sign theorem}

\begin{theorem}\label{thm:two-sign-resonance}
If $0\neq r\in M_{\mathrm{res}}$, then $r$ has at least one positive and at
least one negative visible coordinate.
\end{theorem}

\begin{proof}
Compactness rules out either one-sided sign pattern.  If $r\geq0$, then $z^r$ is holomorphic on $U_K$. Resonance makes it
$\GammaL$-invariant, so it descends to a holomorphic function on the
compact connected manifold $\XL$. It must be constant. Since
$(\C^*)^N\subset U_K$ and $r\neq0$, the monomial $z^r$ is nonconstant, a
contradiction. Apply the same argument to $-r$ to exclude $r\leq0$.
\end{proof}

For
\[
 V_r^-=\{j:r_j<0\},
 \qquad
 V_r^0=\{j:r_j=0\},
\]
one has
\[
 \Pole(r,I)=V_r^-\cup(I\cap V_r^0).
\]

\begin{proposition}[Halfspace form of the pole set]
For resonant $r$,
\[
 \Pole(r,I)
 =
 V_{\phi_{r,h}}^-\cup E',
\]
where
\[
 V_{\phi_{r,h}}^-
 =
 \{j:\phi_{r,h}(\widehat q_j)<0\},
 \qquad
 V_{\phi_{r,h}}^0
 =
 \{j:\phi_{r,h}(\widehat q_j)=0\},
 \qquad
 E'=I\cap V_{\phi_{r,h}}^0.
\]
If $r\neq0$, the height $\phi_{r,h}$ takes both signs on the vertices of
$Q_h$.
\end{proposition}

\begin{proof}
For $j\notin I$, the pole condition is $r_j<0$. For $j\in I$, it is
$r_j<1$, which is equivalent to $r_j\leq0$ because $r_j$ is integral.
The sign identity in Theorem~\ref{thm:resonance-Gale} gives the displayed
halfspace form, and Theorem~\ref{thm:two-sign-resonance} gives both signs.
\end{proof}

\subsection{Halfspace contractibility}
\label{sec:halfspace-contractibility}

Let $Q$ be a simplicial convex $d$-polytope, let
\[
 K=\partial Q,
\]
and let $\phi$ be affine with both positive and negative vertex values.
Set
\[
 V_\phi^-=\{v:\phi(v)<0\},
 \qquad
 V_\phi^0=\{v:\phi(v)=0\},
 \qquad
 V_\phi^+=\{v:\phi(v)>0\}.
\]

\subsubsection{Local polytope lemmas}

\begin{lemma}[Tangent-cone principle]
\label{lem:tangent-cone-principle}
For an affine function $\phi$ on $Q$, if a vertex $v$ has no adjacent vertex of smaller $\phi$-value, then $\phi$ attains its global minimum on $Q$ at $v$. If it has no adjacent vertex of larger value, then $\phi$ attains its global maximum at $v$.
\end{lemma}

\begin{proof}
The tangent cone at $v$ is generated by the incident edge directions
$w-v$. If the linear part of $\phi$ is nonnegative on every such
direction, it is nonnegative on the tangent cone. Since $Q-v$ lies in this
cone, $v$ is a global minimum. Apply the argument to $-\phi$ for maxima.
\end{proof}

\begin{lemma}[Vertex-figure height]
\label{lem:vertex-figure-height}
Let $\phi$ be an affine function on $Q$ and let $Q_v$ be a vertex figure at $v$. There is an affine function $\phi_v$
on $Q_v$ whose value at the vertex corresponding to the edge $[v,w]$ has
the sign of
\[
 \phi(w)-\phi(v).
\]
If $Q$ is simplicial, then $Q_v$ is a simplicial $(d-1)$-polytope and
\[
 \partial Q_v\cong\operatorname{lk}_K(v).
\]
This is an isomorphism of abstract simplicial complexes: the vertex on
$[v,w]$ corresponds to the link vertex $w$.
\end{lemma}

\begin{proof}
Choose the vertex figure by intersecting incident edges with a separating
hyperplane. The intersection point on $[v,w]$ is
$(1-t_w)v+t_ww$ with $t_w>0$. Restrict $\phi-\phi(v)$ to the hyperplane.
Its value is
\[
 t_w(\phi(w)-\phi(v)).
\]
\end{proof}

\begin{lemma}[Cone attachment]
\label{lem:cone-attachment}
Let $T$ be a vertex set and $v\notin T$. Then
\[
 K_{T\cup\{v\}}
 =
 K_T\cup
 \left(
 v*(\operatorname{lk}_K(v))_T
 \right).
\]
If $K_T$ and $(\operatorname{lk}_K(v))_T$ are contractible and the latter
is nonempty, then $K_{T\cup\{v\}}$ is contractible.
\end{lemma}

\begin{proof}
The combinatorial decomposition is immediate by separating faces according
to whether they contain $v$. The two subcomplexes meet in the induced
link. Their geometric realizations form a homotopy pushout because
subcomplex inclusions are cofibrations. It is the homotopy pushout of two
contractible spaces over a contractible space, hence is contractible.
\end{proof}

\subsubsection{The negative subcomplex}

\begin{lemma}[Generic positive perturbation]
\label{lem:generic-positive-perturbation}
There is an affine function $\psi$ such that
\[
 V_\psi^-=V_\phi^-,
\]
all vertices in $V_\phi^0\cup V_\phi^+$ have positive $\psi$-value, and all
vertex values of $\psi$ are distinct.
\end{lemma}

\begin{proof}
First add a sufficiently small positive constant to $\phi$, preserving the
signs of the originally nonzero vertex values and making the zero vertices
positive. Then add a sufficiently small generic affine perturbation.
\end{proof}

\begin{proposition}\label{prop:negative-halfspace-contractible}
The induced complex $K_{V_\phi^-}$ is contractible.
\end{proposition}

\begin{proof}
Induct on $d$. The case $d=1$ is a single negative endpoint.

Choose $\psi$ from
Lemma~\ref{lem:generic-positive-perturbation} and order the negative
vertices
\[
 \psi(v_1)<\cdots<\psi(v_s)<0.
\]
Let $S_i=\{v_1,\ldots,v_i\}$. The first induced complex is a point.

For $i\geq2$, the vertex $v_i$ is neither a global minimum nor a global
maximum. By
Lemma~\ref{lem:tangent-cone-principle}, it has a lower and a higher
neighbour. On the vertex figure, the induced height from
Lemma~\ref{lem:vertex-figure-height} takes both signs, and its negative
vertices are exactly the neighbours lying in $S_{i-1}$.  Under the
simplicial identification with $\operatorname{lk}_K(v_i)$, its negative
induced subcomplex is the attaching link.  By the induction
hypothesis in dimension $d-1$, the attaching link is contractible and
nonempty. Apply
Lemma~\ref{lem:cone-attachment}.
\end{proof}

\subsubsection{Arbitrary zero-level vertices}

\begin{theorem}[Halfspace contractibility]
\label{thm:halfspace-contractibility}
For every
\[
 E'\subseteq V_\phi^0,
\]
\[
 K_{V_\phi^-\cup E'}\text{ is contractible}.
\]
\end{theorem}

\begin{proof}
Induct on $d$. The strict negative complex is contractible by
Proposition~\ref{prop:negative-halfspace-contractible}.

Order the selected zero vertices
\[
 E'=\{z_1,\ldots,z_t\}
\]
and add them successively. Suppose
\[
 T_{i-1}
 =
 V_\phi^-\cup\{z_1,\ldots,z_{i-1}\}.
\]
The zero vertex $z_i$ is neither a global minimum nor a global maximum,
because $\phi$ has both negative and positive vertices. Hence it has a
negative and a positive neighbour.

On the vertex figure $Q_{z_i}$, the induced height has the signs of
$\phi$ on the neighbouring vertices.  Under the simplicial identification
$\partial Q_{z_i}\cong\operatorname{lk}_K(z_i)$, the attaching link is
isomorphic to
\[
 (\partial Q_{z_i})_{V_{\phi_{z_i}}^-\cup E_i''},
\]
where $E_i''$ is the set of previously selected zero neighbours. By the
full induction hypothesis in dimension $d-1$, this link is contractible;
it is nonempty because there is a negative neighbour. Apply
Lemma~\ref{lem:cone-attachment}.
\end{proof}

\begin{corollary}
For every $E'\subseteq V_\phi^0$,
\[
 K_{V_\phi^+\cup E'}
\]
is contractible.
\end{corollary}

\begin{proof}
Apply the theorem to $-\phi$.
\end{proof}

\subsubsection{Resonant sectors}

Let $0\neq r\in M_{\mathrm{res}}$. Then the positively normalized Gale
height $\phi_{r,h}$ takes both signs and
\[
 \Pole(r,I)
 =
 V_{\phi_{r,h}}^-\cup(I\cap V_{\phi_{r,h}}^0).
\]

\begin{theorem}[Vanishing of nonzero resonant sectors]
\label{thm:nonzero-resonant-vanishing}
For every $0\neq r\in M_{\mathrm{res}}$ and every $I$,
\[
 K_{\Pole(r,I)}
\]
is contractible. Consequently,
\[
 H^a(\mathscr P_{r,I}^\bullet)=0
\]
for every $a$.
\end{theorem}

\begin{proof}
Apply
Theorem~\ref{thm:halfspace-contractibility}
to $\phi_{r,h}$ on $Q_h$ and the selected zero set
$I\cap V_{\phi_{r,h}}^0$, then use the
fixed Laurent-sector theorem.
\end{proof}

\section{Derived deck descent and character contraction}
\label{sec:derived-deck-descent}

To pass from $U_K$ to $\XL$, we resolve the invariant subsheaf by
the Koszul complex of the deck generators.  On an evenly covered open
set, bilateral summation proves exactness of this resolution.

Choose an ordered $\Z$-basis
$\gamma_1,\ldots,\gamma_m$ of the deck lattice and put
\[
 V_\Gamma=\GammaL\otimes_\Z\C,
 \qquad
 V_\Gamma^\vee\cong\GammaLCdual,
\]
and let $\xi_1,\ldots,\xi_m$ be the dual basis.

For a $\GammaL$-module $M$, set
\[
 \Delta_i=\gamma_i^*-\operatorname{id}.
\]

\begin{definition}
The group--Koszul complex is
\[
 \Kos_{\GammaL}^s(M)
 =
 M\otimes\Lambda^sV_\Gamma^\vee
\]
with
\[
 d_\Gamma(m\otimes\eta)
 =
 \sum_{i=1}^m
 \Delta_i(m)\otimes\xi_i\wedge\eta.
\]
\end{definition}

We write
\(
 R\Gamma(\GammaL,-)=R((-)^{\GammaL})
\)
for derived invariants.  The derived invariant object is independent of the deck basis; a chosen ordered basis gives the displayed group--Koszul complex, and a change of basis gives a chain-homotopy-equivalent complex.

\begin{proposition}[Koszul model for derived invariants]
\label{prop:group-Koszul-derived-invariants}
Let $M^\bullet$ be a complex of $\C[\GammaL]$-modules.  The chosen
ordered basis identifies
\[
 \C[\GammaL]
 \cong
 \C[t_1^{\pm1},\ldots,t_m^{\pm1}].
\]
Then the total group--Koszul complex with operators
$t_i-1$ represents $R\Gamma(\GammaL,M^\bullet)$.
\end{proposition}

\begin{proof}
The sequence $t_1-1,\ldots,t_m-1$ is regular in the Laurent polynomial
ring, and its quotient is the trivial module $\C$.  The homological Koszul
complex on this sequence is therefore a finite free resolution of $\C$.
Applying
$\operatorname{Hom}_{\C[\GammaL]}(-,M^\bullet)$
and totalizing gives exactly the displayed cochain differential
$\sum_i(\gamma_i^*-1)\xi_i\wedge(-)$, with the standard total sign.
\end{proof}

\subsection{Bilateral summation}

Let $W$ be a complex vector space of arbitrary dimension.  For the continuity statements, assume that $W$
is locally convex and give each sequence space its product topology.  Put
\[
 \operatorname{Seq}(W)=\operatorname{Map}(\Z,W),
 \qquad
 (Tf)(\nu)=f(\nu+1),
 \qquad
 \Delta=T-\operatorname{id}.
\]
Define
\[
 (Hg)(\nu)
 =
 \begin{cases}
 \sum_{t=0}^{\nu-1}g(t),&\nu>0,\\
 0,&\nu=0,\\
 -\sum_{t=\nu}^{-1}g(t),&\nu<0.
 \end{cases}
\]

\begin{lemma}\label{lem:bilateral-summation}
Let $\iota:W\to\operatorname{Seq}(W)$ be the constant inclusion and
$\varepsilon_0(f)=f(0)$. Then
\[
 \Delta H=\operatorname{id},
 \qquad
 H\Delta=\operatorname{id}-\iota\varepsilon_0.
\]
Consequently,
\[
 0\to W\to\operatorname{Seq}(W)
 \xrightarrow{\Delta}\operatorname{Seq}(W)\to0
\]
is exact.
\end{lemma}

\begin{proof}
Both identities are finite telescoping sums.
\end{proof}

If $W$ is locally convex and the sequence space has the product topology,
all the displayed operators are continuous because each output coordinate depends on only finitely many coordinates.

\begin{theorem}[Coinduced Koszul resolution]
\label{thm:coinduced-Koszul}
For
\[
 \operatorname{Coind}_m(W)=\operatorname{Map}(\Z^m,W),
\]
let the $i$-th generator act by
\[
 (\gamma_i f)(\nu)=f(\nu+e_i).
\]
Then the augmented complex
\[
 0\to W\to \Kos_{\Z^m}^\bullet(\operatorname{Coind}_m(W))
\]
is exact. If $W$ is locally convex, the product topology gives a continuous deformation
retraction onto $W$ in degree zero.
\end{theorem}

\begin{proof}
Induct on $m$. In the last coordinate, let $H_m$ be the bilateral summation
operator of Lemma~\ref{lem:bilateral-summation} and set
\(
 \mathsf h_m=H_m\iota_{e_m},
\)
where $\iota_{e_m}$ contracts the exterior factor.  The operator $H_m$
commutes with the first $m-1$ shifts, while
\(
 \iota_{e_m}(\xi_i\wedge)+\xi_i\wedge\iota_{e_m}=0
\)
for $i<m$.  Lemma~\ref{lem:bilateral-summation} therefore gives
\[
 d_\Gamma\mathsf h_m+\mathsf h_m d_\Gamma
 =\operatorname{id}-P_m,
\]
where $P_m$ projects onto sequences constant in the last coordinate with
no $\xi_m$ factor.  Thus the full Koszul complex retracts onto the
$(m-1)$-variable complex; iterate.  Every output coordinate of every
resulting homotopy is a finite coordinate sum, so the deformation
retraction is continuous for the product topology.
\end{proof}

\subsection{The sheaf resolution}

Set
\[
 \mathcal F=\Omega_{U_K}^p,
 \qquad
 \mathcal E=\Omega_{\XL}^p.
\]
On an evenly covered open set $\mathcal V\subseteq \XL$,
\[
 \pi_*\mathcal F(\mathcal V)
 \cong
 \operatorname{Map}(\GammaL,\mathcal E(\mathcal V)).
\]
Indeed, choose one sheet $\widetilde{\mathcal V}$ and let
$j_\gamma:\mathcal V\to\gamma\widetilde{\mathcal V}$ be the inverse
covering chart.  The identification is
\[
 \Theta(s)(\gamma)
 =j_\gamma^*\!\left(s|_{\gamma\widetilde{\mathcal V}}\right),
 \qquad
 \Theta(\delta^*s)(\gamma)=\Theta(s)(\delta\gamma).
\]
Thus the deck basis acts by the coordinate shifts above, and the inclusion
\(\mathcal E\hookrightarrow\pi_*\mathcal F\)
is the constant-family inclusion.  In particular,
\[
 \mathcal E=(\pi_*\mathcal F)^{\GammaL}.
\]

\begin{theorem}[Group--Koszul sheaf resolution]
\label{thm:group-Koszul-sheaf}
The sequence
\[
 0\to\mathcal E
 \to\pi_*\mathcal F
 \to\pi_*\mathcal F\otimes V_\Gamma^\vee
 \to\cdots
 \to\pi_*\mathcal F\otimes\Lambda^mV_\Gamma^\vee
 \to0
\]
is exact.
\end{theorem}

\begin{proof}
Evenly covered opens form a basis.  On each such open the augmented
sequence is the coinduced resolution of
Theorem~\ref{thm:coinduced-Koszul}; hence the displayed sequence is exact
stalkwise.
\end{proof}

\begin{proposition}\label{prop:covering-direct-image}
For $q>0$,
\[
 R^q\pi_*\mathcal F=0.
\]
Consequently,
\[
 R\Gamma(\XL,\pi_*\mathcal F)
 \simeq
 R\Gamma(U_K,\mathcal F).
\]
\end{proposition}

\begin{proof}
Every point has a basis of evenly covered coordinate balls $\mathcal V$.
Their inverse images are disjoint unions of Stein sheets. Cartan's theorem~B
\cite{Cartan1953} gives vanishing of the higher cohomology of $\mathcal F$
on every sheet.  Sheaves on a disjoint union form the product of the sheaf
categories on its components, and products of complex vector spaces are
exact.  Thus cohomology of the disjoint union is the product of the sheet
cohomologies and also vanishes in positive degree.  Taking the filtered
colimit over this basis shows that every stalk of
$R^q\pi_*\mathcal F$ vanishes.  Derived Leray then gives the asserted
equivalence
\cite[Tags~01F4 and~01F5]{StacksProject}.
\end{proof}

\begin{theorem}[Derived deck descent]
\label{thm:derived-deck-descent}
There is a canonical equivalence
\[
 R\Gamma(\XL,\Omega_{\XL}^p)
 \simeq
 R\Gamma\!\left(
 \GammaL,R\Gamma(U_K,\Omega_{U_K}^p)
 \right).
\]
Here the inner derived global-sections object is regarded as an object of
$D(\C[\GammaL])$.
After the ordered basis of $\GammaL$ and an explicit equivariant representative
of $R\Gamma(U_K,\Omega_{U_K}^p)$ are chosen, the right-hand side is represented
by the corresponding group--Koszul total complex.
\end{theorem}

\begin{proof}
Apply derived global sections to
Theorem~\ref{thm:group-Koszul-sheaf}.
The group--Koszul complex is bounded and its exterior factors are
finite-dimensional. Then use
Proposition~\ref{prop:covering-direct-image}.
\end{proof}

For every maximal face $\sigma$, put
\(
 \mathcal V_\sigma=\pi(U_\sigma).
\)
The Cox charts are $\GammaL$-invariant and saturated, so the
$\mathcal V_\sigma$ form a finite open cover of $\XL$ and
\[
 \mathcal V_{\sigma_0}\cap\cdots\cap\mathcal V_{\sigma_a}
 =
 \pi(U_{\sigma_0\cap\cdots\cap\sigma_a}).
\]
Sections of $\pi_*\mathcal F$ on this intersection are exactly the
sections of $\mathcal F$ on the corresponding Stein Cox intersection.
Proposition~\ref{prop:covering-direct-image} and Cartan's theorem~B show
that the quotient cover is acyclic for every term of the sheaf resolution.
The acyclic-cover theorem
\cite[Tag~0FLH]{StacksProject} therefore identifies its
\v{C}ech hypercohomology with
\[
 R\Gamma(\XL,\Omega_{\XL}^p)
 \simeq
 \operatorname{Tot}
 \Kos_{\GammaL}^\bullet(C_p^\bullet).
\]
This is an equivalence in $D(\C)$ represented by a bounded double complex; every group differential is continuous because it is a
finite sum of continuous deck maps.

All derived-category assertions in this paragraph are made in $D(\C)$
after forgetting topology.  The concrete representative remains a bounded
complex of nuclear Fr\'echet spaces, and all displayed differentials and
comparison maps are continuous.  Since the exterior factors are
finite-dimensional, algebraic and completed tensor products with them agree.
For $c\in C_p^a$,
\[
 d_{\mathrm{tot}}(c\otimes\eta)
 =
 \delta_{\check C}c\otimes\eta
 +
 (-1)^a
 \sum_i(\gamma_i^*-1)c\otimes\xi_i\wedge\eta.
\]

\subsection{Fixed-character contraction}

On the Laurent sector $(r,I)$,
\[
 \gamma_i^*=\chi_r(\gamma_i)\operatorname{id}.
\]
Put
\[
 \lambda_i(r)=\chi_r(\gamma_i)-1,
 \qquad
 \alpha_r=\sum_i\lambda_i(r)\xi_i.
\]
The group differential is $\alpha_r\wedge(-)$.

\begin{theorem}[Nontrivial-character contraction]
\label{thm:nontrivial-character-contraction}
If $\chi_r\neq1$, then
\[
 \operatorname{Tot}
 \Kos_{\GammaL}^\bullet(\mathscr P_{r,I}^\bullet)
\]
is contractible.
\end{theorem}

\begin{proof}
Choose $i_0$ with $\lambda_{i_0}(r)\neq0$ and let $\iota_{i_0}$ be
contraction by the basis vector dual to $\xi_{i_0}$. Then
\[
 (\alpha_r\wedge)\iota_{i_0}
 +
 \iota_{i_0}(\alpha_r\wedge)
 =
 \lambda_{i_0}(r)\operatorname{id}.
\]
Thus
\[
 h_{\Gamma,r}
 =
 \lambda_{i_0}(r)^{-1}\iota_{i_0}
\]
contracts the group complex. On Čech degree $a$, set
\[
 H_r=(-1)^a\operatorname{id}\otimes h_{\Gamma,r}.
\]
The mixed Čech terms cancel and
\[
 d_{\mathrm{tot}}H_r+H_rd_{\mathrm{tot}}=\operatorname{id}.
\]
\end{proof}

\noindent This applies equally to infinite-order and finite-order nontrivial
characters.

If $r\in M_{\mathrm{res}}$, the group differential vanishes and
\[
 H^q\left(
 \operatorname{Tot}
 \Kos_{\GammaL}^\bullet(\mathscr P_{r,I}^\bullet)
 \right)
 \cong
 \bigoplus_{a+s=q}
 H^a(\mathscr P_{r,I}^\bullet)
 \otimes\Lambda^sV_\Gamma^\vee.
\]

\begin{remark}
The differences $\chi_r(\gamma_i)-1$ may approach zero as $r$ varies. We
therefore apply the character contractions only after the finite-support
reduction of Theorem~\ref{thm:Euler-finite-support}.
\end{remark}

For a fixed exponent, a nontrivial character gives a contractible
total complex, while a trivial character leaves the \v{C}ech
cohomology tensored with $\Lambda^\bullet\GammaLCdual$.  To pass to
all exponents, it remains to prove the finite-support theorem.

\section{Finite Laurent support and reduction to the zero exponent}
\label{sec:Euler-finite-support}

The preceding contractions apply to one exponent at a time.
We now prove that every cohomology class has finite Laurent support.
Euler division is continuous away from an integral eigenvalue, and
compactness makes the induced Euler action finite-dimensional.

Define
\[
 \mathcal T_p^\bullet
 =
 \operatorname{Tot}
 \Kos_{\GammaL}^\bullet
 \left(
 \check C^\bullet(\mathcal U_K,\Omega_{U_K}^p)
 \right).
\]
For $r\in\Z^N$, let
\[
 \mathcal T_{p,r}^\bullet
 =
 \bigoplus_{|I|=p}
 \operatorname{Tot}
 \Kos_{\GammaL}^\bullet(\mathscr P_{r,I}^\bullet).
\]
Then
\[
 \mathcal T_p^\bullet
 \cong
 \widehat{\bigoplus}_{r\in\Z^N}^{\,\mathrm{an}}
 \mathcal T_{p,r}^\bullet.
\]
Choose the exterior basis of $\Lambda^\bullet V_\Gamma^\vee$ induced
by $\xi_1,\ldots,\xi_m$.  If
\(
c=\sum_{a+s=q,J}c_{a,J}\otimes\xi_J
\)
has total degree $q$, define
\[
 \mathfrak q_{q,R}^{\mathrm{tot}}(c)
 =
 \sum_{a+s=q}
 \sum_{|J|=s}
 \mathfrak q_{a,R}(c_{a,J}).
\]
These seminorms define the finite direct-sum Fr\'echet topology in each
total degree.
Define
\[
 \mathcal T_{p,\mathrm{fin}}^\bullet
 =
 \bigoplus_{r\in\Z^N}^{\mathrm{alg}}
 \mathcal T_{p,r}^\bullet.
\]

We take cohomology in the category of complex vector spaces. All operators on cochains considered below are continuous.

For each $j$, the Euler Lie derivative $\mathcal L_j$ is a continuous
chain endomorphism and
\[
 \mathcal L_j|_{\mathcal T_{p,r}^\bullet}
 =
 r_j\operatorname{id}.
\]

\subsection{Continuous Euler division}

Fix $\lambda\in\C$. Let $\Pi_{j,\lambda}$ be the projection onto exponents
with $r_j=\lambda$, interpreted as zero when $\lambda\notin\Z$, and let
$\Pi_{j,\lambda}^{\perp}=1-\Pi_{j,\lambda}$.

\begin{lemma}\label{lem:integral-spectral-gap}
The number
\[
 \delta(\lambda)
 =
 \inf\{|\nu-\lambda|:\nu\in\Z,\ \nu\neq\lambda\}
\]
is positive.
\end{lemma}

\begin{proof}
The set $\Z\subset\C$ is closed and discrete. If $\lambda\in\Z$, the
displayed infimum is one.
\end{proof}

On the complementary subcomplex define
\[
 R_{j,\lambda}^{(M)}
 |_{\mathcal T_{p,r}^\bullet}
 =
 (r_j-\lambda)^{-M}\operatorname{id}.
\]

\begin{proposition}[Euler division]
\label{prop:Euler-division}
The operator $\mathcal L_j-\lambda$ is a continuous chain automorphism of
the subcomplex supported on $r_j\neq\lambda$. Its inverse is
$R_{j,\lambda}^{(1)}$. For every Köthe seminorm,
\[
 \mathfrak q_{q,R}^{\mathrm{tot}}
 \left(
 R_{j,\lambda}^{(M)}c
 \right)
 \leq
 \delta(\lambda)^{-M}
 \mathfrak q_{q,R}^{\mathrm{tot}}(c).
\]
\end{proposition}

\begin{proof}
On exponent $r$, the operator is multiplication by $r_j-\lambda$.
Lemma~\ref{lem:integral-spectral-gap} bounds all inverse scalars uniformly.
The Köthe seminorms are weighted $\ell^1$-sums of sector coefficients.
\end{proof}

\subsection{Generalized eigenclasses}

By derived descent,
\[
 H^q(\mathcal T_p^\bullet)
 \cong
 H^q(\XL,\Omega_{\XL}^p).
\]
Compactness and the Cartan--Serre finiteness theorem
\cite{CartanSerre1953} imply that this vector space is finite-dimensional.

The commuting Euler operators therefore admit a finite simultaneous
generalized-eigenspace decomposition.

\begin{lemma}\label{lem:eigen-hyperplane}
Let $[c]$ belong to the joint generalized eigenspace with eigenvalue
$\lambda=(\lambda_1,\ldots,\lambda_N)$. Then for every $j$,
\[
 [c]=[\Pi_{j,\lambda_j}c].
\]
If $\lambda_j\notin\Z$, then $[c]=0$.
\end{lemma}

\begin{proof}
Choose $M_j$ with
\[
 (\mathcal L_j-\lambda_j)^{M_j}[c]=0.
\]
Choose a closed representative $c$ of the class and a cochain $b$ such
that
\[
 (\mathcal L_j-\lambda_j)^{M_j}c=d_{\mathrm{tot}}b.
\]
Decompose
\[
 c=\Pi_{j,\lambda_j}c+\Pi_{j,\lambda_j}^{\perp}c.
\]
The complementary component is killed in cohomology by a power of
$\mathcal L_j-\lambda_j$. On the complementary subcomplex this operator
is a chain automorphism by
Proposition~\ref{prop:Euler-division}; hence that component is exact.
Explicitly, after replacing $b$ by its complementary projection
$b_\perp=\Pi_{j,\lambda_j}^{\perp}b$, one has
\[
 (\mathcal L_j-\lambda_j)^{M_j}c_\perp
 =d_{\mathrm{tot}}b_\perp.
\]
Therefore
\[
 c_\perp
 =
 d_{\mathrm{tot}}\left(
 R_{j,\lambda_j}^{(M_j)}b_\perp
 \right).
\]
\end{proof}

\begin{proposition}\label{prop:single-exponent-representative}
Every nonzero joint generalized eigenclass has eigenvalue
$\lambda\in\Z^N$ and admits a representative supported on the single
Laurent exponent $r=\lambda$.
\end{proposition}

\begin{proof}
Apply
Lemma~\ref{lem:eigen-hyperplane}
successively for $j=1,\ldots,N$. All projections commute with one another,
the total differential, and the Euler operators.
\end{proof}

\begin{corollary}
The Euler action on $H^q(\mathcal T_p^\bullet)$ is simultaneously
semisimple with integral joint spectrum.
\end{corollary}

\subsection{The quasi-isomorphism}

\begin{proposition}[Finite Laurent representatives]
Every class in $H^q(\mathcal T_p^\bullet)$ has a finite Laurent-support
representative.
\end{proposition}

\begin{proof}
Decompose the class into the finite sum of its simultaneous generalized
eigencomponents. By
Proposition~\ref{prop:single-exponent-representative},
each nonzero component has a representative on one integral exponent.
\end{proof}

\begin{proposition}[Finite-support primitives]
Let $c\in\mathcal T_{p,\mathrm{fin}}^q$ and suppose $c=d_{\mathrm{tot}}b$ in the completed
complex. If $F$ is the finite exponent support of $c$, then
\[
 c=d_{\mathrm{tot}}(\Pi_Fb),
\]
and $\Pi_Fb$ has finite support.
\end{proposition}

\begin{proof}
Sector projections commute with the total differential:
\[
 d_{\mathrm{tot}}(\Pi_Fb)=\Pi_F(d_{\mathrm{tot}}b)=\Pi_Fc=c.
\]
\end{proof}

\begin{theorem}[Euler finite-support quasi-isomorphism]
\label{thm:Euler-finite-support}
The inclusion
\[
 \mathcal T_{p,\mathrm{fin}}^\bullet
 \hookrightarrow
 \mathcal T_p^\bullet
\]
is a quasi-isomorphism.
\end{theorem}

\begin{proof}
Surjectivity on cohomology is the finite generalized-eigenspace argument
above, and injectivity is the finite-support projection argument.
\end{proof}

Consequently,
\[
 H^q(\mathcal T_p^\bullet)
 \cong
 \bigoplus_{r\in\Z^N}
 H^q(\mathcal T_{p,r}^\bullet).
\]
Here cohomology commutes with the algebraic direct sum because kernels and
images are computed componentwise.  The left-hand side is finite-dimensional
by Cartan--Serre, so only finitely many sector-cohomology summands can be
nonzero in any fixed degree.

\subsection{The zero sector}
\label{sec:zero-sector}

By the finite-support theorem, cohomology is the algebraic sum of
the sector cohomologies.  Nontrivial characters contribute zero by
Theorem~\ref{thm:nontrivial-character-contraction}.  Nonzero resonant
exponents contribute zero by
Theorem~\ref{thm:nonzero-resonant-vanishing}.  Hence only $r=0$ remains.

\begin{theorem}[Reduction to the zero exponent]
\label{thm:zero-sector-survival}
The inclusion
\[
 \mathcal T_{p,0}^\bullet
 \hookrightarrow
 \mathcal T_p^\bullet
\]
is a quasi-isomorphism. Consequently,
\[
 R\Gamma(\XL,\Omega_{\XL}^p)
 \simeq
 \mathcal T_{p,0}^\bullet.
\]
\end{theorem}

\begin{proof}
The finite-support complex splits as the direct sum of the zero sector and
the sum of all nonzero sectors.  Every cocycle in the latter sum has finite
support; choose a primitive in each of its finitely many acyclic sector
components and add them.  Hence the nonzero-sector algebraic sum is acyclic.
Compose with derived descent and the finite-support quasi-isomorphism.
\end{proof}

Since $\chi_0=1$, the group--Koszul differential vanishes on the zero
sector. Therefore
\[
 \mathcal T_{p,0}^\bullet
 =
 \bigoplus_{|I|=p}
 \mathscr P_{0,I}^\bullet
 \otimes
 \Lambda^\bullet\GammaLCdual.
\]

\subsubsection{The Hochster complex}

For $I\subseteq[N]$, the logarithmic generator in the zero-exponent sector is
\[
 e_I=\bigwedge_{i\in I}\frac{dz_i}{z_i}.
\]
Its pole set is
\[
 \Pole(0,I)=I.
\]

Let $\Delta_{\max}$ be the full simplex on the maximal faces of $K$, and
let $\mathfrak D_I\subseteq\Delta_{\max}$ be the subcomplex of Čech simplices whose
common face meets $I$. The fixed-sector theorem gives
\[
 \mathscr P_{0,I}^\bullet
 \cong
 C^\bullet(\Delta_{\max},\mathfrak D_I;\C).
\]

\begin{proposition}[Reduced-cochain model]
\label{prop:zero-reduced-cochain}
After ordering the finite cover indexed by $I$, there is an explicit
zigzag of cochain quasi-isomorphisms
\[
 \mathscr P_{0,I}^\bullet
 \simeq
 \widetilde C^{\bullet-1}(\mathfrak D_I;\C)
 \simeq
 \widetilde C^{\bullet-1}(K_I;\C).
\]
\end{proposition}

\begin{proof}
The relative complex admits a canonical quasi-isomorphism to
\[
 \operatorname{Cone}
 \left(
 C^\bullet(\Delta_{\max})
 \longrightarrow C^\bullet(\mathfrak D_I)
 \right)[-1].
\]
Since $\Delta_{\max}$ is contractible, replace its cochains by $\C$ in
degree zero. The resulting cone is the shifted reduced cochain complex of
$\mathfrak D_I$.

For $I\neq\varnothing$, use the finite Mayer--Vietoris double complex
\[
 B^{u,v}_I
 =
 \bigoplus_{i_0<\cdots<i_u\in I}
 C^v(\mathfrak D_{i_0}\cap\cdots\cap\mathfrak D_{i_u};\C).
\]
The horizontal augmentation is a quasi-isomorphism from
$C^\bullet(\mathfrak D_I)$ to $\operatorname{Tot}B_I$: on each simplicial
cochain the indices of cover members containing its support form a simplex,
so the augmented horizontal row is exact.  Every nonempty intersection in
the vertical direction is a simplex, and it is nonempty precisely when
$\{i_0,\ldots,i_u\}\in K_I$.  Collapsing these intersections to constants
therefore gives a second quasi-isomorphism
\[
 C^\bullet(K_I;\C)\longrightarrow\operatorname{Tot}B_I.
\]
Passing to augmented cochains and shifting yields the claimed reduced
zigzag.  If $I=\varnothing$, then
$\mathscr P_{0,\varnothing}^\bullet=C^\bullet(\Delta_{\max})\simeq\C[0]$,
which agrees with the convention
\[
 \widetilde C^{-1}(K_\varnothing)=\C.
\]
\end{proof}

Taking cohomology recovers Corollary~\ref{cor:zero-sector-cohomology}.

\subsubsection{Chain-level splitting}

\begin{theorem}[Additive Cox--Čech Dolbeault--Hochster splitting]
\label{thm:Cox-Cech-Hochster-splitting}
Let $\XL$ be minimally stable and polytopal with $d\geq1$, and assume
$B$ is invertible.  For $0\leq p\leq N$, there is an isomorphism in the
derived category $D(\C)$
\[
 R\Gamma(\XL,\Omega_{\XL}^p)
 \simeq
 \bigoplus_{|I|=p}
 \widetilde C^{\bullet-1}(K_I;\C)
 \otimes
 \Lambda^\bullet\GammaLCdual.
\]
An element of
\[
 \widetilde C^{a-1}(K_I)
 \otimes
 \Lambda^s\GammaLCdual
\]
has total degree $a+s$; the exterior factor has zero differential and the
right-hand side carries the total tensor-product differential.
\end{theorem}

\begin{proof}
Combine
Theorem~\ref{thm:zero-sector-survival}
with
Proposition~\ref{prop:zero-reduced-cochain}.
\end{proof}

\begin{remark}[Products and the second comparison]
The second proof gives an additive isomorphism in $D(\C)$ (Theorem~\ref{thm:Cox-Cech-Hochster-splitting}). Multiplicative descent requires a
compatible cochain product: the tensor/wedge product satisfies
\[
 (\gamma_i^*-1)(xy)
 =
 (\gamma_i^*-1)(x)\,\gamma_i^*(y)
 +x(\gamma_i^*-1)(y),
\]
so the group--Koszul differential obeys a twisted Leibniz rule.
A multiplicative comparison requires a bar or group-cochain diagonal
and compatible multiplicative nerve maps.  Part~II gives the graded-algebra identification.
\end{remark}

\begin{corollary}[Dolbeault--Hochster decomposition: second proof]
\label{thm:Dolbeault-Hochster-second-proof}
For $0\leq p,q\leq N$,
\[
 H^q(\XL,\Omega_{\XL}^p)
 \cong
 \bigoplus_{\substack{I\subseteq[N]\\|I|=p}}
 \ \bigoplus_{a+s=q}
 \widetilde H^{a-1}(K_I;\C)
 \otimes
 \Lambda^s\GammaLCdual.
\]
\end{corollary}

\begin{proof}
Take cohomology in total degree $q$ in
Theorem~\ref{thm:Cox-Cech-Hochster-splitting}.
\end{proof}

This proves \eqref{eq:intro-main-formula} by holomorphic descent.  The consequences of the formula are developed in
Section~\ref{sec:full-rank-consequences}.

\Needspace{12\baselineskip}
\part{Examples and deformations}
\section{Consequences of the full-rank formula}
\label{sec:full-rank-consequences}

We apply the full-rank formula to the edge groups, duality, genera,
and Betti numbers.  Alexander duality pairs complementary supports;
the exterior deck factor makes every Hodge-row polynomial divisible by $(1+v)^m$.

Assume throughout that $B$ is invertible.  Let $K$ be the visible
$(d-1)$-sphere on $N$ vertices and set
\[
 m=N-d=\operatorname{rank}\GammaL.
\]
Since every $d$-polytope has at least $d+1$ facets, one has $m\geq1$.
For fixed holomorphic degree $p$, write
\[
 \mathscr H_{K,p}(v)
 =
 \sum_{a=0}^{d}
 \left(
 \sum_{\substack{I\subseteq[N]\\|I|=p}}
 \widetilde b_{a-1}(K_I)
 \right)v^a,
 \qquad
 H_p(v)=\sum_{q=0}^{N}h^{p,q}(\XL)v^q.
\]
The Dolbeault--Hochster theorem gives the row identity
\begin{equation}
 \label{eq:row-binomial-packet}
 H_p(v)=(1+v)^m\mathscr H_{K,p}(v).
\end{equation}
The identities for $\mathscr H_K$ use only the sphere $K$.
Their Dolbeault interpretations use the full-rank realization
assumed in this section.

\subsection{The edge groups}

\begin{theorem}[Edge Dolbeault groups]
\label{thm:edge-Dolbeault-groups}
There are natural identifications
\[
 H^\bullet(\XL,\mathcal O_{\XL})
 \cong
 \Lambda^\bullet\GammaLCdual
\]
as graded algebras, and
\[
 H^0(\XL,\Omega_{\XL}^p)=0
 \qquad(1\leq p\leq N).
\]
Moreover,
\[
 H^q(\XL,\Omega_{\XL}^N)
 \cong
 \widetilde H^{d-1}(K;\C)
 \otimes
 \Lambda^{q-d}\GammaLCdual,
\]
where the right-hand side is zero unless $d\leq q\leq N$.  Consequently,
\[
 h^{0,q}(\XL)=\binom mq,
 \qquad
 h^{N,q}(\XL)=\binom m{q-d}.
\]
\end{theorem}

\begin{proof}
For $p=0$, the only support is $I=\varnothing$, and
\[
 \widetilde H^{-1}(K_\varnothing;\C)=\C.
\]
The multiplicative Dolbeault--Cartan model identifies the surviving
closed antiholomorphic generators with $\GammaLCdual$, giving the exterior
algebra statement.  If $p>0$ and $q=0$, then $a=s=0$ would be required,
but $K_I$ is nonempty and has no reduced cohomology in degree $-1$.
Finally, $p=N$ forces $I=[N]$, and $K\simeq S^{d-1}$ has reduced
cohomology only in degree $d-1$, so $a=d$ and $s=q-d$.
\end{proof}

On the regular LVMB torus-bundle locus, the first-edge identities agree
with the kernel formula and $h^{0,1}=m$ obtained from Thiella's bigraded
model \cite[Proposition~2.8]{Thiella2026}.

\begin{corollary}[Irregularity, geometric genus, and non-K\"ahlerness]
\label{cor:non-Kahler-signature}
One has
\[
 h^{0,1}(\XL)=m=N-d,
 \qquad
 h^{1,0}(\XL)=0,
 \qquad
 h^{N,0}(\XL)=0.
\]
In particular, every full-rank minimally stable polytopal LVM manifold in
the range $d\geq1$ is non-K\"ahler.
\end{corollary}

\begin{proof}
The numerical statements follow from
Theorem~\ref{thm:edge-Dolbeault-groups}.  Since $m\geq1$, the bidegrees
$(1,0)$ and $(0,1)$ have different dimensions, contradicting the Hodge
symmetry of a compact K\"ahler manifold.
\end{proof}

\subsection{Alexander--Serre duality and polynomial reciprocity}

For a proper nonempty subset $I\subset[N]$, normalization of the
barycentric coordinates indexed by $I^c$ gives a deformation retraction
\[
 |K|\setminus |K_I|\simeq |K_{I^c}|.
\]
Topological Alexander duality in the sphere $|K|\cong S^{d-1}$ therefore
gives
\[
 \widetilde H^{j}(K_I;\C)
 \cong
 \widetilde H^{d-j-2}(K_{I^c};\C)^\vee
 \otimes
 \widetilde H^{d-1}(K;\C).
\]
The same formula holds for $I=\varnothing$ and $I=[N]$ by the augmented
convention.  This is the induced-subcomplex form of combinatorial
Alexander duality; compare
\cite[Theorem~1.1 and the discussion following it]{BjornerTancer2009}.

\begin{theorem}[Curvature duality]
\label{thm:curvature-duality}
Choose the curvature--Tor comparison of Remark~\ref{rem:first-bridge-choice-control}. Alexander duality on complementary supports then gives the vector-space dualities
\[
 H^{p,a}(\mathscr K_\Lambda)
 \cong
 H^{N-p,d-a}(\mathscr K_\Lambda)^\vee
 \otimes
 \widetilde H^{d-1}(K;\C).
\]
Consequently,
\[
 \kappa_\Lambda^{p,a}
 =
 \kappa_\Lambda^{N-p,d-a}.
\]
\end{theorem}

\begin{proof}
Choose the internal-degree-preserving curvature--Tor comparison of
Remark~\ref{rem:first-bridge-choice-control}.  On the Tor side, the
Hochster decomposition is canonically multigraded by supports $I$ of
cardinality $p$.  Apply Alexander duality to the summand indexed by $I$
and pair it with the summand indexed by $I^c$, of cardinality $N-p$ and
curvature degree $d-a$.  Transporting the resulting duality back to
curvature homology uses the chosen comparison, whereas the numerical
symmetry is independent of every choice.
\end{proof}

Exterior duality gives
\[
 \Lambda^s\GammaLCdual
 \cong
 \left(\Lambda^{m-s}\GammaLCdual\right)^\vee
 \otimes
 \Lambda^m\GammaLCdual.
\]

\begin{theorem}[Combinatorial Serre symmetry]
\label{thm:combinatorial-Serre-symmetry}
The Hochster formula gives the Serre symmetry
\[
 h^{p,q}(\XL)=h^{N-p,N-q}(\XL).
\]
Moreover, analytic Serre duality supplies the canonical perfect pairing
\[
 H_{\bar\partial}^{p,q}(\XL)
 \otimes H_{\bar\partial}^{N-p,N-q}(\XL)
 \longrightarrow H_{\bar\partial}^{N,N}(\XL)\cong\C.
\]
\end{theorem}

\begin{proof}
Combine the dimension statement in
Theorem~\ref{thm:curvature-duality} with exterior duality and use
$d+m=N$.  The top-degree group is
\[
 H_{\bar\partial}^{N,N}(\XL)
 \cong
 \widetilde H^{d-1}(K;\C)
 \otimes
 \Lambda^m\GammaLCdual,
\]
which is one-dimensional.  This is Serre's canonical perfect pairing \cite{Serre1955}. The support summands express it in the
chosen Hochster decomposition.
\end{proof}

\begin{corollary}[Reciprocity]
\label{cor:Dolbeault-reciprocity}
The Hochster and Dolbeault polynomials satisfy
\[
 \mathscr H_K(u,v)
 =
 u^Nv^d\mathscr H_K(u^{-1},v^{-1}),
\]
and
\[
 \operatorname{Dol}_{\XL}(u,v)
 =
 u^Nv^N
 \operatorname{Dol}_{\XL}(u^{-1},v^{-1}).
\]
Moreover,
\[
 H^N(\XL,\Omega_{\XL}^p)=0
 \quad(p<N),
 \qquad
 H^N(\XL,\Omega_{\XL}^N)\cong\C.
\]
\end{corollary}

\begin{proof}
Pair the coefficient of $u^pv^a$ in $\mathscr H_K$ with the coefficient
of $u^{N-p}v^{d-a}$ using Alexander duality.  Multiply by
$(1+v)^m$ and use $d+m=N$ for the Dolbeault reciprocity.  The top-row
statement follows from Serre symmetry and the vanishing of positive-degree
holomorphic forms.
\end{proof}

\subsection{Exterior factors and vanishing genera}

\begin{theorem}[Divisibility of the Hodge rows]
\label{thm:vertical-binomial-divisibility}
For every $0\leq p\leq N$, the row polynomial $H_p(v)$ is divisible by
$(1+v)^m$.  Equivalently, for every integer $0\leq\nu<m$,
\[
 \sum_{q=0}^{N}
 (-1)^q\binom q\nu h^{p,q}(\XL)=0.
\]
\end{theorem}

\begin{proof}
The divisibility is equation~\eqref{eq:row-binomial-packet}.  Its
$\nu$-th derivative vanishes at $v=-1$ for $\nu<m$.  Since
\[
 \frac{1}{\nu!}H_p^{(\nu)}(-1)
 =
 (-1)^\nu
 \sum_q(-1)^q\binom q\nu h^{p,q}(\XL),
\]
the stated identities follow.
\end{proof}

\begin{corollary}[Vanishing holomorphic Euler characteristics]
\label{cor:vanishing-chi-y}
For every $p$,
\[
 \chi(\XL,\Omega_{\XL}^p)
 =
 \sum_q(-1)^qh^{p,q}(\XL)
 =0.
\]
Consequently, the Hirzebruch $\chi_y$-genus vanishes identically:
\[
 \chi_y(\XL)
 =
 \sum_{p=0}^{N}
 \chi(\XL,\Omega_{\XL}^p)y^p
 =0.
\]
In particular,
\[
 \chi(\XL,\mathcal O_{\XL})=0.
\]
\end{corollary}

\begin{proof}
Take $\nu=0$ in
Theorem~\ref{thm:vertical-binomial-divisibility}, and then sum over $p$.
\end{proof}

\subsection{Moment-angle cohomology, Betti numbers, and Fr\"olicher degeneration}

The additive cohomology calculation builds on the work of Goresky--MacPherson and Jewell on arrangement complements
\cite{GoreskyMacPherson1988,Jewell1994}.  Bosio--Meersseman give the
integral calculation in polytope form, together with the induced-complex
homology formulation
\cite[Theorem~10.1 and Remark~10.2]{BosioMeersseman2006}.
Their published proof uses the cohomology of moment-angle manifolds.
An earlier proof used the Goresky--MacPherson formula and de Longueville's
product formula \cite{DeLongueville2000}, as explained in
\cite[Remark~10.7]{BosioMeersseman2006}.
In the Hochster--Tor notation for the visible moment-angle manifold
\cite{Panov2008,BuchstaberPanov2015}, its complexification is
\[
 H^{\ell}(\mathcal Z_K;\C)
 \cong
 \bigoplus_{I\subseteq[N]}
 \widetilde H^{\ell-|I|-1}(K_I;\C).
\]

\begin{theorem}[Curvature groups and moment-angle cohomology]
\label{thm:curvature-moment-angle}
For every $\ell$,
\[
 H^{\ell}(\mathcal Z_K;\C)
 \cong
 \bigoplus_{p+a=\ell}
 H^{p,a}(\mathscr K_\Lambda).
\]
Consequently,
\[
 \sum_{p+a=\ell}\kappa_\Lambda^{p,a}
 =
 b_\ell(\mathcal Z_K),
\]
and
\[
 \operatorname{Poin}_{\mathcal Z_K}(t)
 =
 \mathscr H_K(t,t).
\]
\end{theorem}

\begin{proof}
For $|I|=p$ and $\ell=p+a$,
\[
 \ell-|I|-1=a-1.
\]
Thus the moment-angle and curvature--Hochster decompositions have the same
summands.
\end{proof}

By Theorem~\ref{thm:smooth-moment-angle-model},
\[
 \XL\cong_{\mathrm{homeo}}\mathcal Z_K\times T^m.
\]

\begin{theorem}[Betti numbers and total Hodge dimensions]
\label{thm:Betti-Hodge-equality}
\label{thm:Frolicher-degeneration}
The Poincar\'e polynomial is
\[
 \operatorname{Poin}_{\XL}(t)
 =
 (1+t)^m\mathscr H_K(t,t)
 =
 \operatorname{Dol}_{\XL}(t,t).
\]
Equivalently,
\[
 b_\ell(\XL)
 =
 \sum_{I\subseteq[N]}
 \sum_{a=0}^{d}
 \widetilde b_{a-1}(K_I)
 \binom m{\ell-|I|-a},
\]
where the binomial coefficient is zero outside its natural range.  Hence
for every $\ell$,
\[
 \sum_{p+q=\ell}h^{p,q}(\XL)=b_\ell(\XL).
\]
The Fr\"olicher spectral sequence degenerates at $E_1$, as established in
Theorem~\ref{thm:exact-full-rank-criterion}.
\end{theorem}

\begin{proof}
Use the product homeomorphism with $T^m$ and
Theorem~\ref{thm:curvature-moment-angle}.  The coefficient formula follows
by expanding $(1+t)^m$; the diagonal Dolbeault equality follows from
\[
 \operatorname{Dol}_{\XL}(u,v)
 =
 (1+v)^m\mathscr H_K(u,v).
\]
\end{proof}

\begin{corollary}[Low Betti numbers and Euler characteristic]
\label{cor:Betti-Euler-consequences}
One has
\[
 b_1(\XL)=b_{2N-1}(\XL)=m,
\]
and
\[
 \chi_{\mathrm{top}}(\XL)=0.
\]
More generally, for $0\leq\nu<m$,
\[
 \sum_{\ell=0}^{2N}
 (-1)^\ell\binom\ell\nu b_\ell(\XL)=0.
\]
In particular, the total even and odd Betti numbers are equal.
\end{corollary}

\begin{proof}
The Hochster decomposition gives $b_0(\mathcal Z_K)=1$ and
$b_1(\mathcal Z_K)=0$: in degree one, the empty support contributes
$\widetilde H^0(K_\varnothing)=0$, while a singleton support contributes
$\widetilde H^{-1}$ of a point.  Hence the coefficient of $t$ in
$(1+t)^m\operatorname{Poin}_{\mathcal Z_K}(t)$ is $m$.  The value of
$b_{2N-1}$ follows from Poincar\'e duality.  The remaining identities
follow by differentiating the factor $(1+t)^m$ at $t=-1$.
\end{proof}

\subsection{Join multiplicativity}

Let $K_i=\partial Q_i$ be simplicial polytopal $(d_i-1)$-spheres on
disjoint vertex sets of cardinalities $N_i$, and put
$m_i=N_i-d_i$.  For
$I=I_1\sqcup I_2$,
\[
 K_I=(K_1)_{I_1}*(K_2)_{I_2}.
\]
The reduced cohomology of joins gives
\[
 \widetilde H^{a-1}(L*M)
 \cong
 \bigoplus_{i+j=a}
 \widetilde H^{i-1}(L)
 \otimes
 \widetilde H^{j-1}(M).
\]

\begin{theorem}[Join multiplicativity]
\label{thm:join-multiplicativity}
One has
\[
 \mathscr H_{K_1*K_2}(u,v)
 =
 \mathscr H_{K_1}(u,v)
 \mathscr H_{K_2}(u,v).
\]
Consequently, if $X$ is a full-rank minimally stable polytopal LVM
manifold with visible sphere $K_1*K_2$, then
\[
 \operatorname{Dol}_{X}(u,v)
 =(1+v)^{m_1+m_2}
 \mathscr H_{K_1}(u,v)\mathscr H_{K_2}(u,v).
\]
If $X_i$ are full-rank realizations with visible spheres $K_i$, this is
\(
 \operatorname{Dol}_{X}=\operatorname{Dol}_{X_1}\operatorname{Dol}_{X_2}
\).
\end{theorem}

\begin{proof}
The induced complex on $I_1\sqcup I_2$ is the join of the two induced
complexes.  Sum the reduced-cohomology join formula over all supports and
obtain the combinatorial identity.  Since
\[
 N=N_1+N_2,\qquad d=d_1+d_2,\qquad m=m_1+m_2,
\]
the full-rank Dolbeault statement follows by multiplying by
$(1+v)^{m_1+m_2}$.
\end{proof}

\section{Full-rank examples}
\label{sec:explicit-examples}

Assume throughout this section that $B$ is invertible.  We evaluate
the Hochster formula for simplices, products of simplices, cubes,
and polygons.  Write
\[
 H_p(v)=\sum_q h^{p,q}(\XL)v^q.
\]
For the square we use the full-rank realization of
Proposition~\ref{prop:full-rank-square-companion}.  In each case,
\[
 \operatorname{Dol}_{\XL}(u,v)=(1+v)^m\mathscr H_K(u,v),
 \qquad m=N-d.
\]
The Hopf, square, and polygonal formulas overlap with
\cite{KatzarkovLeeLupercioMeersseman2025}; the precise comparisons
are given below.

\subsection{The simplex}

Let $P=\Delta^d$.  Then
\[
 N=d+1,
 \qquad
 m=1,
 \qquad
 K=\partial\Delta^d.
\]
Every proper nonempty induced subcomplex of $K$ is a simplex.  The entire
Hodge diamond is therefore carried by the empty and full supports.

\begin{proposition}[Simplex formula]
\label{prop:simplex-example}
For the full-rank LVM manifold associated with $\Delta^d$,
\[
 \mathscr H_K(u,v)=1+u^{d+1}v^d
\]
and
\[
 \operatorname{Dol}_{\XL}(u,v)
 =
 (1+v)(1+u^{d+1}v^d).
\]
Equivalently, the only nonzero Hodge numbers are
\[
 h^{0,0}=h^{0,1}=h^{d+1,d}=h^{d+1,d+1}=1.
\]
\end{proposition}

For $d=1$, this is the Hopf-surface polynomial
$(1+v)(1+u^2v)$, namely the $m=1$ case of
\cite[formulas~(4) and~(10)]{KatzarkovLeeLupercioMeersseman2025}.
The diagonal specialization gives the Poincaré polynomial.  The smooth
model gives
\[
 \mathcal Z_K\cong S^{2d+1},
 \qquad
 \XL\cong_{\mathrm{diff}}S^{2d+1}\times S^1.
\]
Accordingly,
\[
 \operatorname{Dol}_{\XL}(t,t)
 =
 (1+t)(1+t^{2d+1}),
\]
the Poincaré polynomial of this smooth manifold.

\subsection{Products of simplices}

Let
\[
 P=\Delta^{d_1}\times\cdots\times\Delta^{d_r},
 \qquad
 d=d_1+\cdots+d_r.
\]
Here every $d_j\geq1$.
Then
\[
 N=d+r,
 \qquad
 m=r,
\]
and the dual boundary complex is
\[
 K
 =
 \partial\Delta^{d_1}*\cdots*\partial\Delta^{d_r}.
\]

\begin{proposition}[Product-of-simplices formula]
\label{prop:product-simplices-example}
One has
\[
 \mathscr H_K(u,v)
 =
 \prod_{j=1}^r
 \left(1+u^{d_j+1}v^{d_j}\right)
\]
and therefore
\[
 \operatorname{Dol}_{\XL}(u,v)
 =
 (1+v)^r
 \prod_{j=1}^r
 \left(1+u^{d_j+1}v^{d_j}\right).
\]
\end{proposition}

\begin{proof}
Apply the simplex calculation to every factor and use join
multiplicativity.
\end{proof}

The same product structure is visible topologically.  The moment-angle
join identity gives
\[
 \mathcal Z_K
 \cong
 \prod_{j=1}^r S^{2d_j+1}.
\]
Hence
\[
 \XL
 \cong_{\mathrm{diff}}
 \left(\prod_{j=1}^rS^{2d_j+1}\right)\times T^r. \]
The corresponding Poincar\'e polynomial is
\[
 \operatorname{Dol}_{\XL}(t,t)
 =
 (1+t)^r
 \prod_{j=1}^r(1+t^{2d_j+1}).
\]

\paragraph{The triangular prism.}
For
\[
 P=\Delta^2\times\Delta^1,
\]
one has $(N,d,m)=(5,3,2)$ and
\[
 \mathscr H_K(u,v)
 =
 1+u^2v+u^3v^2+u^5v^3.
\]
The nonzero Hodge rows are
\[
 \begin{array}{c|c}
 p&H_p(v)\\ \hline
 0&(1+v)^2\\
 2&v(1+v)^2\\
 3&v^2(1+v)^2\\
 5&v^3(1+v)^2.
 \end{array}
\]
In particular,
\[
 \operatorname{Dol}_{\XL}(t,t)
 =
 (1+t)^2(1+t^3)(1+t^5),
\]
matching
\[
 \XL\cong_{\mathrm{diff}}S^3\times S^5\times T^2.
\]

\subsection{Cubes}

Let $P=[0,1]^d=(\Delta^1)^d$.  The complex $K$ is the join of $d$
copies of $S^0$.  Group its $2d$ vertices into the $d$ opposite pairs.
An induced subcomplex is contractible if it selects exactly one vertex
from some pair.  If it selects both vertices from exactly $k$ pairs and
no vertices from the other pairs, then it is
\[
 (S^0)^{*k}\simeq S^{k-1}.
\]

\begin{proposition}[Cube formula]
\label{prop:cube-example}
For the $d$-cube,
\[
 \mathscr H_K(u,v)
 =
 (1+u^2v)^d
\]
and
\[
 \operatorname{Dol}_{\XL}(u,v)
 =
 (1+v)^d(1+u^2v)^d.
\]
More precisely,
\[
 H_{2k}(v)
 =
 \binom dk v^k(1+v)^d,
 \qquad
 0\leq k\leq d,
\]
and every odd holomorphic row vanishes.
\end{proposition}

The smooth model is
\[
 \XL\cong_{\mathrm{diff}}(S^3)^d\times T^d,
\]
so
\[
 \operatorname{Dol}_{\XL}(t,t)
 =
 (1+t)^d(1+t^3)^d
\]
is its Poincaré polynomial.

\paragraph{The three-cube.}
Here $(N,d,m)=(6,3,3)$ and
\[
 \mathscr H_K(u,v)
 =
 1+3u^2v+3u^4v^2+u^6v^3.
\]
The Hodge rows are
\[
 \begin{array}{c|c}
 p&H_p(v)\\ \hline
 0&(1+v)^3\\
 2&3v(1+v)^3\\
 4&3v^2(1+v)^3\\
 6&v^3(1+v)^3.
 \end{array}
\]
The coefficient $3u^4v^2$ comes from the three choices of two opposite
pairs.  Each corresponding proper four-vertex induced subcomplex is
\[
 S^0*S^0\simeq S^1.
\]
These are proper supports with nonzero $\widetilde H^1$.
The remaining nonzero support contributions come from the empty set, the three disconnected pairs of opposite vertices, and the full sphere $K$.

\subsection{The polygonal family}

For polygons the entire proper-support contribution is governed by the
number of connected components of an induced subgraph.  Let $K=C_l$ with
$l\geq3$. Then
\[
 N=l,
 \qquad
 d=2,
 \qquad
 m=l-2.
\]
For
\[
 1\leq p\leq l-1,
\]
set
\[
 a_{l,p}
 =
 \sum_{|I|=p}
 \widetilde b_0((C_l)_I).
\]

\begin{proposition}
\label{prop:polygon-run-count}
\[
 a_{l,p}
 =
 l\binom{l-2}{p-1}
 -
 \binom lp.
\]
\end{proposition}

\begin{proof}
For a nonempty proper subset $I$, let $r(I)$ be the number of cyclic runs.
Then
\[
 \widetilde b_0((C_l)_I)=r(I)-1.
\]
A fixed vertex begins a run when it is selected and its predecessor is not;
the remaining $p-1$ selected vertices are chosen among the other $l-2$
vertices. Thus
\[
 \sum_{|I|=p}r(I)
 =
 l\binom{l-2}{p-1}.
\]
Subtract $\binom lp$.
\end{proof}

The polygonal calculation was considered in
\cite[Theorems~1 and~3, formulas~(8)--(9)]{KatzarkovLeeLupercioMeersseman2025}.
The theorem below gives its coefficients directly from induced
subcomplexes, at the minimally stable value $m=l-2$.  In particular, the
proper-support polynomial is
\[
 A_l(u):=\sum_{p=2}^{l-2}a_{l,p}u^p
 =1+u^l-(1+u)^l+lu(1+u)^{l-2}.
\]
For $l\geq4$, the normalization in formulas~(8)--(9) of that paper is
$s_lu^2P_{l-4}=A_l$, where $s_l=l(l-3)/2$.\footnote{On p.~5 of
\cite{KatzarkovLeeLupercioMeersseman2025}, the definition following~(8)
has $2u$ in place of $2u^l$, and formula~(12) substitutes
$(1+u)^{l-4}$ for $P_{l-4}$.  For the hexagon the correct polynomial is
$A_6=9u^2+16u^3+9u^4$, whereas that substitution gives a middle
coefficient of $18$.}
The comparison here is at full curvature rank.  The polynomial
$P_{l-4}$ in that paper is normalized by $s_l$; the appendix
below uses the unnormalized polynomial $A_l/u^2$.  For $l=3$, the identity
gives $A_3=0$ and recovers the simplex calculation.

\begin{theorem}[Polygonal formula]
\label{thm:polygonal-example}
For the full-rank minimally stable $l$-gon,
\[
 \operatorname{Dol}_{X_l}(u,v)
 =
 (1+v)^{l-2}
 \left[
 1+
 v\sum_{p=2}^{l-2}
 \left(
 l\binom{l-2}{p-1}-\binom lp
 \right)u^p
 +
 u^lv^2
 \right].
\]
\end{theorem}

\begin{proof}
The empty support contributes $1$.  Every proper nonempty induced subgraph
is a disjoint union of paths and contributes only reduced degree zero.
The full cycle contributes
\[
 \widetilde H^1(C_l;\C)=\C.
\]
\end{proof}

\paragraph{The square.}
The polynomial below is also the $m=2$ specialization of
\cite[formula~(23)]{KatzarkovLeeLupercioMeersseman2025}.
There are exactly two disconnected two-vertex induced subcomplexes, the
two pairs of opposite vertices.  Thus
\[
 \mathscr H_{C_4}(u,v)
 =
 1+2u^2v+u^4v^2
\]
and
\[
 \operatorname{Dol}_{X_4}(u,v)
 =
 (1+v)^2(1+2u^2v+u^4v^2).
\]
The Hodge table is given in
Proposition~\ref{prop:full-rank-square-companion}.
The diagonal specialization is
\[
 (1+t)^2(1+t^3)^2,
\]
matching
\[
 \mathcal Z_{C_4}\cong S^3\times S^3,
 \qquad
 X_4\cong_{\mathrm{diff}}S^3\times S^3\times T^2.
\]

\paragraph{The pentagon.}
The run count gives
\[
 a_{5,2}=a_{5,3}=5.
\]
Hence
\[
 \mathscr H_{C_5}(u,v)
 =
 1+5u^2v+5u^3v+u^5v^2
\]
and
\[
 \operatorname{Dol}_{X_5}(u,v)
 =
 (1+v)^3(1+5u^2v+5u^3v+u^5v^2).
\]
The nonzero rows are
\[
 \begin{array}{c|c}
 p&H_p(v)\\ \hline
 0&(1+v)^3\\
 2&5v(1+v)^3\\
 3&5v(1+v)^3\\
 5&v^2(1+v)^3.
 \end{array}
\]
On the diagonal,
\[
 \operatorname{Dol}_{X_5}(t,t)
 =
 (1+t)^3(1+5t^3+5t^4+t^7),
\]
which matches the standard smooth identification
\[
 \mathcal Z_{C_5}
 \cong
 \mathop{\#}_{5}(S^3\times S^4)
\]
together with the torus factor $T^3$.

\subsection{Exact induced-subcomplex enumeration}

The accompanying program \texttt{induced\_subcomplex\_enumeration.py} computes the augmented boundary matrices
of every $K_I$ over $\mathbb Q$, sums their reduced Betti numbers
in Hochster bidegrees, and multiplies by $(1+v)^{N-d}$.  The results are listed below.

\paragraph{Induced-subcomplex calculation.}
\label{prop:computer-example-verification}
Exact enumeration of all induced subcomplexes gives
\[
 \begin{array}{c|c|c}
 \text{polytope}&(N,d,m)&\mathscr H_K(u,v)\\ \hline
 \Delta^2&(3,2,1)&1+u^3v^2\\
 \text{square}&(4,2,2)&1+2u^2v+u^4v^2\\
 \text{pentagon}&(5,2,3)&1+5u^2v+5u^3v+u^5v^2\\
 \text{three-cube}&(6,3,3)&1+3u^2v+3u^4v^2+u^6v^3\\
 \text{triangular prism}&(5,3,2)&1+u^2v+u^3v^2+u^5v^3.
 \end{array}
\]
In each case, the Hodge numbers satisfy
\[
 h^{p,q}=h^{N-p,N-q},
 \qquad
 h^{0,q}=\binom mq,
 \qquad
 h^{N,N}=1.
\]
The diagonal specialization of the Hodge polynomial equals the
Poincar\'e polynomial of the corresponding smooth model. The enumeration reproduces the formulas above, including the reduced-degree shift and the empty-support convention.

\section{Tangent-sheaf cohomology}
\label{sec:tangent-sheaf}

The tangent calculation starts from the same Cox covering as the
calculation for forms.  Its shifted pole condition allows nonzero
resonant sectors to survive.  We compute these sectors, including
their positive \v{C}ech degrees.

Meersseman studied global holomorphic vector fields in the original
construction \cite{Meersseman2000}; fundamental fields and their zero
loci are treated in \cite[Theorem~4]{BiswasDumitrescuMeersseman2020}.
The formula includes the invariant polynomial vector fields in degree zero and the cohomology of shifted pole complexes in higher \v{C}ech degrees.

Throughout this section, $\XL$ is minimally stable and polytopal, $d\geq1$,
and $B$ is invertible.  Thus $\pi:U_K\to\XL$ is the holomorphic covering of
Theorem~\ref{thm:full-rank-Cox-quotient}, with deck group $\GammaL\cong\Z^m$.
Write $G^\vee=\GammaLCdual$, and retain the ordered deck basis used in
Section~\ref{sec:derived-deck-descent}.  For $r\in\Z^N$, put
\[
 \chi_r(\gamma)=\exp\left(\sum_{j=1}^N r_jA_j(\gamma)\right),
 \qquad
 M_{\mathrm{res}}=\{r\in\Z^N:\chi_r=1\}.
\]
The regularity condition is expressed by
\[
 E_i=z_i\frac{\partial}{\partial z_i},\qquad
 z^rE_i=z^{r+e_i}\frac{\partial}{\partial z_i},\qquad
 T_i(r)=\{j:r_j+\delta_{ij}<0\},
\]
where $e_i$ is the $i$th coordinate vector in $\Z^N$.  As before,
$\widetilde H^{-1}(K_\varnothing;\C)=\C$.

\begin{theorem}[Tangent-sheaf sector formula]
\label{thm:tangent-sector-formula}
Let $\XL$ be minimally stable and polytopal with $d\geq1$, and suppose
$B$ is invertible.  For every $q\geq0$ there is an additive isomorphism
\begin{equation}\label{eq:tangent-sector-formula}
 H^q(\XL,\Theta_{\XL})
 \cong
 \bigoplus_{r\in M_{\mathrm{res}}}\ \bigoplus_{i=1}^N
 \ \bigoplus_{\substack{0\leq a\leq d,\ 0\leq s\leq m\\a+s=q}}
 \widetilde H^{a-1}(K_{T_i(r)};\C)\otimes\Lambda^sG^\vee.
\end{equation}
Only finitely many pairs $(r,i)$ contribute nonzero cohomology.  With the
chosen Cox coordinates and deck basis, the formula is induced by an
isomorphism in $D(\C)$
\begin{equation}\label{eq:tangent-derived-formula}
 R\Gamma(\XL,\Theta_{\XL})
 \simeq
 \bigoplus_{r\in M_{\mathrm{res}}}\ \bigoplus_{i=1}^N
 \widetilde C^{\bullet-1}(K_{T_i(r)};\C)
 \otimes\Lambda^\bullet G^\vee.
\end{equation}
The sum on the right is algebraic.  The exterior factor has zero
differential; an element with reduced cochain degree $a-1$ and exterior
degree $s$ has total degree $a+s$.
\end{theorem}

The shifted Laurent estimates and Euler division reduce the completed descent
complex to an algebraic sum of sectors.

\subsection{The shifted pole complex}

The covering differential identifies
$\pi^*\Theta_{\XL}$ with $\Theta_{U_K}$.  On a Cox intersection $U_\tau$,
the standard frame $\partial/\partial z_i$ gives
\[
 z^rE_i\in\Gamma(U_\tau,\Theta_{U_K})
 \quad\Longleftrightarrow\quad
 r_j+\delta_{ij}\geq0\quad(j\in\tau).
\]
Let $\mathscr Q_{r,i}^\bullet$ be the coefficient complex of this term
in $\check C^\bullet(\mathcal U_K,\Theta_{U_K})$.  With the notation of
Section~\ref{sec:Cox-Cech-sectors}, regularity gives
\begin{equation}\label{eq:tangent-relative-sector}
 \mathscr Q_{r,i}^\bullet
 \cong C^\bullet(\Delta_{\max},\mathfrak D_{T_i(r)};\C).
\end{equation}
Indeed, the coefficient line occurs on a normalized \v{C}ech simplex
$[\sigma_0,\ldots,\sigma_a]$ exactly when
$T_i(r)\cap\bigcap_\nu\sigma_\nu=\varnothing$.  Restrictions preserve the
monomial and have the oriented-simplex signs.  Thus \eqref{eq:tangent-relative-sector} is an isomorphism of cochain
complexes.

For nonempty $T$, the subcomplex $\mathfrak D_T$ is covered by the
simplices of maximal faces containing each vertex of $T$.  A finite
intersection is a simplex precisely when the corresponding vertices
span a face of $K_T$.  The nerve is therefore $K_T$.  The relative
cohomology sequence of the contractible simplex $\Delta_{\max}$ gives
\begin{equation}\label{eq:tangent-fixed-cohomology}
 H^a(\mathscr Q_{r,i}^\bullet)
 \cong\widetilde H^{a-1}(K_{T_i(r)};\C).
\end{equation}
For $T_i(r)=\varnothing$, the relative complex is the ordinary cochain
complex of $\Delta_{\max}$, with one copy of $\C$ in cohomological degree
zero.  This is exactly the augmented empty-complex convention.

\subsection{Analytic completion and descent}

For an allowed pair $(r,i)$ on $U_\tau$, set
\[
 \alpha=r+e_i,\qquad
 \nu_\tau(r,i)=\sum_{j\in\tau}\alpha_j
                  +\sum_{j\notin\tau}|\alpha_j|,
 \qquad
 \omega_{\tau,R}(r,i)=R^{\nu_\tau(r,i)}.
\]
The Taylor--Laurent expansion of a holomorphic vector field converges
normally on every compact subset of $U_\tau$.  In the standard frame its
coefficient seminorms are equivalent to
\[
 q_{\tau,R}(c)
 =\sum_{(r,i)\ \mathrm{allowed}}|c_{r,i}|\omega_{\tau,R}(r,i),
 \qquad R\geq2.
\]
Here and below integer radii suffice.  For $2\leq R<L$, translation
$r\mapsto r+e_i$ gives the exact estimate
\begin{equation}\label{eq:tangent-weight-sum}
 \sum_{(r,i)\ \mathrm{allowed}}
 \frac{\omega_{\tau,R}(r,i)}{\omega_{\tau,L}(r,i)}
 =N\left(\frac1{1-R/L}\right)^{|\tau|}
    \left(\frac{1+R/L}{1-R/L}\right)^{N-|\tau|}<\infty.
\end{equation}
If $M_{\tau,L}$ is the sum of the suprema of the standard coefficients
on $\mathfrak A_{\tau,L}$, Cauchy's formula gives
$|c_{r,i}|\leq M_{\tau,L}\omega_{\tau,L}(r,i)^{-1}$.
Together with \eqref{eq:tangent-weight-sum}, this proves continuity of
coefficient extraction.  Conversely, $M_{\tau,R}\leq q_{\tau,R}$ proves
normal convergence and continuity of reconstruction.  Restrictions
preserve the coefficients and their weights.  Consequently
\begin{equation}\label{eq:tangent-analytic-completion}
 \check C^\bullet(\mathcal U_K,\Theta_{U_K})
 \cong
 \mathop{\widehat{\bigoplus}}\limits_{r\in\Z^N,\,1\leq i\leq N}^{\,\mathrm{an}}
 \mathscr Q_{r,i}^\bullet.
\end{equation}
The completion is defined by the finite sums of these seminorms over
the \v{C}ech intersections.  Projection onto any set of pairs $(r,i)$ is continuous and decreases every coefficient seminorm.

For vector fields, deck pullback means
$\gamma^*V=(d\gamma)^{-1}(V\circ\gamma)$.  Thus
\[
 \gamma^*(z^rE_i)=\chi_r(\gamma)z^rE_i,
 \qquad
 [E_j,z^rE_i]=r_jz^rE_i.
\]
These identities also show that deck maps and Euler Lie derivatives
commute with restrictions and with all exponent projections.  They are
continuous: $\sum_j|r_j|\leq\nu_\tau(r,i)+1$ bounds a deck multiplier by
a constant times an increase of radius, and
$|r_j|+1\leq|\alpha_j|+2$ gives
\[
 (|r_j|+1)\omega_{\tau,R}(r,i)
 \leq C_{R,L}\omega_{\tau,L}(r,i).
\]

Put
\[
 \mathcal T_\Theta^\bullet
 =\operatorname{Tot}\Kos_{\GammaL}^\bullet
       \bigl(\check C^\bullet(\mathcal U_K,\Theta_{U_K})\bigr).
\]
The sheaf descent argument of
Section~\ref{sec:derived-deck-descent} applies to this tangent bundle.
On an evenly covered open set $\mathcal V$, one has
\[
 (\pi_*\Theta_{U_K})(\mathcal V)
 \cong\operatorname{Map}(\GammaL,\Theta_{\XL}(\mathcal V)),
\]
with the deck generators acting by shifts.  The bilateral-summation
resolution is exact on these opens and hence resolves $\Theta_{\XL}$.
Evenly covered coordinate balls have disjoint unions of Stein sheets
as inverse images, so Cartan's theorem~B gives
$R^b\pi_*\Theta_{U_K}=0$ for $b>0$.  Every Cox intersection is Stein
as well.  The saturated quotient cover is therefore acyclic for every
term of the group--Koszul resolution.  Its \v{C}ech total complex gives
\begin{equation}\label{eq:tangent-descent}
 R\Gamma(\XL,\Theta_{\XL})\simeq\mathcal T_\Theta^\bullet.
\end{equation}
On \v{C}ech degree $a$, the total differential is
\[
 d_{\mathrm{tot}}(c\otimes\eta)
 =\delta_{\check C}c\otimes\eta
  +(-1)^a\sum_{\mu=1}^m(\gamma_\mu^*-1)c\otimes\xi_\mu\wedge\eta.
\]

\subsection{Euler finite support}

Write $\mathcal T_{\Theta,r}^\bullet$ for the sum of the $N$ fixed-field
total complexes with exponent $r$.  The commuting operators
$\mathcal L_j=[E_j,-]$ act on this complex by the scalars $r_j$.
For $\lambda\in\C$, let $\Pi_{j,\lambda}$ project onto exponents with
$r_j=\lambda$, with zero projection if $\lambda\notin\Z$.  On its
complement, division by $(r_j-\lambda)^M$ is continuous, since
\[
 \inf_{\nu\in\Z,\,\nu\ne\lambda}|\nu-\lambda|>0.
\]
The inverse multiplier commutes with $d_{\mathrm{tot}}$ and is bounded
in every coefficient seminorm by the inverse $M$th power of this gap.

By \eqref{eq:tangent-descent} and coherent cohomology finiteness on
the compact manifold $\XL$, each $H^q(\mathcal T_\Theta^\bullet)$ is
finite-dimensional.  It therefore decomposes into simultaneous
generalized eigenspaces of the $\mathcal L_j$.  Suppose a closed cochain
$c$ represents a class in the generalized eigenspace with eigenvalue
$\lambda=(\lambda_1,\ldots,\lambda_N)$.  For each $j$, choose $M_j$ and
$b_j$ with
\[
 (\mathcal L_j-\lambda_j)^{M_j}c=d_{\mathrm{tot}}b_j.
\]
Applying $1-\Pi_{j,\lambda_j}$ and then the inverse multiplier gives
\[
 (1-\Pi_{j,\lambda_j})c
 =d_{\mathrm{tot}}\!\left(
   (\mathcal L_j-\lambda_j)^{-M_j}
   (1-\Pi_{j,\lambda_j})b_j\right),
\]
where the inverse is taken only on the complementary subcomplex.
Thus projection onto $r_j=\lambda_j$ preserves the class.  Successive
commuting projections give a representative on the single exponent
$r=\lambda\in\Z^N$, or show that the class is zero.  A general class
is a finite sum of such representatives.

Conversely, if a closed finite-support cochain $c$ satisfies
$c=d_{\mathrm{tot}}b$ in the completed complex, projection of $b$ onto
the finite exponent support of $c$ gives a finite-support primitive.
We have proved the quasi-isomorphism
\begin{equation}\label{eq:tangent-finite-support}
 \bigoplus_{r\in\Z^N}^{\mathrm{alg}}\mathcal T_{\Theta,r}^\bullet
 \ \longrightarrow\ \mathcal T_\Theta^\bullet.
\end{equation}
This argument uses the integral Euler spectrum and finite-dimensional
coherent cohomology. The character contractions then apply on each
finite exponent support.

\begin{proof}[Proof of Theorem~\ref{thm:tangent-sector-formula}]
For a fixed pair $(r,i)$, the group differential is exterior
multiplication by
$\sum_\mu(\chi_r(\gamma_\mu)-1)\xi_\mu$.
If this vector is nonzero, choose an index $\mu$ with nonzero
coefficient.  Contraction by the dual basis vector, divided by that
coefficient and multiplied by $(-1)^a$ on \v{C}ech degree $a$, contracts
the total sector.  If $r\in M_{\mathrm{res}}$, the group differential
is zero and the total sector is
$\mathscr Q_{r,i}^\bullet\otimes\Lambda^\bullet G^\vee$.
After \eqref{eq:tangent-finite-support}, these operations take place
in an algebraic direct sum.  Combining
\eqref{eq:tangent-descent} and \eqref{eq:tangent-fixed-cohomology} proves
\eqref{eq:tangent-sector-formula}.

For the derived statement, the relative simplex complex in
\eqref{eq:tangent-relative-sector} is quasi-isomorphic to the shifted
reduced cochains of $\mathfrak D_{T_i(r)}$.  The finite
Mayer--Vietoris double complex for its cover by simplices gives a zigzag
to the reduced cochains of its nerve $K_{T_i(r)}$, just as in
Proposition~\ref{prop:zero-reduced-cochain}.  Taking the algebraic sum
of these zigzags proves \eqref{eq:tangent-derived-formula}.
Finally, each degree on the left has finite dimension.  Hence only
finitely many pairs contribute in each degree; since $0\leq a\leq d$
and $0\leq s\leq m$, there are only finitely many contributing pairs
altogether.
\end{proof}

\subsection{Polynomial fields and the first cohomology group}

The halfspace theorem still locates the possible nonzero resonant
contributions.  Put $S(r)=\{j:r_j<0\}$.  For $0\ne r\in M_{\mathrm{res}}$,
the induced complex $K_{S(r)}$ is contractible by
Theorem~\ref{thm:nonzero-resonant-vanishing}, applied with $I=\varnothing$.
The shifted set satisfies
\[
 T_i(r)=
 \begin{cases}
  S(r)\setminus\{i\},&r_i=-1,\\
  S(r),&r_i\ne-1.
 \end{cases}
\]
Consequently, a nonzero resonant pair can contribute only when $r_i=-1$.
Its contribution is the reduced cohomology left after deleting that
vertex from $K_{S(r)}$.

\begin{corollary}\label{cor:tangent-low-degrees}
The global holomorphic vector fields have the finite monomial basis
\begin{equation}\label{eq:tangent-global-polynomials}
 H^0(\XL,\Theta_{\XL})
 \cong
 \bigoplus_{i=1}^N\ \bigoplus_{\substack{\alpha\in\Z_{\geq0}^N\\
                                      \chi_{\alpha-e_i}=1}}
 \C\,z^\alpha\frac{\partial}{\partial z_i}.
\end{equation}
Moreover, with the additive identifications above,
\begin{equation}\label{eq:tangent-first-cohomology}
 \begin{split}
 H^1(\XL,\Theta_{\XL})\cong{}&
 \bigl(H^0(\XL,\Theta_{\XL})\otimes G^\vee\bigr)\\
 &\quad\oplus
 \bigoplus_{r\in M_{\mathrm{res}}}\ \bigoplus_{i=1}^N
 \widetilde H^0(K_{T_i(r)};\C).
 \end{split}
\end{equation}
If $M_{\mathrm{res}}=\{0\}$, then
\[
 H^q(\XL,\Theta_{\XL})
 \cong\left(\bigoplus_{i=1}^N\C E_i\right)\otimes\Lambda^qG^\vee.
\]
\end{corollary}

\begin{proof}
Every visible label is a vertex of $K$.  Hence
$\widetilde H^{-1}(K_T)$ is nonzero exactly when $T=\varnothing$.
For a tangent sector this means $r+e_i\in\Z_{\geq0}^N$, giving
\eqref{eq:tangent-global-polynomials}; finiteness follows from the
theorem.  In degree one the possibilities are $(a,s)=(0,1)$ and
$(1,0)$, which give \eqref{eq:tangent-first-cohomology}.  For $r=0$,
every $T_i(0)$ is empty, proving the last assertion.
\end{proof}

\subsection{A resonant square}
\label{sec:tangent-square}

Let $m=2$, $N=4$, and let $K$ be the four-cycle with nonfaces
$\{1,3\}$ and $\{2,4\}$.  Consider the blocks
\begin{equation}\label{eq:tangent-square-blocks}
 V=\begin{pmatrix}1&0\\0&1\\1&0\\2&1\end{pmatrix},
 \qquad A=iV,\qquad B=I_2.
\end{equation}
These blocks arise from an admissible LVM configuration.  To see this
directly, put
\[
 c=\left(-\frac15-\frac{3i}{100},-\frac15-\frac{i}{100}\right),
 \qquad
 \Lambda_{a_0}=c,\quad
 \Lambda_{a_1}=c+(1,0),\quad
 \Lambda_{a_2}=c+(0,1),\quad
 \Lambda_j=c+iV_j.
\]
A normalized nonnegative relation among the weights, with coefficients
$s_0,s_1,s_2$ on the indispensable labels and $t_1,\ldots,t_4$ on the
visible labels, is equivalent to
\begin{equation}\label{eq:tangent-square-relations}
 s_1=s_2=\frac15,\qquad
 t_1+t_3+2t_4=\frac3{100},\qquad
 t_2+t_4=\frac1{100},\qquad
 s_0=\frac35-\sum_jt_j.
\end{equation}
The visible solutions form the quadrilateral
\[
 0\leq t_4\leq\frac1{100},\qquad
 0\leq t_3\leq\frac3{100}-2t_4,
 \qquad
 t_1=\frac3{100}-t_3-2t_4,\quad t_2=\frac1{100}-t_4.
\]
Here $s_0=14/25+2t_4\geq14/25$.  The interior has all seven
coefficients positive, so the Siegel condition holds.  Each relation
uses at least two visible labels: the vector $(3/100,1/100)$ is on none
of the four rays generated by the rows of $V$.  Together with the three
positive $s_\mu$, this gives at least five labels, proving weak
hyperbolicity for $2m=4$.  All three $a_\mu$ are indispensable, and each
visible coefficient vanishes on an edge of the quadrilateral.  Its
vertex zero sets are
\[
 \{3,4\},\quad\{1,4\},\quad\{2,3\},\quad\{1,2\}.
\]
Thus the visible complex is the required four-cycle, there are exactly
three indispensable labels, and the configuration is minimally stable.
The differences of the indispensable weights give $B=I_2$, so it has
full rank.  The quotient is a compact complex fourfold.

Write $q_0=e^{-2\pi}$.  The two generators of
$\GammaL=2\pi i\Z^2$ act by
\[
 \begin{aligned}
  g_1(z)&=(q_0z_1,z_2,q_0z_3,q_0^2z_4),\\
  g_2(z)&=(z_1,q_0z_2,z_3,q_0z_4).
 \end{aligned}
\]
Since $0<q_0<1$, the resonance equations are exactly
\[
 r_1+r_3+2r_4=0,\qquad r_2+r_4=0,
 \qquad
 r=(a,b,-a+2b,-b),\quad a,b\in\Z.
\]

For a polynomial field $z^\alpha\partial_i$, invariance is equivalent
to
\[
 \alpha_1+\alpha_3+2\alpha_4=V_{i1},\qquad
 \alpha_2+\alpha_4=V_{i2},\qquad\alpha_j\geq0.
\]
The solutions give the following basis of global fields:
\[
 \begin{gathered}
 z_1\partial_1,\quad z_3\partial_1,\quad
 z_1\partial_3,\quad z_3\partial_3,\quad
 z_2\partial_2,\quad z_4\partial_4,\\
 z_1^2z_2\partial_4,\quad z_1z_3z_2\partial_4,\quad
 z_3^2z_2\partial_4.
 \end{gathered}
\]
Indeed, for $i=1,3$ the second equation forces
$\alpha_2=\alpha_4=0$ and the first becomes $\alpha_1+\alpha_3=1$.
For $i=2$ there is just $\alpha=e_2$.  For $i=4$, either
$\alpha=e_4$, or $\alpha_4=0$, $\alpha_2=1$, and
$\alpha_1+\alpha_3=2$.  In particular, three fields are cubic.

\begin{proposition}\label{prop:tangent-square-count}
For the square configuration \eqref{eq:tangent-square-blocks},
\begin{equation}\label{eq:tangent-square-polynomial}
 \sum_{q\geq0}h^q(\XL,\Theta_{\XL})t^q
 =(9+t)(1+t)^2=9+19t+11t^2+t^3.
\end{equation}
The unique positive-\v{C}ech-degree resonant contribution is
\[
 r_*=(-1,-1,-1,1),\qquad i=2,\qquad T_2(r_*)=\{1,3\},
\]
and it occurs in \v{C}ech degree one.
\end{proposition}

\begin{proof}
The only nonempty induced subcomplexes of a four-cycle with nonzero
reduced cohomology are its two opposite pairs and the full cycle.
For $r=(a,b,-a+2b,-b)$, the shifted exponents in coordinates $2,4$
have sum $\delta_{i2}+\delta_{i4}\geq0$.  They cannot both be negative.
Thus neither $\{2,4\}$ nor the full cycle occurs as a tangent pole set.

If $T_i(r)=\{1,3\}$, the negative shifted exponents in these two
coordinates have sum $2b+\delta_{i1}+\delta_{i3}<0$, so $b<0$.
Regularity in coordinate $2$ forces $i=2$ and $b=-1$.  The remaining
strict inequalities are $a<0$ and $-a-2<0$, giving $a=-1$.
Conversely, this pair has exactly the pole set $\{1,3\}$, whose reduced
$H^0$ is one-dimensional.  Adding this class to the nine degree-zero pairs already enumerated gives
\v{C}ech cohomology of dimension nine in degree zero and one in degree one.
Theorem~\ref{thm:tangent-sector-formula} supplies the exterior factor
$(1+t)^2$.
\end{proof}

The additional class has a concrete \v{C}ech representative.  Order the
maximal faces as $12,14,23,34$ and set
\[
 W=z^{r_*}E_2=\frac{z_4}{z_1z_3}\partial_2.
\]
Give these four faces the scalar values $0,0,1,1$.  On an ordered pair of faces with distinct assigned values, use $W$ as the coefficient of a one-cochain;
on the other pairs use zero.  The intersections of faces with distinct assigned values invert both
$z_1$ and $z_3$, so these coefficients are holomorphic.  The scalar
coefficients are a simplicial coboundary on the full simplex, and hence
the one-cochain is closed.  Its fixed-field sector has no degree-zero
term: every maximal face contains $1$ or $3$, where $W$ has a pole.
It therefore represents a nonzero class.  Resonance makes its group
differential zero, and projection onto $(r_*,2)$ shows that it remains
nonzero in the completed descent complex.

The square has
\[
 h^0(\XL,\Theta_{\XL})=9,\qquad
 h^1(\XL,\Theta_{\XL})=19>2\cdot9.
\]
The extra dimension comes from the disconnected pole complex.
We now compute the Kodaira--Spencer map and the primary obstruction
using the quotient gluing maps.

\section{The Kodaira--Spencer map and deformations}
\label{sec:deformations}

Assume minimal stability, polytopality, $d\geq1$, and invertibility
of $B$.  Meersseman constructed deformation families with parameter spaces of dimension
$m(n-m-1)=Nm$, with universality under additional hypotheses
\cite{Meersseman2000}.  We compute the Kodaira--Spencer map of the
normalized family: its image is the zero-exponent part of
$H^1(\Theta)$.  This gives a criterion for semiuniversality with
smooth base.  The resonant square requires additional parameters
and has a nonzero primary obstruction.

\subsection{The normalized family}

Put $C_0=AB^{-1}\in\operatorname{Mat}_{N\times m}(\C)$.  A complex
linear change of the acting parameter replaces the indispensable
differences by the standard basis.  The weights then have the form
\[
 c,\ c+e_1,\ldots,c+e_m,\qquad c+(C_0)_1,\ldots,c+(C_0)_N,
\]
where rows are viewed as weights.  Keep $c$ and the indispensable
weights fixed and replace $C_0$ by a variable matrix $C$.

\begin{proposition}\label{prop:normalized-deformation-family}
There is a polydisk $S$ about $C_0$ on which these weights define a
proper holomorphic submersion $\mathscr X_S\to S$.  Every fibre is
minimally stable and has the same visible complex $K$.  The family has
the presentation
\begin{equation}\label{eq:normalized-deformation-family}
 \mathscr X_S\cong(S\times U_K)/G,\qquad G=2\pi i\Z^m,
 \qquad
 g_{\mu,C}(z)_i=e^{2\pi i C_{i\mu}}z_i .
\end{equation}
The $G$-action on the total space is free and proper.
\end{proposition}

\begin{proof}
The subsets of labels whose convex hull contains the origin are
locally constant at an admissible configuration.  Indeed, a subset
containing the origin contains a minimal such subset.  Carath\'eodory's
theorem and weak hyperbolicity imply that this minimal subset has
$2m+1$ elements and is a full-dimensional simplex with positive
barycentric coordinates for the origin.  Affine independence and
positivity persist under small perturbations.  For a subset not
containing the origin, the distance from its compact convex hull to the
origin is positive and remains positive after a small perturbation.
There are only finitely many subsets.  Shrinking $S$ preserves all
these tests simultaneously, and hence admissibility, the indispensable
labels, and $K$.

The proof of Lemma~\ref{lem:parametric-LVM-quotient} applies over a polydisk: the Hessian is taken in the acting parameter,
and the implicit-function theorem permits any finite-dimensional
parameter space.  The relative minimizing transversal is a smooth
submersion to $S$, and its intersection with the inverse image of a
compact subset of $S$ is compact.  This proves joint properness of the
connected action and properness of the holomorphic family.

Deprojectivization at the distinguished indispensable coordinate gives
$S\times U_K\times(\C^*)^m$, with action
\[
 T\cdot(C,z,w)=(C,e^{CT}z,e^Tw).
\]
Locally choose a logarithm of $w$ and set $T=-\log w$.  The slice $w=1$
meets each orbit, and two of its points represent the same orbit
exactly when they differ by $G=2\pi i\Z^m$.  These local gauges are
holomorphic in $C$ as well.  The argument of
Theorem~\ref{thm:full-rank-Cox-quotient} therefore gives the relative
covering and \eqref{eq:normalized-deformation-family}.  Joint
properness also follows by restricting the proper connected action
to this closed discrete subgroup and the slice.
\end{proof}

Let $\gamma_\mu=2\pi i e_\mu$ and let $\xi_\mu\in G^\vee$ be the dual
basis, so that $\xi_\mu(\gamma_\nu)=\delta_{\mu\nu}$.  Write
\[
 \mathfrak e=\bigoplus_{i=1}^N\C E_i,\qquad
 \mathfrak g=H^0(X_{C_0},\Theta_{X_{C_0}})
             =\mathfrak e\oplus\mathfrak g_{\ne0},
 \qquad
 \mathcal C_\Theta=
 \bigoplus_{r\in M_{\mathrm{res}},\,i}
       \widetilde H^0(K_{T_i(r)};\C).
\]
Here $\mathfrak g_{\ne0}$ is the span of the nonzero-exponent
polynomial fields in \eqref{eq:tangent-global-polynomials}.  These
are vector-space decompositions relative to the chosen coordinates.

\begin{theorem}[Kodaira--Spencer map]\label{thm:normalized-KS}
With deck pullback and the total differential of
Section~\ref{sec:tangent-sheaf}, the Kodaira--Spencer map of
\eqref{eq:normalized-deformation-family} is
\begin{equation}\label{eq:normalized-KS}
 \quad
 \operatorname{KS}_{C_0}(\dot C)
 =-2\pi i\sum_{i=1}^N\sum_{\mu=1}^m
       \dot C_{i\mu}\,[E_i\otimes\xi_\mu].
 \quad
\end{equation}
It is injective, with image $\mathfrak e\otimes G^\vee$, and
\begin{equation}\label{eq:normalized-KS-cokernel}
 \operatorname{coker}\operatorname{KS}_{C_0}
 \cong(\mathfrak g_{\ne0}\otimes G^\vee)\oplus\mathcal C_\Theta.
\end{equation}
For an unnormalized variation of the blocks, the matrix in
\eqref{eq:normalized-KS} is
\[
 \dot C=\dot A B^{-1}-AB^{-1}\dot B B^{-1}.
\]
\end{theorem}

\begin{proof}
The Kodaira--Spencer map is the connecting homomorphism of
\[
 0\longrightarrow\Theta_{X_{C_0}}
 \longrightarrow\Theta_{\mathscr X_S}|_{X_{C_0}}
 \longrightarrow\mathcal O_{X_{C_0}}\otimes T_{C_0}S
 \longrightarrow0.
\]
On $S\times U_K$, lift $\dot C$ by the constant horizontal vector
$H$.  If $F_\mu(C,z)=(C,g_{\mu,C}(z))$, differentiation of its inverse
block matrix gives
\[
 F_\mu^*H-H
 =-(dg_{\mu,C_0})^{-1}
      \left.\frac{d}{dt}\right|_{t=0}g_{\mu,C_0+t\dot C}
 =-2\pi i\sum_i\dot C_{i\mu}E_i .
\]
The horizontal lifts agree on all Cox intersections, so the
\v{C}ech-degree-one component vanishes.  The group differential is pullback
minus the identity, giving exactly \eqref{eq:normalized-KS}.

For $r=0$ all tangent pole sets are empty.  Theorem
~\ref{thm:tangent-sector-formula} therefore identifies this part of
$H^1(\Theta)$ with $\mathfrak e\otimes G^\vee$, with its displayed
basis.  This proves injectivity and the image assertion.
Corollary~\ref{cor:tangent-low-degrees} gives the cokernel.
Finally, differentiate $C=AB^{-1}$.  A common translation of the weights multiplies every projective
coordinate by the same scalar, preserving the quotient. Thus the
family of all diagonal weights has the same Kodaira--Spencer image.
\end{proof}

\subsection{A semiuniversal family}

We use deformations with a fixed identification of the central fibre.
A semiuniversal family is a versal family whose Kodaira--Spencer map
is an isomorphism. Universality further requires uniqueness of the
inducing map.
Kuranishi's existence theorem gives a semiuniversal family over an analytic
germ in $H^1(X,\Theta_X)$, with that vector space as its Zariski
tangent space \cite{Kuranishi1962}; see also
\cite[Section~I]{MeerssemanKuranishi2017} for the distinction between
completeness, semiuniversality, and universality, including nonreduced
bases.

\begin{theorem}[Semiuniversality of the normalized family]\label{thm:diagonal-semiuniversality}
Under the standing full-rank hypotheses, the following are equivalent:
\begin{enumerate}[label=\textup{(\roman*)}]
 \item the normalized family is infinitesimally complete at $C_0$;
 \item $h^1(X_{C_0},\Theta_{X_{C_0}})=Nm$;
 \item $\mathfrak g_{\ne0}=0$ and $\mathcal C_\Theta=0$.
\end{enumerate}
When they hold, the Kuranishi germ is smooth of dimension $Nm$, and
the normalized family is semiuniversal at $C_0$.  In particular,
$M_{\mathrm{res}}=\{0\}$ is sufficient for these conclusions.
The locus defined by \textup{(ii)} is open in $S$.
\end{theorem}

\begin{proof}
By Theorem~\ref{thm:normalized-KS}, the Kodaira--Spencer map is
injective of rank $Nm$, and its cokernel is
\eqref{eq:normalized-KS-cokernel}.  Since $m\geq1$, this proves the
equivalences.

Suppose they hold.  Let $\mathcal K\subset H^1(X_{C_0},\Theta_{X_{C_0}})$
be a Kuranishi germ, with its analytic structure.  Completeness gives
a holomorphic map of germs $f:(S,C_0)\to(\mathcal K,0)$ inducing the
normalized family.  The derivative of the composite
$S\to\mathcal K\hookrightarrow H^1(\Theta)$ is its Kodaira--Spencer
map, after the standard tangent identification, and is an isomorphism.
The holomorphic inverse-function theorem makes this composite a local
biholomorphism.  Every function in the defining ideal of $\mathcal K$
pulls back to zero.  Since the composite induces an isomorphism of
local analytic rings, the defining ideal of $\mathcal K$ is zero. Thus $\mathcal K$ is a smooth reduced germ and $f$ is a local isomorphism.  The normalized family is consequently
semiuniversal.

If $M_{\mathrm{res}}=\{0\}$, Corollary~\ref{cor:tangent-low-degrees}
gives (ii).  Finally, upper semicontinuity in the proper family gives
$h^1(X_C,\Theta_{X_C})\leq Nm$ near a point satisfying (ii), whereas
Theorem~\ref{thm:normalized-KS} gives the reverse inequality throughout
$S$.  This proves openness.
\end{proof}

The normalized family can be semiuniversal with smooth base even when $H^2(X,\Theta_X)\neq0$.  On the nonresonant
locus,
\[
 h^2(X,\Theta_X)=N\binom m2,
\]
which is positive for $m\geq2$.  The locus $h^1(\Theta)=Nm$ is open in the normalized parameter space and contains the nonresonant configurations. Nonresonance is specified by countably many character conditions. When the equality fails, the
cokernel of the Kodaira--Spencer map records the infinitesimal
deformations outside the normalized family.

\subsection{Second-order descent}
\label{sec:second-order-descent}

The sector isomorphism of Theorem~\ref{thm:tangent-sector-formula} is
additive. We compute the primary obstruction by expanding the chart
transitions and deck transformations to second order.  Let
$\operatorname{ob}_2(\alpha)\in H^2(X,\Theta_X)$ denote the primary
obstruction to lifting the first-order deformation $\alpha$ from
$\C[\epsilon]/(\epsilon^2)$ to $\C[\epsilon]/(\epsilon^3)$.
We use the convention $\operatorname{ob}_2(\alpha)=
\tfrac12[\alpha,\alpha]$.

These descent equations compute obstructions to arbitrary
deformations of $X$.  Indeed, pull an infinitesimally deformed
structure sheaf back to the underlying topological covering
$U_K\to X$.  This gives a $G$-equivariant deformation of $U_K$.
The Stein Cox charts admit infinitesimal trivializations, by the
vanishing of their first tangent cohomology.  Their transition maps
and the lifted deck maps then describe the deformation.  On
extending a trivialization by one order, its failure to satisfy
descent is a total degree-two cocycle; changing the extension adds
a total coboundary.  When the Cox cover deforms, the corrections include changes to both its chart
transitions and its deck maps.

\begin{proposition}\label{prop:deck-primary-obstruction}
For $\alpha=\sum_\mu V_\mu\otimes\xi_\mu$, with
$V_\mu\in\mathfrak g$, one has
\begin{equation}\label{eq:deck-primary-obstruction}
 \operatorname{ob}_2(\alpha)=
 \sum_{\mu<\nu}[V_\mu,V_\nu]\otimes\xi_\mu\wedge\xi_\nu.
\end{equation}
In particular, if $m\geq2$ and $\mathfrak g$ is nonabelian, $X$ has
an obstructed infinitesimal deformation.
\end{proposition}

\begin{proof}
Lift the global fields to $G$-invariant fields on $U_K$.  To first
order, deform the $\mu$th deck map by
$\exp(-\epsilon V_\mu)\circ g_\mu$.  Its horizontal-lift cocycle is
$V_\mu$, with the sign convention in
Theorem~\ref{thm:normalized-KS}.  Write a second-order lift on the
$j$th Cox chart as
\[
 F_{\mu,j}
 =g_\mu\circ\exp(-\epsilon V_\mu-\epsilon^2A_{\mu,j}).
\]
The first-order expression is unchanged, since $V_\mu$ is deck
invariant.  Point maps compose from right to left, and their
pullbacks act in the reverse order.  On functions this gives
\[
 e^{-\epsilon V_\nu}e^{-\epsilon V_\mu}
 -e^{-\epsilon V_\mu}e^{-\epsilon V_\nu}
 =-\epsilon^2[V_\mu,V_\nu]\pmod{\epsilon^3}.
\]
Conjugation by $g_\nu$ replaces $A_{\mu,j}$ by $g_\nu^*A_{\mu,j}$.
The relation $F_{\mu,j}F_{\nu,j}=F_{\nu,j}F_{\mu,j}$ is therefore
equivalent, through order two, to
\begin{equation}\label{eq:deck-second-correction}
 (g_\mu^*-1)A_{\nu,j}-(g_\nu^*-1)A_{\mu,j}
 =[V_\mu,V_\nu].
\end{equation}
This identifies the coefficient on the oriented group two-cell
$(\mu,\nu)$, including its sign and normalization.
The first-order \v{C}ech transitions are unchanged, so the
uncorrected obstruction has only \v{C}ech degree zero.
Second-order changes to the local trivializations and transitions
add the differential of a total degree-one cochain.

Every bracket is again a $G$-invariant global field.  Its class in
\[
 \mathfrak g\otimes\Lambda^2G^\vee
 \ \subset\ H^2(X,\Theta_X)
\]
is the corresponding direct summand of the sector formula.  More
explicitly, projection onto a resonant exponent kills the group
differential, and a \v{C}ech differential cannot produce
\v{C}ech degree zero.  Thus every nonzero bracket in this summand represents a nonzero class in total degree two.  Choosing two noncommuting fields proves the last
assertion.
\end{proof}

For $m=1$, $\Lambda^2G^\vee=0$, so the displayed obstruction term vanishes.
The criterion above concerns the full-rank descent presentation with $m\geq2$.

\subsection{The primary obstruction of the square}
\label{sec:square-deformations}

Return to the configuration of Section~\ref{sec:tangent-square} and
write
\[
 (x,y,u,v)=(z_1,z_3,z_2,z_4),\qquad q_0=e^{-2\pi}.
\]
Then $U_K=(\C^2_{x,y}\setminus0)\times(\C^2_{u,v}\setminus0)$ and
\[
 g_1(x,y,u,v)=(q_0x,q_0y,u,q_0^2v),\qquad
 g_2(x,y,u,v)=(x,y,q_0u,q_0v).
\]
Every global field has a unique expression
\begin{equation}\label{eq:square-global-Lie-algebra}
 V=\left(M\binom{x}{y}\right)\!\cdot
       \binom{\partial_x}{\partial_y}
       +b\,u\partial_u+c\,v\partial_v
       +P(x,y)u\partial_v,
 \qquad P\in\C[x,y]_2,
\end{equation}
where $M\in\operatorname{Mat}_{2\times2}(\C)$ and $\C[x,y]_2$
is the vector space of homogeneous quadratic polynomials.
Define the linear functional
\[
 \ell(V)=-\operatorname{tr}M-b+c.
\]
Let $\eta$ be the additional \v{C}ech class represented by
$W=v(xy)^{-1}\partial_u$ in Section~\ref{sec:tangent-square}.
The cover there refines the two sets $y\ne0$ and $x\ne0$, in that
order.  We retain that ordering.  The sector formula gives
\[
 \begin{aligned}
 H^1(X,\Theta_X)&=(\mathfrak g\otimes G^\vee)\oplus\C\eta,\\
 H^2(X,\Theta_X)&=(\mathfrak g\otimes\Lambda^2G^\vee)
                         \oplus(\C\eta\otimes G^\vee).
 \end{aligned}
\]

\begin{theorem}\label{thm:square-primary-obstruction}
For $\alpha=V_1\otimes\xi_1+V_2\otimes\xi_2+s\eta$, the primary
obstruction, with \v{C}ech degree placed before group degree, is
\begin{equation}\label{eq:square-primary-obstruction}
 \quad
 \operatorname{ob}_2(\alpha)=
 [V_1,V_2]\otimes\xi_1\wedge\xi_2
 -s\,\eta\otimes\bigl(\ell(V_1)\xi_1+\ell(V_2)\xi_2\bigr).
 \quad
\end{equation}
Consequently $\alpha$ lifts to second order if and only if
\begin{equation}\label{eq:square-second-order-equations}
 [V_1,V_2]=0,\qquad s\ell(V_1)=s\ell(V_2)=0.
\end{equation}
This full-rank square example has obstructed infinitesimal deformations.  Its normalized diagonal family has
Kodaira--Spencer rank $8$ and cokernel dimension $11$.
\end{theorem}

\begin{proof}
We first compute the action of global fields on the \v{C}ech class:
\begin{equation}\label{eq:square-Cech-action}
 [V,\eta]=\ell(V)\eta.
\end{equation}
The four diagonal fields act on $W$ with eigenvalues $-1,-1,-1,1$
in the order $x\partial_x,y\partial_y,u\partial_u,v\partial_v$.
For the remaining generators, direct differentiation gives
\[
 \begin{aligned}
 [y\partial_x,W]&=-\frac{v}{x^2}\partial_u,&
 [x\partial_y,W]&=-\frac{v}{y^2}\partial_u,\\
 [P(x,y)u\partial_v,W]&=
       \frac{P(x,y)}{xy}(u\partial_u-v\partial_v).
 \end{aligned}
\]
The first expression extends to $x\ne0$ and the second to $y\ne0$.
For $P=x^2,xy,y^2$, the last coefficients are $x/y,1,y/x$,
respectively, and extend to one of these sets.  They are therefore
\v{C}ech coboundaries.  The extending vector fields are
$G$-invariant, so their primitives have zero group differential.
This proves \eqref{eq:square-Cech-action} in the total complex; the
associated group-degree-two term vanishes.

We now expand the descent equations with arbitrary second-order
corrections.  On the four Cox charts let $c_{ij}W$ denote the
\v{C}ech representative of $\eta$, so $c_{ij}$ is the difference
of the scalar values $0,0,1,1$ assigned in
Section~\ref{sec:tangent-square}.  Put
\[
 F_{\mu,j}=g_\mu\circ
       \exp(-\epsilon V_\mu-\epsilon^2A_{\mu,j}),\qquad
 T_{ij}=\exp(-\epsilon s c_{ij}W-\epsilon^2B_{ij}).
\]
Here $T_{ij}$ maps chart $i$ to chart $j$, and
$\Delta_\mu=g_\mu^*-1$.  The group relations give
\eqref{eq:deck-second-correction}.  Expanding
$F_{\mu,j}\circ T_{ij}=T_{ij}\circ F_{\mu,i}$ gives
\begin{equation}\label{eq:mixed-second-correction}
 A_{\mu,j}-A_{\mu,i}-\Delta_\mu B_{ij}
 =-s c_{ij}[V_\mu,W].
\end{equation}
Indeed, after removing the common factor $g_\mu$, the
second-order coefficient of the left point map minus the right
point map is
$-A_{\mu,j}+A_{\mu,i}+\Delta_\mu B_{ij}
-s c_{ij}[V_\mu,W]$.  This expansion gives the relative minus sign in the mixed obstruction term.

The triple-overlap relation is $\delta_{\check C}B=0$.
The quadratic term vanishes because all first-order transitions
are commuting constant multiples of $W$, and
$c_{ij}+c_{jk}=c_{ik}$.  Thus the three correction equations,
in bidegrees $(0,2)$, $(1,1)$, and $(2,0)$, are
\[
 \delta_G A=[V_1,V_2]\otimes\xi_1\wedge\xi_2,\qquad
 \delta_{\check C}A-\delta_G B
       =-s\sum_\mu[V_\mu,cW]\otimes\xi_\mu,\qquad
 \delta_{\check C}B=0.
\]
Their right-hand sides represent the primary obstruction in the
total convention of Section~\ref{sec:tangent-sheaf}.

Apply Proposition~\ref{prop:deck-primary-obstruction} and
\eqref{eq:square-Cech-action}.  This gives
\eqref{eq:square-primary-obstruction}.  The two summands of $H^2$
are independent, which proves \eqref{eq:square-second-order-equations}.
Sufficiency can also be seen directly.  The invariant primitives
constructed above give a zero-cochain $H_\mu$ with
\[
 \delta_{\check C}H_\mu
 =[V_\mu,cW]-\ell(V_\mu)cW,\qquad
 \Delta_\nu H_\mu=0\quad\text{for every }\nu.
\]
If \eqref{eq:square-second-order-equations} holds, take
$A_\mu=-sH_\mu$ and $B=0$.  All three correction equations then
hold.  The corrected transitions and deck maps therefore define
a deformation over $\C[\epsilon]/(\epsilon^3)$.

For example,
\[
 V_1=y\partial_x,\qquad V_2=x\partial_y,\qquad s=0
\]
give
\[
 \operatorname{ob}_2(\alpha)=
 (y\partial_y-x\partial_x)\otimes\xi_1\wedge\xi_2\ne0.
\]
This is a nonzero zero-exponent class in $H^2(\Theta)$.
There are also mixed obstructions: for
$\alpha=u\partial_u\otimes\xi_1+\eta$, the obstruction is
$\eta\otimes\xi_1\ne0$.
Finally, $Nm=8$ and $h^1(\Theta)=19$, so
Theorem~\ref{thm:normalized-KS} gives the rank and cokernel.
\end{proof}

Equations~\eqref{eq:square-second-order-equations} give the conditions for
lifting a deformation to second order. Higher obstruction equations, and hence
the full Kuranishi germ, remain to be determined. The diagonal family is not versal: its Kodaira--Spencer image has dimension eight, whereas $h^1(X,\Theta_X)=19$.

\subsection{A family integrating the extra \v{C}ech class}

The class $\eta$ is tangent to a proper holomorphic family. We construct the
family from extensions of line bundles over a scalar Hopf surface.

\begin{proposition}\label{prop:square-Cech-family}
There is a proper holomorphic submersion $\mathscr Y\to\Delta$ with
central fibre the resonant square and
$\operatorname{KS}_0(\partial_s)=\eta$.
\end{proposition}

\begin{proof}
Let $B^\circ=\C^2_{x,y}\setminus0$ and cover it by
$B_y=\{y\ne0\}$ and $B_x=\{x\ne0\}$.  For $s\in\Delta$, glue a
rank-two holomorphic vector bundle $E_s$ by the fibre transition
\begin{equation}\label{eq:square-extension-transition}
 \binom{u_1}{v_1}=
 T_s(x,y)\binom{u_0}{v_0},\qquad
 T_s(x,y)=
 \begin{pmatrix}1&-s/(xy)\\0&1\end{pmatrix}
\end{equation}
from $B_y$ to $B_x$.  The matrix is holomorphic on their
intersection and has determinant one, and hence defines a
holomorphic bundle over $B^\circ\times\Delta$.

The contraction $(x,y)\mapsto(q_0x,q_0y)$ lifts on each
trivialization by $D=\operatorname{diag}(1,q_0^2)$, since
\[
 T_s(q_0x,q_0y)D=D\,T_s(x,y).
\]
It follows that $E_s$ descends to a vector bundle $F_s$ on the
scalar Hopf surface
$H=B^\circ/\langle q_0\rangle$, holomorphically in $s$.
Fibrewise scalar multiplication by $q_0$ commutes with the
transitions.  Set
\[
 Y_s=(F_s\setminus0)/\langle q_0\,\mathrm{id}\rangle.
\]
Locally over $H\times\Delta$ this is the product with the compact
Hopf surface $(\C^2\setminus0)/\langle q_0\rangle$.
The local products glue by the bundle transitions, which commute
with the scalar action.  They give a complex manifold
$\mathscr Y$, a holomorphic fibre bundle over $H\times\Delta$,
and a proper holomorphic submersion $\mathscr Y\to\Delta$.

At $s=0$, the transition is the identity and the two quotient
generators are exactly $g_1,g_2$ above.  Thus $Y_0=X$.
Differentiating \eqref{eq:square-extension-transition}, the
horizontal lift from the second chart pulled back to the first
differs from its horizontal lift by
$v(xy)^{-1}\partial_u=W$.  The deck lifts are independent of $s$.
The Kodaira--Spencer cocycle is therefore precisely $\eta$.
\end{proof}

\section{Fixed curvature rank in visible dimension three}
\label{sec:three-dimensional-rigidity}

For a three-dimensional visible polytope, the basic ring has
Hilbert function $(1,m,m,1)$.  The middle multiplication maps now take values in the $m$-dimensional space
$R_\Lambda^2$.  We prove fixed-rank
rigidity for every normal fan of a triangular prism or a cube and
give holomorphic families with varying curvature image.

There is a geometric restriction on the subspaces which can occur as
curvature images.  Conjugation on the real presentation of the basic
ring will be denoted by a bar.

\begin{lemma}[Curvature and conjugation]
\label{lem:curvature-conjugation}
For every configuration under the standing hypotheses,
\[
 U_\Lambda+\overline{U_\Lambda}=R_\Lambda^1.
\]
Consequently,
$\dim(U_\Lambda\cap\overline{U_\Lambda})=m-2\delta$.
In particular, if $m=2$ and $\operatorname{rank}B=1$, the curvature
image is a nonreal line.
\end{lemma}

\begin{proof}
The real-part injectivity of Proposition~\ref{prop:real-part-injective}
gives $\mathfrak h_\Lambda\cap\overline{\mathfrak h_\Lambda}=0$.  Hence
\[
 \operatorname{Ann}(\mathfrak h_\Lambda)
 +\operatorname{Ann}(\overline{\mathfrak h_\Lambda})
 =E_\Lambda^\vee.
\]
Visible projection followed by the quotient to $R_\Lambda^1$ is
surjective.  Its image on the first annihilator is $U_\Lambda$ by
Theorem~\ref{thm:explicit-curvature}, and its image on the second is
$\overline{U_\Lambda}$.  The dimension formula follows from
$\dim U_\Lambda=m-\delta$.
\end{proof}

\subsection{The prism and the cube}

In visible dimension three, $m=N-3$.  The tetrahedron has $m=1$,
so Theorem~\ref{thm:real-indispensable-frame} permits only full rank.
For the triangular prism, $m=2$ and the only deficient rank is one. For the cube, $m=3$ and the only deficient rank is two.  These statements concern the simple
polytope $P$; their dual simplicial polytopes are the tetrahedron,
triangular bipyramid, and octahedron.

\begin{theorem}[Prism rigidity]
\label{thm:prism-fixed-rank}
Suppose the visible simple polytope is combinatorially
$\Delta^2\times\Delta^1$.  Every minimally stable realization of
curvature rank one has
\[
 \operatorname{Dol}_{X_\Lambda}(u,v)
 =(1+u)(1+v)^2(1+uv)(1+u^3v^2).
\]
At rank two its polynomial is
$(1+v)^2(1+u^2v)(1+u^3v^2)$.
Thus its Hodge diamond is determined by the curvature rank.
\end{theorem}

\begin{proof}
Write $R=R_\Lambda$ and let $U=\C\ell$ be the curvature line.
The dual complex is $\partial\Delta^2*\partial\Delta^1$, with one
minimal nonface of cardinality two and one of cardinality three.
Eliminating the three real linear relations presents $R$ as a
quotient of $\operatorname{Sym}(R^1)$, where $\dim R^1=2$.
Its Hilbert function is $(1,2,2,1)$.
The quadratic part of its ideal is therefore one-dimensional,
generated by a nonzero product $ab$ of two real linear forms:
$a$ and $b$ are the classes of the two disjoint facets.
They may be proportional.

If multiplication by $\ell$ from $R^1$ to $R^2$ had a nonzero
kernel, then $\ell c$ would be a nonzero scalar multiple of $ab$
in $\operatorname{Sym}^2(R^1)$ for some nonzero $c\in R^1$.
Unique factorization would make $\ell$ proportional to $a$ or $b$.
This would make $U$ real, contrary to
Lemma~\ref{lem:curvature-conjugation}.  Thus
$\ell:R^1\to R^2$ is an isomorphism.
Multiplication from $R^0$ to $R^1$ is injective, and the perfect
pairing $R^1\otimes R^2\to R^3$ makes
$\ell:R^2\to R^3$ surjective.

The one-generator reduced Koszul complex consequently has one
class in each of the bidegrees
\[
 (0,0),\quad(1,1),\quad(3,2),\quad(4,3).
\]
Its polynomial is $(1+uv)(1+u^3v^2)$.
The defect and closed antiholomorphic generators give the stated
exterior factors.  The full-rank formula follows from
Proposition~\ref{prop:product-simplices-example}.
\end{proof}

For the product normal fan, $R=\C[x,y]/(x^3,y^2)$.
For the line $\C y$, multiplication from $R^1$ to $R^2$ loses rank, giving
additional reduced classes in bidegrees $(2,1)$ and $(2,2)$.
Lemma~\ref{lem:curvature-conjugation} excludes this line from the curvature images of admissible LVM configurations.

\begin{lemma}[A quadratic Artinian quotient]
\label{lem:quadratic-hyperplane}
Let $R$ be a graded Artinian algebra generated by $R^1$, with an ideal
of relations generated in degree two.  For every hyperplane
$U\subset R^1$, one has $UR^1=R^2$.
\end{lemma}

\begin{proof}
The quotient $R/(U)$ is a quotient of $\C[z]$ whose ideal is
generated by homogeneous quadratics.  If its degree-two part were
nonzero, every quadratic relation would vanish identically after
restriction, and the quotient would be $\C[z]$.  This contradicts
the Artinian hypothesis.
\end{proof}

\begin{theorem}[Cube rigidity]
\label{thm:cube-fixed-rank}
Suppose the visible simple polytope is combinatorially a cube.
Every minimally stable realization of curvature rank two has
\[
 \operatorname{Dol}_{X_\Lambda}(u,v)
 =(1+u)(1+v)^3(1+uv)(1+u^2v)^2.
\]
At rank three its polynomial is $(1+v)^3(1+u^2v)^3$.
Thus its Hodge diamond is determined by the curvature rank.
\end{theorem}

\begin{proof}
Here $R=R_\Lambda$ has Hilbert function $(1,3,3,1)$.
The three minimal nonfaces are the three pairs of opposite facets.
After eliminating the linear relations, the ideal is generated by
three quadratics.  This holds for every normal fan with this
combinatorial type.

Let $U\subset R^1$ be the two-dimensional curvature image.
Lemma~\ref{lem:quadratic-hyperplane} gives $UR^1=R^2$.
The first Koszul map $\Lambda^2U\to R^1\otimes U$ is injective,
and the perfect pairing $R^1\otimes R^2\to R^3$ makes
$R^2\otimes U\to R^3$ surjective.
The Gorenstein pairing together with exterior multiplication on $U$
identifies the remaining maps with transposes of these maps, up to
sign.  Explicitly, the dimensions and ranks in the chains
$C^{p,a}=R^a\otimes\Lambda^{p-a}U$ are
\[
\begin{array}{c|c|c}
 p&(\dim C^{p,0},\dim C^{p,1},\dim C^{p,2},\dim C^{p,3})
   &(\operatorname{rank}d_0,\operatorname{rank}d_1,
      \operatorname{rank}d_2)\\ \hline
 0&(1,0,0,0)&(0,0,0)\\
 1&(2,3,0,0)&(2,0,0)\\
 2&(1,6,3,0)&(1,3,0)\\
 3&(0,3,6,1)&(0,3,1)\\
 4&(0,0,3,2)&(0,0,2)\\
 5&(0,0,0,1)&(0,0,0).
\end{array}
\]
The reduced polynomial is
\[
 1+uv+2u^2v+2u^3v^2+u^4v^2+u^5v^3
 =(1+uv)(1+u^2v)^2.
\]
Restoring the exterior factors gives the formula at deficient curvature rank.
The full-rank formula again follows from
Proposition~\ref{prop:product-simplices-example}.
\end{proof}

\subsection{Admissible families at constant rank}

Product configurations give the constant-rank families below and the Panov--Ustinovsky comparison over a six-dimensional visible polytope.

\begin{lemma}[Product configurations]
\label{lem:product-real-frame}
Choose integers $d_1,\ldots,d_m\geq1$.
In $\R^m\oplus\R^m$ put
\[
 g_0=(-\mathbf1,\mathbf1),\qquad
 g_\mu=(e_\mu,\mathbf1)\quad(1\leq\mu\leq m).
\]
For each $\mu$, take a group $G_\mu$ of $d_\mu+1$ visible labels,
all with weight $(0,-e_\mu)$.
If $F:\R^{2m}\to\C^m$ is an invertible real-linear map, the images
of these weights form an admissible minimally stable configuration
with
\[
 P=\Delta^{d_1}\times\cdots\times\Delta^{d_m},
 \qquad
 R_\Lambda=\C[x_1,\ldots,x_m]/
            (x_1^{d_1+1},\ldots,x_m^{d_m+1}).
\]
Here $x_\mu$ is the common class of the variables in $G_\mu$.
\end{lemma}

\begin{proof}
It suffices to check convex relations before applying $F$.
Write $c_0,\ldots,c_m$ for the ghost coefficients and $t_\mu$ for
the sum of the coefficients in $G_\mu$.  The first real block gives
$c_\mu=c_0$.  The second gives $t_\mu=(m+1)c_0$.
Normalization therefore forces
\[
 c_0=\cdots=c_m=\frac1{(m+1)^2},
 \qquad t_1=\cdots=t_m=\frac1{m+1}.
\]
Conversely these equalities give a convex relation for every
distribution of the group masses.  Every support contains all
$m+1$ ghosts and at least one label in each of the $m$ groups.
Its cardinality is at least $2m+1$, so weak hyperbolicity holds.
The supports with one label from each group are positive affine
simplices of this cardinality.  Every visible label can be omitted,
whereas every ghost is indispensable.  The relation polytope is
the product of the simplices of distributions within the groups,
which proves the assertion about $P$.

For the ring, write $\sigma_\mu=\sum_{j\in G_\mu}u_j$ for the group
sums of a visible coefficient vector and write $\eta$ for its ghost
coefficients in the affine chart based at $g_0$.
Annihilation of the real affine action matrix first gives
$(I+\mathbf1\mathbf1^t)\sigma=0$, and then
$(I+\mathbf1\mathbf1^t)\eta=0$.
Thus $\sigma=\eta=0$.  This calculation is unchanged by $F$,
and describes
$\operatorname{Ann}(\mathfrak h_\Lambda+
\overline{\mathfrak h_\Lambda})$.
The visible linear relations are exactly the differences within
each group.  The minimal nonfaces are the full groups, giving the
displayed ring.
\end{proof}

Let $\mathbb H=\{s\in\C:\operatorname{Im}s>0\}$.
For the prism take the three ghost weights
\[
 (-1-i,1+s),\qquad(1,1+s),\qquad(i,1+s),
\]
and visible weights $(0,-1)$ on a group of three labels and
$(0,-s)$ on a group of two labels.
They are the images in Lemma~\ref{lem:product-real-frame} for the
real frame
\[
 (e_1,e_2;f_1,f_2)\longmapsto
 ((1,0),(i,0);(0,1),(0,s)).
\]
The frame is invertible for $s\in\mathbb H$.
The indispensable block has rows $(2+i,0)$ and $(1+2i,0)$,
and the curvature exact sequence gives
\[
 R=\C[x,y]/(x^3,y^2),\qquad
 U_s=\C\bigl((1+2s)x-(2+s)y\bigr).
\]
Indeed, $\ker B=\C(0,1)$ and the two visible group values of
$A|_{\ker B}$ are $-(2+s)$ and $-(1+2s)$.
Writing $s=a+ib$, the determinant of the real and imaginary
coefficient vectors of the displayed generator is $3b$.
The curvature line is nonreal and varies with $s$.

For the cube take the four ghost weights
\[
\begin{split}
 &(-1-i,-1+i,1+s),\qquad(1,i,1+s),\\
 &(i,i,1+s),\qquad(0,1+i,1+s),
\end{split}
\]
and the three visible groups of size two with weights
\[
 (0,-i,0),\qquad(0,0,-1),\qquad(0,0,-s).
\]
The corresponding real frame is
\[
\begin{split}
 (e_1,e_2,e_3)&\longmapsto
 ((1,0,0),(i,0,0),(0,1,0)),\\
 (f_1,f_2,f_3)&\longmapsto
 ((0,i,0),(0,0,1),(0,0,s)).
\end{split}
\]
Its indispensable block and curvature image are
\[
 B=\begin{pmatrix}2+i&1&0\\1+2i&1&0\\1+i&2&0\end{pmatrix},
 \qquad \operatorname{rank}B=2,
\]
\[
\begin{split}
 R&=\C[x,y,z]/(x^2,y^2,z^2),\\
 U_s&=\{\alpha x+\beta y+\gamma z:
       (1+s)\alpha+(2+s)\beta+(1+2s)\gamma=0\}.
\end{split}
\]
The two real rows of the normal vector have a $2$ by $2$ minor
$-\operatorname{Im}s$, so $U_s+\overline{U_s}=R^1$.
These planes also vary with $s$.

\begin{proposition}[Constant-rank holomorphic families]
\label{prop:three-dimensional-constant-rank-families}
The prism and cube configurations above define proper holomorphic
submersions over $\mathbb H$.  Their fibres have complex dimensions
five and six, respectively.  Their curvature subspaces vary, their
curvature ranks are constant, and every Hodge number is constant.
\end{proposition}

\begin{proof}
The real frames are invertible throughout $\mathbb H$.
Lemma~\ref{lem:product-real-frame} therefore gives all convex-support
tests independently of $s$.  The weights are holomorphic in $s$.
Lemma~\ref{lem:parametric-LVM-quotient} applies on every disk in
$\mathbb H$; its minimizing-transversal proof applies over
$\mathbb H$ itself and gives joint properness and a proper
holomorphic submersion.  The displayed matrices give the ranks,
and Theorems~\ref{thm:prism-fixed-rank}
and~\ref{thm:cube-fixed-rank} give the Hodge numbers.
\end{proof}

\subsection{A higher-dimensional comparison}

Panov and Ustinovsky compare two products of Calabi--Eckmann
manifolds over
$\Delta^1\times\Delta^1\times\Delta^2\times\Delta^2$ and obtain
different values of $h^{2,1}$ \cite[Examples~5.12--5.13]{PanovUstinovsky2012}.
The same calculation has the following realization within the
minimally stable class.

\begin{proposition}[A fixed-rank comparison over a six-dimensional polytope]
\label{prop:known-six-dimensional-fixed-rank}
There are two admissible minimally stable configurations with the same
visible simple polytope
$\Delta^1\times\Delta^1\times\Delta^2\times\Delta^2$,
with $m=4$, $\operatorname{rank}B=2$, and complex dimension ten,
whose Hodge numbers $h^{2,1}$ are $9$ and $8$.
\end{proposition}

\begin{proof}
Use Lemma~\ref{lem:product-real-frame} with group sizes $(2,2,3,3)$.
Set
\[
 M=I_4+\mathbf1\mathbf1^t,\qquad
 M^{-1}=I_4-\frac15\mathbf1\mathbf1^t,
\]
and choose, in turn,
\[
 T_0=\begin{pmatrix}1&i&0&0\\0&0&1&i\end{pmatrix},
 \qquad
 T_1=\begin{pmatrix}1&0&i&0\\0&1&0&i\end{pmatrix}.
\]
Define a real-linear map $F_j:\R^4\oplus\R^4\to\C^4$ by
\[
 F_j(p,q)=
 \bigl(p_1+ip_2,\ p_3+ip_4,\ (T_jM^{-1}q)_1,\
                                  (T_jM^{-1}q)_2\bigr).
\]
Both maps are invertible over $\R$, since $T_j:\R^4\to\C^2$ is
invertible over $\R$.
The ghost differences span the first two complex coordinates,
so $B$ has rank two and its kernel consists of the last two action
coordinates.
On visible group sums the map $\beta_\Lambda$ in the curvature
exact sequence is $-T_jM^{-1}M=-T_j$.
Consequently the two curvature images are $\ker T_0$ and $\ker T_1$
in
\[
 R=\C[x_1,x_2,x_3,x_4]/(x_1^2,x_2^2,x_3^3,x_4^3).
\]
This proves admissibility, minimal stability, and the stated rank
with identical convex-support tests.

The one-generator Koszul complex over
$\C[x,y]/(x^{p+1},y^{q+1})$, $p\leq q$, with differential $x-y$,
has reduced polynomial
\[
 A_{p,q}(u,v)=
 \left(\sum_{a=0}^{p}(uv)^a\right)(1+u^{q+1}v^q).
\]
Its cokernel is $\C[x]/x^{p+1}$.
The kernel is generated as a vector space by
$x^a\sum_{k=0}^{q}x^{q-k}y^k$, $0\leq a\leq p$:
these elements are annihilated by $x-y$, are independent by their
distinct degrees, and have the required dimension $p+1$.
Every nonzero coefficient in the two chosen pairings can be
absorbed by rescaling the corresponding variable.
The reduced polynomials are therefore
\[
 A_{1,1}A_{2,2},\qquad A_{1,2}^2.
\]
The full polynomials have the common exterior factor
$(1+u)^2(1+v)^4$.
Their coefficients of $u^2v$ are respectively $9$ and $8$.
\end{proof}

Theorem~\ref{thm:prism-fixed-rank} includes all prism normal fans,
and Theorem~\ref{thm:cube-fixed-rank} includes all cube normal fans.
They leave open the corresponding question for other simple
three-polytopes.  The higher-dimensional comparison records the
Panov--Ustinovsky phenomenon in the present curvature notation.

\appendix
\section{Hypergeometric formulas for the polygonal coefficients}
\label{sec:lock-hypergeometric-appendix}

The cyclic-run coefficients define a sequence of reciprocal polynomials
$P_\nu$. We derive a terminating Gauss representation, an Euler integral, a
rational generating function, and a recurrence for this sequence.

Let $l\geq4$ and put
\[
 \nu=l-4.
\]
For $0\leq k\leq\nu$, define
\[
 c_{\nu,k}
 =
 a_{\nu+4,k+2},
 \qquad
 a_{l,p}
 =
 l\binom{l-2}{p-1}-\binom lp,
\]
and set
\[
 P_\nu(u)
 =
 \sum_{k=0}^{\nu}c_{\nu,k}u^k.
\]
Thus the polygonal Dolbeault polynomial is
\begin{equation}
\label{eq:polygon-Dol-Pnu}
 \operatorname{Dol}_{X_{\nu+4}}(u,v)
 =
 (1+v)^{\nu+2}
 \left[
 1+u^2vP_\nu(u)+u^{\nu+4}v^2
 \right].
\end{equation}

\subsection{The coefficient family}

\begin{proposition}[Closed coefficient formula]
\label{prop:polygon-hypergeom-coeff}
For $0\leq k\leq\nu$,
\[
 c_{\nu,k}
 =
 \frac{(\nu+1)(\nu+2)(\nu+4)}
 {(k+2)(\nu-k+2)}
 \binom{\nu}{k}.
\]
Equivalently, for $2\leq p\leq l-2$,
\[
 a_{l,p}
 =
 \frac{l(l-2)(l-3)}{p(l-p)}
 \binom{l-4}{p-2}.
\]
\end{proposition}

\begin{proof}
Starting from
\[
 a_{l,p}
 =
 l\binom{l-2}{p-1}-\binom lp,
\]
write both binomial coefficients over the common factorial denominator.
A direct cancellation gives
\[
 a_{l,p}
 =
 \frac{l(l-2)(l-3)}{p(l-p)}
 \frac{(l-4)!}{(p-2)!(l-p-2)!}.
\]
Substitute $l=\nu+4$ and $p=k+2$.
\end{proof}

\begin{corollary}[Positivity and reciprocity]
\label{cor:polygon-hypergeom-reciprocity}
Every coefficient of $P_\nu$ is a positive integer, and
\[
 c_{\nu,k}=c_{\nu,\nu-k}.
\]
Consequently,
\[
 P_\nu(u)=u^\nu P_\nu(u^{-1}).
\]
In particular, $P_\nu$ is a reciprocal polynomial of degree $\nu$.
\end{corollary}

\begin{proof}
Positivity and integrality follow from the cyclic-run interpretation in
Proposition~\ref{prop:polygon-run-count}.  Symmetry is immediate from the closed coefficient formula,
since
\[
 (k+2)(\nu-k+2)
\]
and $\binom{\nu}{k}$ are invariant under $k\leftrightarrow\nu-k$.
\end{proof}

\subsection{The terminating Gauss representation}

\begin{theorem}[Terminating hypergeometric form]
\label{thm:polygon-hypergeometric-form}
For every $\nu\geq0$,
\[
 P_\nu(u)
 =
 \frac{(\nu+1)(\nu+2)}2
 \left[
 {}_2F_1(-\nu,2;3;-u)
 +
 u^\nu{}_2F_1(-\nu,2;3;-u^{-1})
 \right].
\]
Both Gauss series terminate. Expanding the second summand gives a polynomial in $u$, including at $u=0$.
\end{theorem}

\begin{proof}
The upper parameter $-\nu$ makes the series terminate.  Using
\[
 (-\nu)_k=(-1)^k\frac{\nu!}{(\nu-k)!},
 \qquad
 \frac{(2)_k}{(3)_k}=\frac2{k+2},
\]
one obtains
\[
 {}_2F_1(-\nu,2;3;-u)
 =
 \sum_{k=0}^{\nu}
 \binom{\nu}{k}\frac2{k+2}u^k.
\]
After replacing $k$ by $\nu-k$ in the reciprocal summand,
\[
 u^\nu{}_2F_1(-\nu,2;3;-u^{-1})
 =
 \sum_{k=0}^{\nu}
 \binom{\nu}{k}\frac2{\nu-k+2}u^k.
\]
The coefficient of $u^k$ on the right-hand side of the asserted formula is
therefore
\[
 (\nu+1)(\nu+2)\binom{\nu}{k}
 \left(
 \frac1{k+2}+\frac1{\nu-k+2}
 \right),
\]
which equals
\[
 \frac{(\nu+1)(\nu+2)(\nu+4)}
 {(k+2)(\nu-k+2)}\binom{\nu}{k}.
\]
Apply Proposition~\ref{prop:polygon-hypergeom-coeff}.
\end{proof}

\subsection{Euler's integral and the elementary closed form}

\begin{corollary}[Beta-integral representation]
\label{cor:polygon-beta-integral}
For every $\nu\geq0$,
\[
 P_\nu(u)
 =
 (\nu+1)(\nu+2)
 \int_0^1
 t\left((1+ut)^\nu+(u+t)^\nu\right)\,dt.
\]
The identity is polynomial in $u$ and is valid over any field of
characteristic zero after interpreting the integral coefficientwise.
\end{corollary}

\begin{proof}
Euler's beta integral gives
\[
 {}_2F_1(-\nu,2;3;-u)
 =
 2\int_0^1t(1+ut)^\nu\,dt.
\]
The reciprocal terminating polynomial satisfies
\[
 u^\nu{}_2F_1(-\nu,2;3;-u^{-1})
 =
 2\int_0^1t(u+t)^\nu\,dt.
\]
Substitute these two identities into
Theorem~\ref{thm:polygon-hypergeometric-form}.
\end{proof}

\begin{proposition}[Elementary rational form]
\label{prop:polygon-rational-form}
The polynomial $P_\nu$ admits the closed expression
\[
 P_\nu(u)
 =
 \frac{
 1+u^{\nu+4}
 -(1+u)^{\nu+2}
 \left(u^2-(\nu+2)u+1\right)
 }{u^2}.
\]
The numerator is divisible by $u^2$, so the displayed quotient belongs to
$\mathbb Z[u]$.
\end{proposition}

\begin{proof}
Evaluating the integrals in Corollary~\ref{cor:polygon-beta-integral} gives,
for $u\neq0$,
\[
 \int_0^1t(1+ut)^\nu\,dt
 =
 \frac1{u^2}
 \left[
 \frac{(1+u)^{\nu+2}-1}{\nu+2}
 -
 \frac{(1+u)^{\nu+1}-1}{\nu+1}
 \right],
\]
while
\[
 \int_0^1t(u+t)^\nu\,dt
 =
 \frac{(1+u)^{\nu+2}-u^{\nu+2}}{\nu+2}
 -
 u\frac{(1+u)^{\nu+1}-u^{\nu+1}}{\nu+1}.
\]
Multiplication by $(\nu+1)(\nu+2)$ and simplification gives the stated
formula in $\mathbb C(u)$.  Its left-hand side is the polynomial
$P_\nu(u)$, so the numerator on the right is divisible by $u^2$ and the
identity extends to $u=0$.
\end{proof}

\begin{corollary}
\label{cor:polygon-hypergeom-special-values}
The first values and the specialization at $u=1$ are
\[
 \begin{aligned}
 P_0(u)&=2,\\
 P_1(u)&=5(1+u),\\
 P_2(u)&=9+16u+9u^2,\\
 P_3(u)&=14+35u+35u^2+14u^3,
 \end{aligned}
\]
and
\[
 P_\nu(1)=2+\nu 2^{\nu+2}.
\]
\end{corollary}

\begin{proof}
The first identities follow from the coefficient formula.  Set $u=1$ in
Proposition~\ref{prop:polygon-rational-form} for the last one.
\end{proof}

\subsection{The rational generating function}

\begin{theorem}[Generating function]
\label{thm:polygon-hypergeom-generating}
In $\mathbb Z[u][[q]]$,
\[
 \sum_{\nu\geq0}P_\nu(u)q^\nu
 =
 \frac{2-(1+u)q}
 {(1-q)(1-uq)(1-(1+u)q)^2}.
\]
\end{theorem}

\begin{proof}
Put $A=1+u$.  The closed form of
Proposition~\ref{prop:polygon-rational-form} can be rewritten in $\mathbb C(u)$ as
\[
 P_\nu(u)
 =
 \frac1{u^2}
 +u^2u^\nu
 -\frac{A^2(u^2-u+1)}{u^2}A^\nu
 +\frac{A^2}{u}(\nu+1)A^\nu.
\]
Summing the four elementary series gives
\[
 \sum_{\nu\geq0}P_\nu(u)q^\nu
 =
 \frac1{u^2(1-q)}
 +\frac{u^2}{1-uq}
 -\frac{A^2(u^2-u+1)}{u^2(1-Aq)}
 +\frac{A^2}{u(1-Aq)^2}.
\]
Combining the fractions yields the asserted rational function.  Since the coefficients on the left lie in $\mathbb Z[u]$, this is an identity in $\mathbb Z[u][[q]]$.
\end{proof}

\subsection{The holonomic recurrence}

Let $E$ denote the forward shift in the polygonal parameter:
\[
 (EP)_\nu=P_{\nu+1}.
\]

\begin{corollary}[Fourth-order recurrence]
\label{cor:polygon-hypergeom-recurrence}
The sequence $\{P_\nu(u)\}_{\nu\geq0}$ is annihilated by
\[
 (E-1)(E-u)(E-(1+u))^2.
\]
Equivalently, for every $\nu\geq0$,
\[
 \begin{aligned}
 P_{\nu+4}
 ={}&3(1+u)P_{\nu+3}
 -(3u^2+7u+3)P_{\nu+2}\\
 &+(u^3+5u^2+5u+1)P_{\nu+1}
 -u(1+u)^2P_\nu.
 \end{aligned}
\]
Together with the four initial polynomials in
Corollary~\ref{cor:polygon-hypergeom-special-values}, this determines the complete
polygonal coefficient family.
\end{corollary}

\begin{proof}
The denominator of the generating function is
\[
 (1-q)(1-uq)(1-(1+u)q)^2.
\]
Therefore its coefficient sequence satisfies the constant-coefficient
recurrence with characteristic polynomial
\[
 (z-1)(z-u)(z-(1+u))^2.
\]
Replacing $z$ by the shift $E$ gives the operator statement.  Expanding the
operator gives
\[
 E^4-3(1+u)E^3+(3u^2+7u+3)E^2
 -(u^3+5u^2+5u+1)E+u(1+u)^2,
\]
which is the displayed recurrence.
\end{proof}

\subsection{The cyclic-run interpretation}

The coefficient $c_{\nu,k}$ is the sum of the reduced zeroth Betti numbers of
the $(k+2)$-vertex induced subcomplexes of $C_{\nu+4}$:
\[
 c_{\nu,k}
 =
 \sum_{\substack{I\subseteq[\nu+4]\\|I|=k+2}}
 \widetilde b_0\!\left((C_{\nu+4})_I\right).
\]
Thus the coefficient of $u^k$ in $P_\nu$ counts connected components minus
one, summed over the $(k+2)$-vertex induced subgraphs; the full cycle supplies the separate
$u^{\nu+4}v^2$ term in \eqref{eq:polygon-Dol-Pnu}.

\section*{Acknowledgments}
E.L. gratefully acknowledges Cinvestav for a sabbatical leave during
which this work was prepared. He thanks the International Center for
Mathematical Sciences at the Institute of Mathematics and Informatics
of the Bulgarian Academy of Sciences for its hospitality.
His work was supported by the Simons Foundation, grant
SFI-MPS-T-Institutes-00007697, and by the Ministry of Education and Science
of the Republic of Bulgaria, grant DO1-239/10.12.2024.

ChatGPT and Claude were used to assist with checking calculations and
arguments and refining the exposition. ChatGPT also assisted with
bibliography and citation checks. The authors take responsibility
for the mathematical results and their presentation.

\begingroup
\small
\raggedright
\printbibliography[heading=bibintoc,title={References}]
\endgroup

\Needspace{20\baselineskip}
\section*{Author information}
\noindent
\textsc{Ludmil Katzarkov}. Simons Center for Geometry and Physics,
Stony Brook University, Stony Brook, NY 11794-3636, USA; and Institute
of Mathematics and Informatics, Bulgarian Academy of Sciences,
Sofia 1113, Bulgaria.  \texttt{l.katzarkov@miami.edu}

\medskip
\noindent
\textsc{Kyoung-Seog Lee}. Department of Mathematics, Pohang University of
Science and Technology (POSTECH), Pohang 37673, Republic of Korea.
\texttt{kyoungseog@postech.ac.kr}

\medskip
\noindent
\textsc{Ernesto Lupercio} (corresponding author). Departamento de
Matem\'aticas, Centro de Investigaci\'on y de Estudios Avanzados del Instituto
Polit\'ecnico Nacional (CINVESTAV--IPN), Mexico City 07360, Mexico; and
International Center for Mathematical Sciences, Institute of Mathematics
and Informatics, Bulgarian Academy of Sciences, Acad. G. Bonchev St.,
bl.~8, Sofia 1113, Bulgaria.
\texttt{lupercio@math.cinvestav.mx}

\medskip
\noindent
\textsc{Laurent Meersseman}. Univ Angers, CNRS, LAREMA, SFR MATHSTIC,
F-49000 Angers, France.  \texttt{laurent.meersseman@univ-angers.fr}

\end{document}